\documentclass[12pt]{amsart}

  \usepackage{amssymb}
  \usepackage{latexsym}
  \usepackage{amsfonts}
  \usepackage{amsthm}
  \usepackage{amsmath}
  \usepackage{hyperref}
  \usepackage{comment}
  \usepackage{graphicx}
  \usepackage{fullpage}
  \usepackage{url}
  \numberwithin{equation}{section}

  \theoremstyle{plain}
  \newtheorem{theorem}{Theorem}[section]
  \newtheorem{prop}[theorem]{Proposition}
  \newtheorem{lemma}[theorem]{Lemma}
  
  \newtheorem{cor}[theorem]{Corollary}
  
  \theoremstyle{definition}
  \newtheorem{defn}[theorem]{Definition}
  
  \theoremstyle{remark}
  
  \newtheorem{remark}[theorem]{Remark}

  \newcommand{\RR}{\mathbb{R}}
  \newcommand{\BB}{\mathbb{B}}
  \newcommand{\PD}{\partial}
  \newcommand{\lla}{\left\langle}
  \newcommand{\rra}{\right\rangle}
  
  \newcommand{\lf}{\left}
  \newcommand{\rt}{\right}

  \renewcommand{\Re}{\operatorname{Re}}
  \renewcommand{\Im}{\operatorname{Im}}
  \newcommand{\rg}{\operatorname{rg}}

\title{Blowup Stability at Optimal Regularity for the Critical Wave Equation}

\author{Roland Donninger}
\thanks{R.D.~is supported by the Austrian Science Fund FWF, Project P
  30076: ``Self-similar blowup in dispersive wave equations''.}
\address{Universit\"at Wien, Fakult\"at f\"ur Mathematik,
  Oskar-Morgenstern-Platz 1, 1090 Vienna, Austria}
\email{roland.donninger@univie.ac.at}

\author{Ziping Rao}
\address{Universit\"at Wien, Fakult\"at f\"ur Mathematik,
  Oskar-Morgenstern-Platz 1, 1090 Vienna, Austria}
\email{ziping.rao@univie.ac.at}
  
\begin{document}

\begin{abstract}
We establish Strichartz estimates for the radial energy-critical wave
equation in 5 dimensions in similarity coordinates. Using these, we
prove the nonlinear asymptotic stability of the ODE blowup in the energy space.
\end{abstract}

\maketitle


\section{Introduction}

\noindent We consider the 5 dimensional energy-critical wave equation
\begin{equation}\label{eqn:wave5}
\begin{cases}
& (\PD^2_t - \Delta_x) u(t,x) = |u(t,x)|^\frac{4}{3}u(t,x) \\
& u[0] = (f,g)
\end{cases}
\end{equation}
for $u:I\times \RR^5 \to \RR, I \subset \RR, 0 \in I$, and $u[t]:=(u(t,\cdot), \PD_t u(t,\cdot))$.

For studying well-posedness of this equation in low regularity, the suitable solution concept is given by \textit{Duhamel's formula}
\begin{equation}\label{eqn:duhamel}
u(t,\cdot) = \cos(t|\nabla|) f + \frac{\sin (t|\nabla|)}{|\nabla|} g + \int_0^t \frac{\sin ((t-s)|\nabla|)}{|\nabla|} \lf(|u(s,\cdot)|^\frac{4}{3}u(s,\cdot)\rt) ds,
\end{equation}
where $\cos(t|\nabla|)$ and $\frac{\sin(t|\nabla|)}{|\nabla|}$ are the
standard wave propagators.
We say that a function $u$ is a solution to \eqref{eqn:wave5} if it satisfies the integral equation \eqref{eqn:duhamel}.

It is known that this equation is locally well-posed in $\dot{H}^1
\times L^2(\RR^5)$ and this is the minimal regularity required, see
\cite{liso1995} and \cite{sogg1995}. The proof relies on the
celebrated \textit{Strichartz estimates}. These estimates capture the dispersive nature of the equation by describing a space-time improvement of the solution. This gives enough control of the nonlinearity.

One of the most important features of \eqref{eqn:wave5} is the
possibility of singularity formation (or \textit{blowup}) in finite
time from smooth initial data. This is evidenced by the explicit solution 
$$
u^T(t,x):= c_5(T-t)^{-\frac{3}{2}}, \; \; \; c_5 = \lf( \frac{15}{4} \rt)^{\frac{3}{4}},
$$
where $T>0$ is a free parameter. The solution $u^T$ is homogeneous in space but can be truncated by finite speed of propagation. This leads to an explicit example of singularity formation from smooth compactly supported data.

In order to see whether or not this blowup is generic, we need to study its stability. Ideally one would like to study the stability at optimal regularity, which requires certain Strichartz estimates. More precisely, as shown in \cite{donn2017}, this involves establishing Strichartz estimates for wave equations with self-similar potentials, since perturbing around this solution produces an equation of the form
$$
\lf(\PD^2_t - \Delta_x + (T-t)^{-2} V\lf(\frac{x}{T-t} \rt) \rt)
u(t,x) = \text{nonlinear terms}.
$$
Our goal here is to establish such estimates in 5 dimensions, which
will also give insight into other critical models such as the wave
maps equation. Compared to the 3 dimensional case treated in
\cite{donn2017}, the 5 dimensional case gives rise to a number of new substantial difficulties.

To be more precise, we study the stability of the blowup solution
$u^T$ locally in the backwards lightcone in the energy space. Since
$f\in \dot{H}^1(\RR^5)$ implies $f\in H^1_{loc}(\RR^5)$, it is natural
to consider perturbations in the space $H^1\times L^2$. We restrict
ourselves to radial data, in which case \eqref{eqn:wave5}
effectively reduces to
\begin{equation}\label{eqn:wave5rad}
\begin{cases}
& (\PD^2_t - \PD^2_r - \frac{4}{r}\PD_r) u(t,r) = |u(t,r)|^\frac{4}{3} u(t,r)\\
& u[0] = (f,g).
\end{cases}
\end{equation}
Our notion of solutions in the lightcone is given by the evolution semigroup in similarity coordinates, see Section \ref{subsec:simcoord}. Our main result is the following.

\begin{theorem}\label{thm:main} There exist $\delta, M >0$ such that the following holds. For all radial functions $f,g$ on $\RR^5$ with
$$
\| (f,g) - u^1[0] \|_{H^1\times L^2(\BB_{1+\delta}^5)} \leq \frac{\delta}{M},
$$
there exists a $T\in [1-\delta,1+\delta]$ such that the solution $u$ to $\eqref{eqn:wave5}$ exists in the backwards lightcone $\Gamma_T := \{ (t,x) \in [0,T)\times \RR^5: |x| \leq T-t\}$ and satisfies
$$
\int_0^T \frac{\|u (t,\cdot)- u^T(t,\cdot) \|^2_{L^5(\BB_{T-t}^5)}}{\| u^T(t,\cdot) \|^2_{L^5(\BB_{T-t}^5)}} \frac{dt}{T-t} \lesssim \delta^2.
$$
In particular, $u$ blows up at the tip of the cone $\Gamma_T$.
\end{theorem}

\subsection{Remarks}
\begin{itemize}
\item Naively, one might think that stability of the ODE blowup
  solution $u^1$ means that small initial perturbations of $u^1$ lead
  to a solution that converges back to $u^1$. This is, however, wrong
  in general as the perturbation might change the blowup time.
  Consequently, 
   Theorem \ref{thm:main} says that if we
  slightly perturb the initial data for the ODE blowup solution $u^1$,
  the corresponding solution
  $u$ approaches $u^T$ as $t\to T-$ for some $T$ close to $1$.
  In particular, the blowup time depends continuously on the
  perturbation.

  \item Theorem \ref{thm:main} shows that there exists an
    \textit{open} set in the energy space of initial data that lead to
    the ODE blowup. In terms of regularity, this result is optimal.

    \item The convergence of the solution $u$ to $u^T$ is normalized to the blowup behaviour of
      $u^T$ and measured in a Strichartz norm. Observe that $\int_0^T\frac{dt}{T-t}=\infty$ and thus,
      the integrand
     \[ \frac{\|u (t,\cdot)- u^T(t,\cdot) \|^2_{L^5(\BB_{T-t}^5)}}{\|
         u^T(t,\cdot) \|^2_{L^5(\BB_{T-t}^5)}} \]
     must go to zero in an averaged sense as $t\to T-$.
\end{itemize}

\subsection{Related results} Singularity formation in critical wave
equations has attracted a lot of interest in recent years. There are
different types of blowup. In type I blowup, the energy norm of the solution becomes unbounded in finite time (e.g.~the aforementioned ODE blowup). On the other hand, type II blowup solutions have bounded energy norm. Blowup solutions of type II were constructed in \cite{kst2009}, \cite{krsc2014}, \cite{hira2012}, and \cite{jend2017}. Substantial progress has been made in the study of type II blowup. In \cite{dkm2011u, dkm2012u, dkm2012, dkm2016}, type II blowup profiles were classified completely. It is worth mentioning that all type II blowup solutions are unstable. However, in recent breakthrough work, a family of type II blowup solutions that is ``as stable as it can possibly be" has been identified, see \cite{krie2017} and \cite{bukr2017}.

Less is known for the type I case. For subcritical equations, there is
the remarkable series of papers \cite{meza2003, meza2005, meza2007,
  meza2008, meza2012a, meza2012b, meza2015}. Unfortunately, the methods
developed there do not extend to the critical case. A numerical study
in \cite{btt2004} suggests that type I blowup is generic. The
stability of the ODE blowup in the lightcone was established in a
stronger topology in \cite{dosc2014, dosc2016,
  dosc2017}. Subsequently, \cite{donn2017} was able to prove this
stability in the energy norm by establishing suitable Strichartz estimates in 3
dimensions.

Finally, our results may also be interesting from the point of view
of Strichartz estimates for wave equations with variable
coefficients. This is a very active area of current research, see
e.g.~\cite{dafa2008, meta2012}.

\subsection{Outline} We follow the approach from \cite{donn2017}, but
we will point out some major differences in this work. As is well
known, the 3 dimensional radial wave equation is equivalent to the 1
dimensional wave equation with a Dirichlet condition at the
origin. This special structure was heavily 
exploited in \cite{donn2017}. As a consequence, 
the present paper is far from being a straightforward adaptation of the
methods from \cite{donn2017}.

Following the machinery from \cite{dosc2012, donn2011, dsa2012}, we first introduce similarity coordinates in the lightcone $\Gamma_T$ by
\begin{equation}\label{eqn:simcoor}
\tau :=  \log \frac{T}{T-t}, \qquad \rho := \frac{r}{T-t}.
\end{equation}
Notice that this coordinate transformation depends on the parameter
$T$, but we drop this dependence in the notation for simplicity. Under
the transformation $(t,r)\mapsto (\tau,\rho)$, $\Gamma_T$ is mapped to the infinite cylinder $\RR_+ \times \overline{\BB^5}$. We rewrite \eqref{eqn:wave5rad} as an evolution problem of the form
$$
\PD_\tau \Psi(\tau) = \tilde{\mathbf{L}}_0 \Psi(\tau) + \mathbf{F} (\Psi(\tau)),
$$
where $\tilde{\mathbf{L}}_0$ is a spatial differential operator and
$\mathbf F$ is the nonlinearity.
The ODE blowup $u^T$ in this formulation is given by the constant solution $(c_5, \frac{2}{3} c_5)$. Perturbing around this solution by inserting the ansatz $\Psi(\tau) = (c_5, \frac{2}{3} c_5) + \Phi(\tau)$ yields
$$
\PD_\tau \Phi(\tau) = (\tilde{\mathbf{L}}_0 + \mathbf{L}') \Phi(\tau) + \mathbf{N} (\Phi(\tau)).
$$
Here we obtain a potential term $\mathbf{L}'$ from the linearisation
of $\mathbf{F}$, and $\mathbf{N}$ is the remaining nonlinearity. We
show that the closure of $\tilde{\mathbf{L}}_0$, denoted by
$\mathbf{L}_0$, generates a semigroup $\mathbf{S}_0(\tau)$ on
$\mathcal{H} := H^1 \times L^2 (\BB^5)$. For the first component of this semigroup, we have the Strichartz estimates
$$
\lf\Vert \lf[\mathbf{S}_0(\cdot) \mathbf{f} \rt]_1 \rt\Vert_{L^p(\RR_+;L^q(\BB^5))} \lesssim \|\mathbf{f} \|_{\mathcal{H}}, \; \; \frac{1}{p}+\frac{5}{q} = \frac{3}{2}
$$
for $p\in [2, \infty]$ and $q\in [\frac{10}{3},5]$. This follows from the standard Strichartz estimates for the wave operator via a scaling argument.

From the bounded perturbation theorem, we also have that $\mathbf{L}_0
+ \mathbf{L}'$ generates a semigroup $\mathbf{S}(\tau)$ on
$\mathcal{H}$. The operator $\mathbf{L}_0 + \mathbf{L}'$ has one
unstable eigenvalue $\lambda=1$ coming from the time translation
symmetry of the equation, and we use the Riesz projection $\mathbf{P}$
to remove this unstable subspace. An application of the
Gearhart-Pr\"uss theorem yields the bound $\| \mathbf{S}(\tau)
(\mathbf{I} - \mathbf{P}) \|_{\mathcal{B}(\mathcal{H})} \leq
C_\epsilon e^{\epsilon\tau}$ for any (fixed) $\epsilon>0$. In order to obtain Strichartz estimates for the semigroup $\mathbf{S}(\tau)$, we find an explicit representation from Laplace inversion. For sufficiently regular $\mathbf{f}$, let $\tilde{\mathbf{f}} =(\tilde{f}_1,\tilde{f}_2) := (\mathbf{I} - \mathbf{P}) \mathbf{f}$. Then we have
$$
[\mathbf{S}(\tau) \tilde{\mathbf{f}}]_1(\rho) = \frac{1}{2\pi i}\lim_{N\to\infty} \int_{\epsilon-i N}^{\epsilon + i N} e^{\lambda \tau} \int_0^1 G(\rho,s;\lambda) F_\lambda(s) ds d\lambda,
$$
where $G$ is the Green function of the spectral ODE associated to $\mathbf{L}$ and $F_\lambda(s) := (\lambda + \frac{5}{2}) \tilde{f}_1(s)+ s\tilde{f}_1'(s) + \tilde{f}_2(s)$.

The Green function $G$ is constructed by considering a perturbation of
the free equation by a potential, and this is done by Volterra
iterations. Then, we obtain a decomposition $G(\rho, s;\lambda) =
G_0(\rho, s;\lambda) + \tilde{G}(\rho, s;\lambda)$, where $G_0$ is the
Green function of the free equation. Hence the representation of
$\mathbf{S}(\tau)$ also decomposes into a free part
$\mathbf{S}_0(\tau)$ and a perturbation $\mathbf{T}(\tau)$
accordingly. The perturbative part has an extra decay as $|\Im
\lambda| \to \infty$, which is ultimately what we exploit in
establishing the Strichartz estimates. Due to the singular behaviour
of the fundamental system, the construction of $G$ here is
considerably more difficult than in the 3 dimensional case. 


Next, we turn to proving Strichartz estimates for $\mathbf{T}(\tau)$ by carefully examining the oscillatory kernels. An extra difficulty comes in, compared to the 3 dimensional case, from the more complicated form of the free fundamental solutions, one of which is given by
$$
\varphi_1 (\rho; \lambda) = \rho^{-3}(1+\rho)^{\frac{1}{2}-\lambda}(2+\rho(-1+2\lambda)).
$$
One part of this solution has growth of order $O(\lla \lambda \rra)$, which destroys the nice extra decay of $\tilde{G}$. However, this part is also less singular in $\rho$ near $0$ ($\rho^{-2}$ instead of $\rho^{-3}$). Balancing the undesired growth in $\lambda$ and the singularity in $\rho$ is a delicate procedure, which requires additional work. With this we prove Strichartz estimates for the full semigroup
$$
\lf\Vert \lf[\mathbf{S}(\cdot) \mathbf{f} \rt]_1 \rt\Vert_{L^p(\RR_+;L^q(\BB^5))} \lesssim \|\mathbf{f} \|_{\mathcal{H}}, \; \; \frac{1}{p}+\frac{5}{q} = \frac{3}{2}
$$
for $p\in [2, \infty]$ and $q\in [\frac{10}{3},5]$. With the same
techniques we also improve the growth bound from the Gearhart-Pr\"uss theorem to $\| \mathbf{S}(\tau)(\mathbf{I} - \mathbf{P}) \|_{\mathcal{B}(\mathcal{H})} \lesssim 1$.

Using Strichartz estimates and the improved growth bound, we are able to control the nonlinearity as in \cite{donn2017} and finish the proof of Theorem \ref{thm:main}.

\subsection{Notation}

We write $\mathbb{B}^5_R:= \{x\in \RR^5: |x| < R\}$ and $\BB^5 := \BB^5_1$. A function $f\in C^1([0,1])$ with $f'(0)=0$ gives rise to a radial function on $\overline{\BB^5}$. Hence we define
$$
\| f \| ^p _{L^p(\BB^5_R)} := \int_0^R |f(r)|^p r^4 dr,
$$
and
$$
\|f \|^2 _{H^1(\BB^5_R)} := \int_0^R |f'(r)|^2 r^4 dr + \int_0^R |f(r)|^2 r^4 dr.
$$

Bold letters denote 2-component functions, for example, $\mathbf{f} = (f_1, f_2)$. We write $f(x) = O(g(x))$ if $|f(x)| \lesssim |g(x)|$, and $f \sim g$ as $x\to a$ if $\lim_{x\to a} \frac{f(x)}{g(x)}=1$. We also write $f \simeq g$ if $f\lesssim g \lesssim f$. Lastly, we use the Japanese bracket notation $\lla x \rra := \sqrt{1+|x|^2}$.

\section{Similarity coordinates and semigroup theory}

\subsection{Similarity coordinates}\label{subsec:simcoord}

As mentioned before, similarity coordinates are defined by \eqref{eqn:simcoor} on $\Gamma_T$.
For a solution $u \in C^\infty(\Gamma_T)$ to \eqref{eqn:wave5rad}, we define
$$ 
\psi(\tau ,\rho) = (T e^{-\tau})^\frac{3}{2} u(T-T e^{-\tau}, Te^{-\tau}\rho) ,
$$
and we have $\psi \in C^\infty(\RR_+ \times \overline{\BB^5})$. Then, by setting
\begin{equation} \label{eqn:coordtrans}
\begin{aligned}
\psi_1(\tau ,\rho) &:= \psi(\tau ,\rho) ,\\ 
\psi_2(\tau ,\rho) &:= \lf( \PD_\tau + \rho \PD_\rho + \frac{3}{2} \rt) \psi(\tau,\rho) ,
\end{aligned}
\end{equation}
we obtain the first order system from \eqref{eqn:wave5rad}, 
\begin{equation}\label{eqn:wavesim}
\begin{cases}
& \PD_\tau \psi_1 = -\rho \PD_\rho\psi_1 - \frac{3}{2}\psi_1 + \psi_2 \\
& \PD_\tau \psi_2 = \PD_\rho^2 \psi_1 + \frac{4}{\rho} \PD_\rho\psi_1 -\rho \PD_\rho \psi_2 - \frac{5}{2} \psi_2+ |\psi_1|^{\frac{4}{3}} \psi_1 \\
& \psi_1(0,\rho) = T^{\frac{3}{2}} f(T\rho), \; \; \psi_2(0,\rho) = T^{\frac{5}{2}} g( T\rho ).
\end{cases}
\end{equation}
The ODE blowup solution $u^T$ corresponds to the constant solution $(\psi^T_1,\psi^T_2) = (c_5, \frac{3}{2} c_5)$. We take a perturbation around this solution by the ansatz
\begin{equation}\label{eqn:perturbation}
(\psi_1, \psi_2) = \lf(c_5, \frac{3}{2} c_5 \rt) + (\varphi_1, \varphi_2).
\end{equation}
This gives the system
\begin{equation}\label{eqn:NLWpert}
\begin{cases}
& \PD_\tau  \varphi_1  = -\rho \PD_\rho\varphi_1 - \frac{3}{2}\varphi_1 + \varphi_2  \\
& \PD_\tau \varphi_2  =  \PD_\rho^2 \varphi_1 + \frac{4}{\rho} \PD_\rho\varphi_1 -\rho \PD_\rho \varphi_2 - \frac{5}{2} \varphi_2  + F(\varphi_1)  \\
& \varphi_1(0,\rho) = \psi_1(0,\rho) - c_5,\; \;  \varphi_2(0,\rho) = \psi_2(0,\rho) - \frac{3}{2} c_5
\end{cases}
\end{equation}
where $F(x) = - c_5^{\frac{7}{3}} + |c_5 + x |^{\frac{4}{3}}(c_5 + x )$. We take the Taylor expansion of $F$ around zero: $F(x) = F(0) + F'(0) x + N(x)$. We see that $F(0)=0$, and
\begin{equation*}
F'(0) = \frac{7}{3} |c_5 + x|^\frac{4}{3}\Big |_{x=0} = \frac{7}{3} \cdot \frac{15}{4} = \frac{35}{4}.
\end{equation*}
This linearisation produces a potential term. The nonlinearity is then given by
$$
N(x)  = F(x) - F(0) - F'(0) x  = \lf| c_5  + x \rt|^{\frac{4}{3}}  \lf( c_5 + x \rt)-   c_5^{\frac{7}{3}}  - \frac{35}{4} x.
$$

Hence we define the formal differential operators
\begin{equation*}
\tilde{\mathbf{L}}_0 \mathbf{u} (\rho) := \begin{pmatrix}
-\rho u_1'(\rho) - \frac{3}{2}u_1(\rho) + u_2(\rho) \\
u''_1(\rho) + \frac{4}{\rho} u'_1(\rho) -\rho u'_2(\rho) - \frac{5}{2} u_2(\rho)
\end{pmatrix},
\end{equation*}
\begin{equation*}
\mathbf{L}' \mathbf{u} (\rho) := \begin{pmatrix}
0 \\
\frac{35}{4} u_1(\rho)
\end{pmatrix},
\end{equation*}
and
\begin{equation*}
\mathbf{N} ( \mathbf{u} ) (\rho) := \begin{pmatrix}
0 \\
N( u_1(\rho) )
\end{pmatrix}.
\end{equation*}
Let $\Phi(\tau)(\rho) = (\varphi_1(\tau,\rho),\varphi_2(\tau,\rho))$. We can then write the system \eqref{eqn:NLWpert} as
\begin{equation}
\begin{cases}
& \PD_\tau \Phi(\tau) = ( \tilde{\mathbf{L}}_0 + \mathbf{L}') \Phi(\tau) + \mathbf{N} (\Phi(\tau)) \\
& \Phi(0)(\rho) = (\varphi_1(0,\rho), \varphi_2(0,\rho)).
\end{cases}
\end{equation}

We study this evolution on the Banach space
$$
\mathcal{H}:= \{\mathbf{f} \in H^1 \times L^2(\mathbb{B}^5): \mathbf{f} \text{ radial} \},
$$
with the norm
$$
\|\mathbf{f}\|^2_\mathcal{H}:= \|f_1\|^2_{H^1 (\mathbb{B}^5)}+\|f_2\|^2_{L^2(\mathbb{B}^5)}
$$
for $\mathbf{f}=(f_1,f_2)$. In next parts of this section, we will see that the closure of $\tilde{\mathbf{L}}_0 + \mathbf{L}'$ with a suitable domain generates a semigroup $\mathbf{S}(\tau)$ on $\mathcal{H}$, which leads to the concept of energy-class strong lightcone solutions.

\begin{defn}\label{defn:sol} We say that $u:\Gamma_T \to \mathbb{C}$ is an \textit{energy-class solution to \eqref{eqn:wave5rad} in the lightcone} if the corresponding $\Phi: [0,\infty) \to \mathcal{H}$ is in $C([0,\infty);\mathcal{H})$ and satisfies
$$
\Phi(\tau) = \mathbf{S}(\tau)\mathbf{u} + \int_0^\tau \mathbf{S} (\tau - \sigma) \mathbf{N}(\Phi(\sigma)) d\sigma
$$
for all $\tau>0$.
\end{defn}



\subsection{Semigroup of the free evolution}

For the linear differential operator $\tilde{\mathbf{L}}_0$, we define its domain to be
$$
\mathcal{D}(\tilde{\mathbf{L}}_0) := \{ \mathbf{f}=(f_1,f_2)\in C^2\times C^1([0,1]): f_1'(0)=0 \}.
$$
So now $\tilde{\mathbf{L}}_0$ becomes a densely defined (unbounded) linear operator on $\mathcal{H}$. We would like to show that the closure of $\tilde{\mathbf{L}}_0$ generates a semigroup. We start with introducing an equivalent inner product on the dense subspace $\mathcal{D}(\tilde{\mathbf{L}}_0)$ given by 
$$
(\mathbf{f} | \mathbf{g} )_E: = \int_0^1 f_1'(r) \overline{g_1'(r)}r^4 dr  + \int_0^1 f_2(r) \overline{g_2(r)} r^4 dr +  f_1(1) \overline{g_1(1)} ,
$$
and we denote by $\| \cdot \|_E$ the norm defined by $(\cdot|\cdot)_E$.

\begin{lemma}\label{lemma:equivnorm} The norms $\|\cdot\|_E$ and $\|\cdot\|_\mathcal{H}$ are equivalent norms on $\mathcal{D}(\tilde{\mathbf{L}}_0)$. In other words, we have
$$
\|\mathbf{f}\|_E \simeq \|\mathbf{f}\|_{\mathcal{H}}
$$
for all $\mathbf{f} \in \mathcal{D}(\tilde{\mathbf{L}}_0)$.

In particular, taking the completions of $\mathcal{D}(\tilde{\mathbf{L}}_0)$ with respect to both the norms $\| \cdot \|_E$ and $\|\cdot \|_\mathcal{H}$, we conclude that they are also equivalent norms on $\mathcal{H}$.
\end{lemma}

\begin{proof} Here we need to show
\begin{equation}\label{eqn:equivh1}
\int_0^1 |f'(r) |^2 r^4 dr + |f(1)|^2 \lesssim \|f\|^2_{H^1(\BB^5)} \lesssim \int_0^1 |f'(r)|^2 r^4 dr + |f(1)|^2
\end{equation}
for all $f\in C^2([0,1])$. Since $\|f\|_E^2 = \|\Re f\|_E^2 + \| \Im f\|_E^2 $, we only need to consider the case when $f$ is real-valued.

By the Sobolev embedding $H^1(\frac{1}{2},1) \hookrightarrow L^\infty(\frac{1}{2},1)$, we have
$$
|f(1)|^2  \lesssim \|f\|_{H^1(\frac{1}{2},1)}^2 \lesssim \|f\|^2_{H^1(\BB^5)},
$$
which gives the first inequality of \eqref{eqn:equivh1}.

Next, we need to control the $L^2$ part of the $H^1$ norm of $f$. Since $r\in[0,1]$, we can write
\begin{align*}
|f(r)|^2 r^4  = & \lf| \int_0^r \frac{d}{ds} \lf( f(s)s^2 \rt) ds \rt|^2 \\
\leq & \int_0^1 \lf| f'(s)s^2 + 2 f(s) s \rt|^2 ds \\
= & \int_0^1 |f'(s)|^2 s^4 ds +  4 \int_0^1 |f(s)|^2 s^2ds + 2 \int_0^1 \lf( \frac{d}{ds} |f(s)|^2 \rt) s^3 ds  \\
= & \int_0^1 |f'(s)|^2 s^4 ds + 2 \lf( |f(s)|^2s^3 \rt) \Big |_{s=0}^{s=1} - 6 \int_0^1  |f(s)|^2 s^2 ds\\
\lesssim & \int_0^1 |f'(s)|^2 s^4 ds +  |f(1)|^2 .
\end{align*}
The right hand side is independent of $r$, so we can integrate over $r$ and obtain
$$
\int_0^1 |f(r)|^2 r^4 dr \lesssim \int_0^1 |f'(s)|^2 s^4 ds +  |f(1)|^2,
$$
and we also have the second inequality of \eqref{eqn:equivh1}.
\end{proof}

With this set up, we now turn to the operator $\tilde{\mathbf{L}}_0$.

\begin{prop}\label{prop:semigp}
The operator $\tilde{\mathbf{L}}_0:\mathcal{D}(\tilde{\mathbf{L}}_0)\subset \mathcal{H} \to \mathcal{H}$ is closable and its closure $\mathbf{L}_0$ generates a strongly-continuous one-parameter semigroup $\{ \mathbf{S}_0(\tau):\tau\geq 0\}$ on $\mathcal{H}$, with the estimate
$$
\| \mathbf{S}_0(\tau) \mathbf{f} \|_{\mathcal{H}} \lesssim \| \mathbf{f} \|_{\mathcal{H}},
$$
for all $\tau \geq 0$ and $\mathbf{f} \in \mathcal{H}$.
\end{prop}

\begin{proof} We use the Lumer-Philips theorem (\cite[p.83 Theorem 3.15 and p.88 Proposition 3.23]{enna2000}) with the inner product $(\cdot | \cdot)_E$ and norm $\|\cdot\|_E$. Here we have to show two things:
\begin{itemize}
\item $\Re(\tilde{\mathbf{L}}_0 \mathbf{u} | \mathbf{u})_E \leq 0 $ for all $\mathbf{u} \in \mathcal{D}(\tilde{\mathbf{L}}_0)$;
\item range of $(\lambda - \tilde{\mathbf{L}}_0)$ is dense for some $\lambda > 0$.
\end{itemize}

For the first part, let $\mathbf{u} = (u_1,u_2) \in \mathcal{D}(\tilde{\mathbf{L}}_0)$. We compute
\begin{align*}
\Re(\tilde{\mathbf{L}}_0 \mathbf{u} | \mathbf{u})_E
= & \Re\lf( \int_0^1 \lf(-\frac{5}{2} |u_1'(\rho)|^2 -\rho u_1''(\rho)\overline{u_1'(\rho)}  +u_2'(\rho)\overline{u_1'(\rho)} \rt) \rho^4 d\rho \rt)\\
& + \Re\lf( \int_0^1 \lf(u_1''(\rho)\overline{u_2(\rho)} + \frac{4}{\rho} u_1'(\rho) \overline{u_2(\rho)} - \rho u_2'(\rho) \overline{u_2(\rho)} - \frac{5}{2} |u_2(\rho)|^2 \rt)\rho^4 d\rho \rt) \\
&  - \Re\lf(u_1'(1) \overline{u_1(1)}-\frac{3}{2}|u_1(1)|^2  +u_2(1) \overline{u_1(1)} \rt) .
\end{align*}
We first collect the terms containing $u_1$ only and use integration by parts,
\begin{align*}
I_1 
& =  -\frac{5}{2} \int_0^1  |u_1'(\rho)|^2 \rho^4 d\rho - \frac{1}{2} \int_0^1 \frac{d}{d\rho} \lf(|u_1'(\rho)|^2 \rt) \rho^5 d\rho - \Re\lf(u_1'(1) \overline{u_1(1)} \rt) -\frac{3}{2} |u_1(1)|^2  \\
& = - \frac{1}{2}|u_1'(1)|^2 -  \Re\lf(u_1(1) \overline{u_1'(1)} \rt) - \frac{3}{2}|u_1(1)|^2 .
\end{align*}
Similarly we collect the terms containing $u_2$ only,
$$
I_2 = -\frac{1}{2} \int_0^1 \frac{d}{d\rho} \lf( |u_2(\rho)|^2 \rt)\rho^5 d\rho - \frac{5}{2} \int_0^1 |u_2(\rho)|^2 \rho^4 d\rho 
 = - \frac{1}{2} |u_2(1)|^2.
$$
And the rest of the terms are
\begin{align*}
 I_3  
= &  \Re\lf( \int_0^1 \lf( u_2'(\rho)\overline{u_1'(\rho)} \rho^4 + 4 u_1'(\rho) \overline{u_2(\rho)} \rho^3 \rt) d\rho  + \int_0^1 \frac{d}{d\rho}\lf( u_1'(\rho) \rt) \overline{u_2(\rho)}\rho^4 d\rho + u_2(1) \overline{u_1(1)} \rt) \\
= & \Re \lf( u_1'(1) \overline{u_2(1)} \rt) + \Re\lf(u_2(1) \overline{u_1(1)} \rt).
\end{align*}
Summing up all terms, we get
\begin{align*}
& \Re (\tilde{\mathbf{L}}_0 \mathbf{u} | \mathbf{u})_E \\
= &  I_1 + I_2 + I_3 \\
= &- \frac{1}{2}|u_1'(1)|^2 - \frac{3}{2}|u_1(1)|^2 - \frac{1}{2} |u_2(1)|^2  + \Re \lf( u_1'(1) \overline{u_2(1)} \rt) + \Re\lf(u(1) \overline{(u_2(1)-u_1'(1))} \rt) \\
\leq & - \frac{1}{2}|u_1'(1)|^2 - \frac{3}{2}|u_1(1)|^2 - \frac{1}{2} |u_2(1)|^2  + \Re \lf( u_1'(1) \overline{u_2(1)} \rt) + \frac{1}{2} |u_1(1)|^2 + \frac{1}{2} |u_2(1)-u_1'(1)|^2  \\
= & - |u_1(1)|^2  \leq 0.
\end{align*}

Next, in order to show that the range of $(\lambda -\tilde{\mathbf{L}}_0)$ is dense, we show that the system of ODEs
$$
(\lambda-\tilde{\mathbf{L}}_0) \mathbf{u} = \mathbf{f}
$$
admits a solution $\mathbf{u}\in \mathcal{D}(\tilde{\mathbf{L}}_0)$ for every $\mathbf{f} \in C^\infty\times C^\infty([0,1])$, which is a dense subspace of $\mathcal{H}$. Writing $\mathbf{f}=(f_1,f_2)  \in C^\infty\times C^\infty([0,1])$, we need to solve the system
\begin{align*}
\lambda u_1(\rho) +\rho u_1'(\rho) + \frac{3}{2} u_1(\rho) - u_2(\rho) & = f_1(\rho), \\
\lambda u_2(\rho) - u_1''(\rho) -\frac{4}{\rho} u'_1(\rho) + \rho u_2'(\rho) +\frac{5}{2} u_2(\rho) & = f_2(\rho).
\end{align*}
From the first equation we see that
\begin{equation}\label{eqn:u2}
u_2(\rho) = \lf( \lambda + \frac{3}{2} \rt) u_1(\rho) + \rho u_1'(\rho)  - f_1(\rho).
\end{equation} 
Inserting this into the second one, we obtain a single ODE
\begin{equation}\label{eqn:ODElambda}
\lf(-1+\rho^2 \rt) u_1''(\rho) + \lf( - \frac{4}{\rho} + \rho\lf(2\lambda + 5 \rt) \rt) u_1'(\rho) + \lf( \lambda + \frac{5}{2} \rt)\lf( \lambda + \frac{3}{2} \rt) u_1(\rho) = F_\lambda(\rho),
\end{equation}
where $F_\lambda(\rho) := \lf( \lambda + \frac{5}{2} \rt) f_1 (\rho) +
\rho f_1'(\rho) + f_2$. By definition we have $F_\lambda \in
C^\infty([0,1])$. We choose $\lambda = 1$ and use the variation of
constants method to solve \eqref{eqn:ODElambda}. The homogeneous equation (i.e.,  $F_\lambda=0$) admits a fundamental system
$$
\phi_1 (\rho)  = \frac{2+\rho}{\rho^3(1+\rho)^{\frac{1}{2}}}, \qquad \phi_0 (\rho)  =  \frac{1}{\rho^3} \lf(\frac{2+\rho}{(1+\rho)^{\frac{1}{2}}} - \frac{2-\rho}{(1-\rho)^{\frac{1}{2}}} \rt).
$$
The Wronskian is given by
$$
W(\rho):= W(\phi_0,\phi_1)(\rho) = \frac{3}{\rho^4(1-\rho^2)^{3/2}}.
$$
A solution to the non-homogeneous problem is then given by
\begin{align*}
u_1(\rho)  = &  \phi_0(\rho) \int_\rho^1 \frac{\phi_1(s)}{W(s)}\frac{F_1(s)}{-1+s^2} ds + \phi_1(\rho) \int_0^\rho \frac{\phi_0(s)}{W(s)}\frac{F_1(s)}{-1+s^2} ds \\
= & - \frac{\phi_0(\rho)}{3}  \int_\rho^1 s(2+s)(1-s)^{\frac{1}{2}}  F_1(s) ds  \\ & \qquad \qquad  - \frac{\phi_1(\rho)}{3}  \int_0^\rho s \lf( (2+s)(1-s)^{\frac{1}{2}} - (2-s)(1+s)^{\frac{1}{2}} \rt)F_1(s)  ds
\end{align*}
By definition, we have $u_1 \in C^2(0,1)$, hence we only need to check the behaviour at the boundary points $\rho=0$ and $\rho=1$. Using l'Hospital's rule we find that $u_1$ is $C^1([0,1])$ with $u_1'(0)=0$. The second derivative of $u_1$ is given by
\begin{align*}
u_1''(\rho) =  - \frac{\phi_0''(\rho)}{3}   & \int_\rho^1  s(2+s)(1-s)^{\frac{1}{2}} F_1(s) ds \\
& -  \frac{\phi_1''(\rho)}{3}  \int_0^\rho s \lf( (2+s)(1-s)^{\frac{1}{2}} - (2-s)(1+s)^{\frac{1}{2}} \rt)F_1(s)  ds + \frac{F_1(\rho)}{(-1+\rho^2)}.
\end{align*}
Again an application of l'Hospital's rule shows that at $u_1''$ is not singular at $\rho=0$. At $\rho = 1$, observe that the second term is not singular. In the first term, we have $\phi_0''(\rho) \sim -\frac{3}{4} (1-\rho)^{-5/2}$ and $\int_0^\rho s \lf( (2+s)(1-s)^{1/2}  \rt)F_1(s)  ds \sim \frac{2}{3} (1-\rho)^{3/2} F_1(1)$ as $\rho \to 1$. So
$$
- \frac{\phi_0''(\rho)}{3} \int_\rho^1 s(2+s)(1-s)^{1/2} F_1(s) ds \sim \frac{1}{2} (1-\rho)^{-1} F_1(1) \text{ as } \rho \to 1.
$$
This cancels out exactly with the singularity in the last term, since $F_1(\rho)(-1+\rho^2)^{-1} \sim -\frac{1}{2}F(1)(1-\rho)^{-1}$ as $\rho \to 1$. Thus $u_1''$ is indeed continuous at $\rho=1$, and $u_1\in C^2([0,1])$. From \eqref{eqn:u2} we also see that $u_2 \in C^1([0,1])$. Hence we can conclude that the range of $(1-\tilde{\mathbf{L}}_0)$ is dense.

Applying the Lumer-Philips theorem we see that $\tilde{\mathbf{L}}_0$ is closable, its closure $\mathbf{L}_0$ generates a $C_0$ semigroup $\{\mathbf{S_0}(\tau):\tau\geq 0 \}$ on $\mathcal{H}$ satisfying the bound
$$
\| \mathbf{S_0}(\tau) \mathbf{f} \|_E \leq \|\mathbf{f} \|_E
$$
for all $\tau \geq 0$ and $\mathbf{f}\in \mathcal{H}$. But since the norms $\| \cdot \|_E$ and $\|\cdot \|_{\mathcal{H}}$ are equivalent, the conclusion follows.
\end{proof}

This gives the growth bound $\|\mathbf{S}_0(\tau)\|_{\mathcal{H}} \lesssim 1$, and we also have that the spectrum $\sigma(\mathbf{L_0})$ satisfies
$$
\sigma(\mathbf{L}_0) \subset \{z\in \mathbb{C} | \Re z \leq 0 \}
$$
by \cite[p.55 Theorem 1.10]{enna2000}.

\subsection{The free Strichartz estimates}

For a solution to \eqref{eqn:wave5rad} that blows up at $r=0, t=T$, due to the finite speed of propagation, this blow up is influenced only by initial conditions supported in $\mathbb{B}^5_T$.

Solutions to the free wave equation with initial conditions supported in $\mathbb{B}^5_T$, when restricted to the backwards lightcone, correspond to the semigroup $\{\mathbf{S}_0(\tau):\tau>0\}$ in similarity coordinates. The standard Strichartz estimates imply Strichartz estimates in similarity coordinates by a scaling argument.

\begin{prop}\label{prop:freestrichartz}
Let $p\in [2, \infty]$ and $q\in [\frac{10}{3},5]$ such that $\frac{1}{p}+\frac{5}{q} = \frac{3}{2}$. Then we have the bound
$$
\lf\Vert \lf[\mathbf{S}_0(\cdot) \mathbf{f} \rt]_1 \rt\Vert_{L^p(\RR_+;L^q(\BB^5))} \lesssim \|\mathbf{f} \|_{\mathcal{H}}
$$
for all $\mathbf{f} \in \mathcal{H}$. And hence we also have
$$
\lf\Vert \int_0^\tau  \lf[\mathbf{S}_0(\tau -\sigma)  \mathbf{h}(\sigma,\cdot) \rt]_1 d\sigma  \rt\Vert_{L_\tau^p(\RR_+;L^q(\BB^5))} \lesssim \|\mathbf{h} \|_{L^1(\RR_+,\mathcal{H})},
$$
for all $\mathbf{h}\in C([0,\infty),\mathcal{H})\cap L^1(\RR_+,\mathcal{H})$.
\end{prop}

\begin{proof} Let $\mathbf{f}=(f_1,f_2)\in C^2\times C^1([0,1])$ with $f'(0)=0$. Then $(f_1,f_2)$ can be extended to radial functions $(\tilde{f}_1, \tilde{f}_2)\in C^2\times C^1(\RR^5)$ by Sobolev extension, in such a way that
$$
\|\tilde{f}_1\|_{\dot{H}^1(\RR^5)} \lesssim \| f_1\|_{H^1(\BB^5)} , \text{ and } \|\tilde{f}_2 \|_{L^2(\RR^5)} \lesssim \|f_2\|_{L^2(\BB^5)}.
$$
By the change of coordinates \eqref{eqn:simcoor} with $T=1$,
the free evolution $\mathbf{S}_0(\cdot) \mathbf{f}$ is given by the
classical solution $u \in C^2(\RR_+ \times \RR_+)$, restricted to the
backwards lightcone $\Gamma_1$, of the free radial wave equation
\begin{equation}\label{eqn:freewave}
\begin{cases}
& (\PD^2_t - \PD^2_r - \frac{4}{r} \PD_r) u(t,r) = 0 \\
& u[0] = (\tilde{f}_1,\tilde{f}_2) .
\end{cases}
\end{equation}
In other words, we have
$$
[\mathbf{S}_0(\tau) \mathbf{f}]_1(\rho) = e^{-\frac{3}{2}\tau} u(1-e^{-\tau}, e^{-\tau} \rho)
$$
for $(\tau,\rho) \in \RR_+ \times [0,1]$ by \eqref{eqn:coordtrans}. 
Now we have
\begin{align*}
\|\lf[\mathbf{S}_0 (\tau) \mathbf{f} \rt]_1 \|_{L^5(\mathbb{B}^5)
} = & \lf( \int_0^1|\lf[\mathbf{S}_0 (\tau)  \mathbf{f} \rt]_1 (\rho)|^5 \rho^4 d\rho \rt)^{\frac{1}{5}} \\
= & \lf( \int_0^{1 } \lf| e^{-\frac{3}{2}\tau} u(1-e^{-\tau}, e^{-\tau} \rho )\rt|^5 \rho^4 d\rho \rt)^{\frac{1}{5}} \\
= &  e^{-\frac{1}{2}\tau} \lf(\int_0^{e^{-\tau}} \lf| u(1-e^{-\tau}, r)\rt|^5 r^4 dr \rt)^\frac{1}{5} \\
\leq & e^{-\frac{1}{2}\tau} \|u(1-e^{-\tau},\cdot)\|_{L^5 (\RR^5)} ,
\end{align*}
then
\begin{align*}
\|\lf[\mathbf{S}_0(\cdot)  \mathbf{f} \rt]_1\|_{L^2(\RR_+; L^5(\mathbb{B}^5))} \leq & \lf( \int_0^\infty \lf(e^{-\frac{1}{2}\tau} \|u(1-e^{-\tau},\cdot)\|_{L^5 (\RR^5)} \rt)^2 d\tau \rt)^\frac{1}{2}\\
= & \lf( \int_0^1  \|u(t,\cdot)\|_{L^5 (\RR^5)}^2 dt \rt)^\frac{1}{2}\\
\leq & \|u\|_{L^2(\RR_+; L^5(\RR^5))}.
\end{align*}
By the Strichartz estimates for wave equations (\cite[Theorem 2.6]{tao2006} with $d=5, s=1,q=2,r=5$), we have
$$
\|u\|_{L^2(\RR_+, L^5(\RR^5))} \lesssim \| (\tilde{f}_1,\tilde{f}_2) \| _{\dot{H}^1 \times L^2 (\RR^5)},
$$
hence we obtain 
$$
\|\lf[\mathbf{S}_0 (\cdot) \mathbf{f} \rt]_1\|_{L^2(\RR_+; L^5(\mathbb{B}^5))} \lesssim    \| (\tilde{f}_1,\tilde{f}_2) \| _{\dot{H}^1 \times L^2 (\RR^5)} \lesssim \|(f_1, f_2)\|_{H^1 \times L^2 (\mathbb{B}^5)} = \|\mathbf{f}\|_\mathcal{H}.
$$
By density this holds for all $\mathbf{f}\in \mathcal{H}$.

The other endpoint comes from the Sobolev embedding $H^1(\BB^5) \hookrightarrow L^{\frac{10}{3}}(\BB^5)$ and Proposition \ref{prop:semigp},
$$
\| \lf[\mathbf{S}_0 (\tau) \mathbf{f} \rt]_1\|_{L^\frac{10}{3}(\mathbb{B}^5)} \lesssim  \| \lf[\mathbf{S}_0 (\tau) \mathbf{f} \rt]_1\|_{H^1 (\mathbb{B}^5))} \lesssim \|\mathbf{S}_0 (\tau) \mathbf{f} \|_{\mathcal{H}} \lesssim \|  \mathbf{f} \|_{\mathcal{H}}.
$$
Hence $\| \lf[\mathbf{S}_0 (\cdot) \mathbf{f} \rt]_1\|_{L^\infty(\RR_+; L^\frac{10}{3}(\mathbb{B}^5))} \lesssim \|  \mathbf{f} \|_{\mathcal{H}}$.

For other pairs $(p,q)$ we can interpolate, for all $q \in [\frac{10}{3},5]$,
$$
\| [\mathbf{S}_0 (\tau) \mathbf{f} ]_1 \|_{L^q(\BB^5)} \leq \| [\mathbf{S}_0 (\tau) \mathbf{f} ]_1 \|_{L^{\frac{10}{3}}(\BB^5)}^{\frac{10}{q}-2} \| [\mathbf{S}_0 (\tau) \mathbf{f} ]_1 \|_{L^5(\BB^5)}^{-\frac{10}{q}+3}.
$$
Then using H\"older's inequality we have
$$
\| [\mathbf{S}_0 (\cdot) \mathbf{f} ]_1 \|_{L^p(\RR_+; L^q(\BB^5))} \leq  \| [\mathbf{S}_0 (\cdot) \mathbf{f} ]_1 \|_{L^\infty(\RR_+; L^{\frac{10}{3}}(\BB^5))}^{\frac{10}{q}-2}  \| [\mathbf{S}_0 (\cdot) \mathbf{f} ]_1 \|_{L^2(\RR_+; L^5(\BB^5))}^{-\frac{10}{q}+3} \lesssim \| \mathbf{f}\|_{\mathcal{H}},
$$
provided $(-\frac{10}{q}+3)p=2$, which gives $\frac{1}{p}+\frac{5}{q}=\frac{3}{2}$.

The inhomogeneous estimate is given by Minkowski's inequality as in \cite{donn2017}.
\end{proof}

\subsection{The linearisation}

The operator $\mathbf{L}':\mathcal{H} \to \mathcal{H}$ is bounded and also compact. We define $\mathbf{L} = \mathbf{L}_0 + \mathbf{L}'$ with $\mathcal{D}(\mathbf{L}) = \mathcal{D}(\mathbf{L}_0)$. Then by the bounded perturbation theorem, we have that $\mathbf{L}$ generates a semigroup $\mathbf{S}(\tau)$.

\begin{lemma}\label{lemma:specL} The spectrum of $\mathbf{L}$ satisfies
$$
\sigma(\mathbf{L}) \subset \{z\in \mathbb{C} | \Re z \leq 0 \} \cup \{ 1 \}.
$$
\end{lemma}
\begin{proof} Let $\lambda\in \sigma(\mathbf{L})$. If $\Re \lambda \leq 0$, then the assertion is satisfied. Thus we consider the case when $\Re \lambda > 0$. In this case, $\lambda \in \sigma(\mathbf{L}) \backslash \sigma (\mathbf{L}_0)$, so the resolvent $\mathbf{R}_{\mathbf{L}_0} (\lambda)$ exists. Then from  $(\lambda-\mathbf{L}) = [1-\mathbf{L}'\mathbf{R}_{\mathbf{L}_0}(\lambda)](\lambda - \mathbf{L}_0)$, we see that $1\in \sigma(\mathbf{L}'\mathbf{R}_{\mathbf{L}_0}(\lambda))$. But since $\mathbf{L}'\mathbf{R}_{\mathbf{L}_0}(\lambda)$ is compact, we have $1\in \sigma_p(\mathbf{L}'\mathbf{R}_{\mathbf{L}_0}(\lambda))$, so there exists $\mathbf{f}\in \mathcal{H}$ with $\mathbf{f}=\mathbf{L}'\mathbf{R}_{\mathbf{L}_0}(\lambda) \mathbf{f}$. Defining $\mathbf{u}:= \mathbf{R}_{\mathbf{L}_0}(\lambda) \mathbf{f}$, we have $(\lambda - \mathbf{L}) \mathbf{u}=0$. Hence $\lambda\in \sigma_p(\mathbf{L})$, and there exists a nonzero $\mathbf{u} \in \mathcal{D}(\mathbf{L})$ such that $(\lambda-\mathbf{L})\mathbf{u} = 0$. Following a similar computation as in Proposition \ref{prop:semigp}, the spectral equation implies
$$
-(1-\rho^2)u''(\rho) +\lf( - \frac{4}{\rho} + \rho(2\lambda + 5)  \rt) u'(\rho) + \lf(\lambda + \frac{5}{2}\rt)\lf(\lambda + \frac{3}{2}\rt)u (\rho) - \frac{35}{4} u (\rho) = 0,
$$
where $u= u_1$. We define $h(\rho^2) = u(\rho)$ and set $z=\rho^2$, then this ODE is transformed into the hypergeometric differential equation
\begin{equation}\label{eqn:ODEhgde1}
z(1-z)h''(z) +\lf( c - (a+b+1)z \rt) h'(z) - ab h(z) = 0,
\end{equation}
where $a = \frac{\lambda}{2}+\frac{5}{2}$, $b=\frac{\lambda}{2} - \frac{1}{2}$, $c=\frac{5}{2}$. In the case when $\Re(c-a-b) = \frac{1}{2}-\Re \lambda$ is not a non-positive integer, a fundamental system at $z=1$ is given by
\begin{align*}
h_1(z;\lambda) &=  {_2F_1} (a,b;a+b+1-c;1-z), \\
\tilde{h}_1(z;\lambda) &= (1-z)^{c-a-b} {_2F_1} (c-a,c-b; c-a -b+1;1-z),
\end{align*}
see \cite{dlmf}. When $\frac{1}{2}-\Re \lambda = -n$ for $n \in \mathbb{N}_0$, one solution is still given by $h_1$, while the second solution is given by
$$
\tilde{h}_1(z;\lambda) = ch_1(z;\lambda) \log(1-z) + (1-z)^{-n} g(z;\lambda)
$$
where $g$ is analytic around $z=1$ with $g(1;\lambda)\not=0$ and $c$ is a constant which can be zero unless $n=0$. Since $\Re \lambda > 0$, the requirement that $u_1 \in H^1(\BB^5)$ excludes the solution $\tilde{h}_1$, so a solution $h$ must be a multiple of $h_1$. A fundamental system at $z=0$ is given by
\begin{align*}
h_0 (z; \lambda) &= {_2F_1} (a,b;c;z), \\
\tilde{h}_0 (z; \lambda) &= z^{1-c} {_2F_1}(a-c+1,b-c+1;2-c;z).
\end{align*}
From the connection formula, we have
\begin{equation}\label{eqn:connform}
h_1(z;\lambda) = \frac{\Gamma(1-c) \Gamma(a+b-c+1)}{\Gamma(a -c+1) \Gamma(b-c+1)}h_0(z;\lambda)+\frac{\Gamma(c-1)\Gamma(a+b-c+1)}{\Gamma(a)\Gamma(b)} \tilde{h}_0(z;\lambda).
\end{equation}
Since $\rho\mapsto \tilde{h}_0(\rho^2;\lambda)$ does not belong to
$H^1(\BB^5_{1/2})$, we need to have
$$
\frac{\Gamma(c-1)\Gamma(a+b-c+1)}{\Gamma(a)\Gamma(b)}  = 0.
$$
This happens only when
$$
-a = - \frac{\lambda}{2} - \frac{5}{2}  \in \mathbb{N}_0 \text{ or } -b = -\frac{\lambda}{2}+\frac{1}{2} \in \mathbb{N}_0,
$$
and the only possibility for $\Re\lambda >0$ is when $\lambda = 1$.
\end{proof}

\begin{lemma}\label{lemma:rbound} For every $\epsilon > 0$, there exist constants $C_\epsilon, R_\epsilon >0$, such that
$$
\| \mathbf{R}_{\mathbf{L}}(\lambda) \|_{\mathcal{H}} \leq C_\epsilon
$$
for all $\lambda$ with $\Re \lambda > \epsilon$ and $|\lambda | > R_\epsilon$.
\end{lemma}
\begin{proof} Fix $\epsilon > 0$. From Proposition \ref{prop:semigp}, we have
$$
\| \mathbf{R}_{\mathbf{L}_0}(\lambda) \|_\mathcal{H} \lesssim \frac{1}{\Re \lambda}
$$
by \cite[p.55 Theorem 1.10]{enna2000}. From $(\lambda-\mathbf{L}) = [1-\mathbf{L}'\mathbf{R}_{\mathbf{L}_0}(\lambda)](\lambda - \mathbf{L}_0)$, $\mathbf{R}_\mathbf{L}(\lambda)$ exists if and only if $[1-\mathbf{L}'\mathbf{R}_{\mathbf{L}_0}(\lambda)]^{-1}$ exists. 
Now let $\mathbf{u}=\mathbf{R}_{\mathbf{L}_0}(\lambda)\mathbf{f}$, then $\mathbf{u}$ and $\mathbf{f}$ satisfy
$$
\lf( \lambda + \frac{3}{2} \rt) u_1(\rho) = f_1(\rho) -\rho u_1'(\rho) + u_2(\rho).
$$
Recall that
\begin{equation*}
\mathbf{L}' \mathbf{u} (\rho) := \begin{pmatrix}
0 \\
\frac{35}{4} u_1(\rho)
\end{pmatrix},
\end{equation*}
then
\begin{align*}
\lf| \lambda + \frac{3}{2} \rt| \|\mathbf{L}'\mathbf{R}_{\mathbf{L}_0} (\lambda)\mathbf{f}\|_\mathcal{H} = & \lf| \lambda + \frac{3}{2} \rt| \|\mathbf{L}'\mathbf{u} \|_{\mathcal{H}} \lesssim \lf| \lambda + \frac{3}{2} \rt|  \|u_1\|_{L^2(\BB^5)} \\
\leq & \|f_1\|_{L^2(\BB^5)} + \|(\cdot) u_1'\|_{L^2(\BB^5)} + \|u_2 \|_{L^2(\BB^5)} \\
\lesssim & \|\mathbf{f}\|_\mathcal{H} + \|\mathbf{u}\|_\mathcal{H} \\
= & \|\mathbf{f}\|_\mathcal{H} + \|\mathbf{R}_{\mathbf{L}_0}(\lambda)\mathbf{f} \|_\mathcal{H} \\
\lesssim & \|\mathbf{f}\|_\mathcal{H}.
\end{align*}
So $\|\mathbf{L}'\mathbf{R}_{\mathbf{L}_0}(\lambda) \|_\mathcal{H} \lesssim \lf| \lambda + \frac{3}{2} \rt|^{-1}$, and for $|\lambda|$ large enough, we obtain $\|\mathbf{L}'\mathbf{R}_{\mathbf{L}_0} (\lambda)\|_\mathcal{H}  < 1$. Therefore the Neumann series for $[1-\mathbf{L}'\mathbf{R}_{\mathbf{L}_0}(\lambda)]^{-1}$ converges, and we have
\begin{align*}
\|\mathbf{R}_\mathbf{L}(\lambda) \|_\mathcal{H} = & \| \mathbf{R}_{\mathbf{L}_0} (\lambda) [1-\mathbf{L}'\mathbf{R}_{\mathbf{L}_0}]^{-1} (\lambda)\|_\mathcal{H} \\
\leq & \| \mathbf{R}_{\mathbf{L}_0}\|_\mathcal{H} \sum_{k=0}^\infty \|\mathbf{L}'\mathbf{R}_{\mathbf{L}_0}(\lambda)\|_\mathcal{H}^k \\
\leq & C_\epsilon
\end{align*}
as claimed.
\end{proof}

\begin{lemma}\label{lemma:eigenspace} The geometric multiplicity of the eigenvalue $1 \in \sigma_p(\mathbf{L})$ is 1, and its geometric eigenspace is spanned by
$$
\mathbf{g}(\rho) = \begin{pmatrix} 2 \\ 5
\end{pmatrix}.
$$
\end{lemma}
\begin{proof} First we note that $\mathbf{g}\in \mathcal{D}(\mathbf{L})$. A straightforward computation shows that $(1-\mathbf{L})\mathbf{g} = 0$. Now suppose that $\tilde{\mathbf{g}}$ is another eigenfunction with eigenvalue $1$, then $[\tilde{\mathbf{g}}]_1$ also satisfies the transformed ODE \eqref{eqn:ODEhgde1} with $\lambda=1$. In this case $[\tilde{\mathbf{g}}]_1(\rho) = c \cdot {_2 F_1}(a,b,c;\rho^2)$ for $b=\frac{\lambda}{2}-\frac{1}{2} =0$ and some constant $c\in \mathbb{C}\backslash\{0\}$. Therefore $[\tilde{\mathbf{g}}]_1(\rho)$ is a constant, and we must have $[\tilde{\mathbf{g}}]_1 = c' [\mathbf{g}]_1$ for some constant $c'\in \mathbb{C}\backslash\{0\}$. The second component is uniquely determined by the first component, hence we must have $\tilde{\mathbf{g}} = c' \mathbf{g}$.
\end{proof}

The next lemma shows that the algebraic multiplicity of the eigenvalue $\lambda=1$ is also 1, so we can use the Riesz projection to remove the unstable subspace.

\begin{lemma}\label{lemma:boundedpert} There exists a bounded projection $\mathbf{P}: \mathcal{H} \to \lla \mathbf{g} \rra$ such that $\mathbf{P}\mathbf{S}(\tau) = \mathbf{S}(\tau)\mathbf{P}$. In particular, we have $\mathbf{S}(\tau) \mathbf{Pf} = e^\tau \mathbf{P f}$ for all $\tau \geq 0$ and $\mathbf{f}\in \mathcal{H}$. 
\end{lemma}
\begin{proof}
Since $\lambda=1$ is an isolated point in the spectrum, we define the projection
$$
\mathbf{P} = \frac{1}{2\pi i} \int_\gamma \mathbf{R}_{\mathbf{L}}(\lambda) d\lambda,
$$
where $\gamma$ is a curve in $\rho(\mathbf{L})$ enclosing the point $1$, such that no other points of $\sigma(\mathbf{L})$ lie in it. Then $\mathbf{P}$ commutes with the operator $\mathbf{L}$ (this means $\mathbf{P}\mathbf{L}\subset \mathbf{L}\mathbf{P}$) and the semigroup $\mathbf{S}(\tau)$. We have a decomposition of the Hilbert space $\mathcal{H} = \ker \mathbf{P} \oplus \rg \mathbf{P}$. Then we have the parts of the operator $\mathbf{L}$, $\mathbf{L}_{\ker\mathbf{P}}$ and $\mathbf{L}_{\rg \mathbf{P}}$, as operators on the spaces $ \ker \mathbf{P}$ and $\rg \mathbf{P}$ respectively, with spectra $\sigma (\mathbf{L}_{\ker\mathbf{P}}) = \sigma(\mathbf{L}) \backslash \{1\}$ and $\sigma(\mathbf{L}_{\rg\mathbf{P}}) = \{1\}$. It is obvious that $\lla \mathbf{g} \rra \subset \rg \mathbf{P}$, so we need to show the other inclusion.

First of all, notice that $\dim \rg \mathbf{P} < \infty$. Otherwise,
$1$ would be in the essential spectrum of $\mathbf{L}$, which is
impossible since $\mathbf{L}'$ is compact and $1 \notin
\sigma(\mathbf{L}_0)$(\cite[p.244 Theorem 5.35]{kato1995}). Hence
there is a minimal $k \in \mathbb{N}$ such that $(1-\mathbf{L}_{\rg
  \mathbf{P}})^k \mathbf{u}=0$ for all $\mathbf{u} \in \rg
\mathbf{P}\cap \mathcal{D}(\mathbf{L})$. If $k=1$ then we are
done. Now we argue by contradiction. If $k\geq 2$, then there exists $\mathbf{u} \in \mathcal{D}(\mathbf{L})$ such that $(1-\mathbf{L})\mathbf{u} = \mathbf{g}$. This is the same as saying that there exists a solution to the inhomogeneous ODE
$$
-(1-\rho^2)u''(\rho) +\lf( - \frac{4}{\rho} + 7 \rho  \rt) u'(\rho) = 12 ,
$$
where $u=u_1$. For the homogeneous problem, we already know that one solution is given by the constant solution $g_1(\rho) = 2$. Another linearly independent solution to the homogeneous problem is given by
$$
\tilde{g}_1(\rho) = \frac{1+4\rho^2-8\rho^4}{\rho^3(1-\rho^2)^\frac{1}{2}}.
$$
The Wronskian is
$$
W(g_1, \tilde{g}_1) (\rho) = \frac{6}{\rho^4(1-\rho^2)^\frac{3}{2}}.
$$
Solutions to the inhomogeneous problem are given by
\begin{align*}
u(\rho) = & C_1 + C_2 \tilde{g}_1(\rho) + 24 \int_0^\rho \frac{\tilde{g}_1(s)}{W(g_1, \tilde{g}_1) (s)} \frac{1}{1-s^2} ds + 24 \tilde{g}_1(\rho) \int_\rho^1  \frac{1}{W(g_1, \tilde{g}_1) (s)} \frac{1}{1-s^2} ds \\
= &  C_1 + C_2 \tilde{g}_1(\rho) + 24 \int_0^\rho s(1+4s^2-8s^4) ds + \frac{24(1+4\rho^2-8\rho^4) }{\rho^3(1-\rho^2)^\frac{1}{2}} \int_\rho^1 s^4(1-s^2)^\frac{1}{2} ds
\end{align*}
where $C_1, C_2 \in \RR$ are arbitrary constants. At $\rho=1$, the
requirement that $u \in H^1(\BB^5)$ gives $C_2=0$ and by looking at
the asymptotic behavior at $\rho=0$, we see that we must have
$$
 \int_0^1 s^4(1-s^2)^\frac{1}{2} ds = 0.
$$
However, this clearly fails because the integrand is positive for all $s\in(0,1)$.
\end{proof}

Next, we apply the Gearhart--Pr\"uss theorem to obtain a growth bound for the semigroup. 

\begin{lemma}\label{lemma:growthbound}
For any $\epsilon>0$, there exists a constant $C_{\epsilon} > 0$ such that
$$
\| \mathbf{S}(\tau)(\mathbf{I} - \mathbf{P})\mathbf{f}\|_{\mathcal{H}} \leq C_\epsilon e^{\epsilon \tau} \| (\mathbf{I} -\mathbf{P}) \mathbf{f} \|_{\mathcal{H}}
$$
for all $\tau \geq 0$ and $\mathbf{f}\in \mathcal{H}$.
\end{lemma}
\begin{proof}
From Lemma \ref{lemma:rbound} and the fact that $\sigma(\mathbf{L}_{\ker \mathbf{p}}) \subset \{ z\in \mathbb{C} | \Re z \leq 0\}$, we have that for every $\epsilon>0$ there exists $C_\epsilon>0$ such that
$$
\| \mathbf{R}_\mathbf{L}(\lambda) (\mathbf{I}-\mathbf{P}) \|_{\mathcal{H}} \leq C_\epsilon
$$
for all $\lambda$ with $\Re \lambda > \epsilon$. Then from the Gearhart--Pr\"uss theorem (\cite[p.302 Theorem 1.11]{enna2000}) we obtain the stated growth bound.
\end{proof}

%
%

\subsection{Explicit representation of the semigroup}

Let $\mathbf{f} \in C^2\times C^1([0,1]) \cap \mathcal{D}(\mathbf{L})$ and $\tilde{\mathbf{f}} := (\mathbf{I}-\mathbf{P}) \mathbf{f}$. Then $\tilde{\mathbf{f}} \in C^2\times C^1([0,1])\cap \mathcal{D}(\mathbf{L}_0)$. By Laplace inversion (\cite[p.234 Corollary 5.15]{enna2000}), we have
$$
\mathbf{S}(\tau)(\mathbf{I}-\mathbf{P}) \mathbf{f} = \frac{1}{2\pi i} \lim_{N \to \infty} \int_{\epsilon-i N}^{\epsilon + i N} e^{\lambda \tau} \mathbf{R}_\mathbf{L}(\lambda) \tilde{\mathbf{f}} d\lambda
$$
for any $\epsilon > 0$. Hence we would like to find an explicit expression for $\mathbf{R}_\mathbf{L}(\lambda)$. Since  $\mathbf{u} = \mathbf{R}_\mathbf{L} (\lambda) \tilde{\mathbf{f}}$ means $(\lambda-\mathbf{L}) \mathbf{u} = \tilde{\mathbf{f}}$, following the exact same computation as in Proposition \ref{prop:semigp} we get the ODE
\begin{equation}\label{eqn:ODElambdaV}
-(1-\rho^2)u''(\rho) +\lf( - \frac{4}{\rho} + \rho(2\lambda + 5)  \rt) u'(\rho) + \lf(\lambda + \frac{5}{2}\rt)\lf(\lambda + \frac{3}{2}\rt)u(\rho) - \frac{35}{4} u(\rho) = F_\lambda,
\end{equation}
where $u=u_1$ and
$$
F_\lambda(\rho) = \lf(\lambda+\frac{5}{2} \rt)\tilde{f}_1(\rho) + \rho \tilde{f}'_1(\rho)  + \tilde{f}_2(\rho).
$$
Thus
$$
[\mathbf{R}_\mathbf{L} (\lambda) \tilde{\mathbf{f}}]_1(\rho) = \int_0^1 G(\rho,s;\lambda) F_\lambda(s) ds,
$$
where $G$ is the Green function of Equation \eqref{eqn:ODElambdaV}. Then we have the expression
\begin{equation}\label{eqn:laplaceinv}
[\mathbf{S}(\tau)(\mathbf{I}-\mathbf{P}) \mathbf{f}]_1(\rho) = \frac{1}{2\pi i}\lim_{N\to\infty} \int_{\epsilon-i N}^{\epsilon + i N} e^{\lambda \tau} \int_0^1 G(\rho,s;\lambda) F_\lambda(s) ds d\lambda
\end{equation}
for any $\epsilon>0$.

\section{Constructing the Green function}

In this section we construct the Green function for \eqref{eqn:ODElambdaV}. Most of our construction here works for the equation
\begin{equation}\label{eqn:ODEpotential}
-(1-\rho^2)u''(\rho) +\lf( - \frac{4}{\rho} + \rho(2\lambda + 5)  \rt) u'(\rho) + \lf(\lambda + \frac{5}{2}\rt)\lf(\lambda + \frac{3}{2}\rt)u(\rho) + V(\rho) u(\rho) = F_\lambda
\end{equation}
with any $V\in C^\infty([0,1])$.

First of all, we set
$$
v(\rho) := \rho^2 \lf( 1- \rho^2 \rt)^{\frac{1}{4}+\frac{\lambda}{2}} u(\rho).
$$
Then \eqref{eqn:ODEpotential} with $F_\lambda=0$ is equivalent to
\begin{equation}\label{eqn:ODEv}
v''(\rho) + \frac{-2+\frac{1}{4}\rho^2 (11+4\lambda-4\lambda^2)}{\rho^2 \lf( 1- \rho^2 \rt)^2} v(\rho) = \frac{V(\rho)}{1-\rho^2}v(\rho),
\end{equation}
but now \eqref{eqn:ODEv} has no first order terms. Notice that if $v(\cdot;\lambda)$ is a solution to \eqref{eqn:ODEv}, then so is $v(\cdot;1-\lambda)$.

\subsection{Symbol calculus}

\begin{defn} For a function $f:I \subset \RR \to \mathbb{C}$, $a\in \overline{I}$, and $\alpha\in \RR$, we write $f(x) = \mathcal{O}((x-a)^\alpha)$ if
$$
\lf|\PD^j_x f(x) \rt| \leq C_j |x-a|^{\alpha-j}
$$
for all $x\in I$ and $j\in \mathbb{N}_0$. Similarly, for a function $f:\RR \to \mathbb{C}$ and $\alpha \in \RR$, we write $f(x) = \mathcal{O}(\lla x \rra ^\alpha)$ if
$$
\lf| \PD^j_x f(x) \rt| \leq C_j \lla x \rra^{\alpha-j}
$$
for all $x\in \RR$ and $j\in \mathbb{N}_0$.

We say that these functions \textit{are of symbol type}, or \textit{behave like symbols}.

\end{defn}


 Functions of symbol type have some nice properties. For example, we can easily check that if $f(x) = \mathcal{O}(x^\alpha)$ and $g(x)=\mathcal{O}(x^\alpha)$, then $f(x)+g(x)=\mathcal{O}(x^\alpha)$. Some further properties of functions of symbol type are discussed in Appendix \ref{app:symbol}.

The same notation is also used for functions with more than one variables. For example, for $f: I \times \RR \to \mathbb{C}$, we write $f(x,y) = \mathcal{O}(x^\alpha \lla y \rra^\beta)$, if
$$
\lf|\PD^j_x \PD^k_y f(x,y) \rt| \leq C_{j,k} |x|^{\alpha-j} \lla y \rra^{\beta-k}
$$
for all $(x,y) \in U$ and $j,k\in \mathbb{N}_0$.

\begin{remark} In the following, we work with functions of three variables $\rho, s$, and $\omega$. We always have $\rho,s\in [0,1]$ and $\omega \in \RR$. We write $\mathcal{O}_o$ and $\mathcal{O}_e$ to indicate oddness and evenness with respect to the variable $\omega$. For example, $f(\rho,\omega) = \mathcal{O}_o(\rho \lla \omega\rra^{-1})$ means that $f(\rho,\cdot):\RR \to \mathbb{C}$ is an odd function for each $\rho$.
\end{remark}

\subsection{The homogeneous equation}

On solving \eqref{eqn:ODEv} with $V=0$, we obtain a fundamental system
\begin{align*}
\psi_1(\rho;\lambda) & = \frac{1}{\rho} (1-\rho)^{\frac{1}{4}+\frac{\lambda}{2}}(1+\rho)^{\frac{3}{4}-\frac{\lambda}{2}} (2+\rho(-1+2\lambda)), \\
\tilde{\psi}_1(\rho;\lambda) & = \psi_1(\rho; 1-\lambda) = \frac{1}{\rho} (1+\rho)^{\frac{1}{4}+\frac{\lambda}{2}}(1-\rho)^{\frac{3}{4}-\frac{\lambda}{2}} (2+\rho(1-2\lambda)).
\end{align*}
We define another solution,
$$
\psi_0 (\rho;\lambda) := \psi_1 (\rho;\lambda) -\tilde{\psi}_1(\rho;\lambda).
$$
Notice that $\psi_0$ is not singular at $\rho = 0$, since the singularities of $\psi_1$ and $\tilde{\psi}_1$ cancel out.

Since \eqref{eqn:ODEv} has no first order terms, the Wronskian of any two solutions is independent of $\rho$. Some useful Wronskians between the above three solutions are 
\begin{align*}
 W(\lambda) := & W(\psi_1(\cdot;\lambda), \tilde{\psi}_1(\cdot;\lambda)) 
 =  W(\psi_0(\cdot;\lambda), \psi_1(\cdot;\lambda))  =   W(\psi_0(\cdot;\lambda), \tilde{\psi}_1(\cdot;\lambda)) \\  = & (3-2\lambda)(1+2\lambda)(-1+2\lambda).
\end{align*}

We can define another solution as follows.
\begin{lemma}\label{lemma:sol0} Let $\lambda= \epsilon+i\omega$. There exists $\delta_0>0$, such that 
$$
\tilde{\psi}_0(\rho;\lambda) : = \psi_0 (\rho;\lambda) \int_\rho^{\delta_0\lla \omega \rra ^{-1}} \psi_0 (s;\lambda) ^{-2} ds
$$
is also a solution to \eqref{eqn:ODEv} with $V=0$ for $\rho \in (0, \delta_0 \lla \omega \rra^{-1}]$. Moreover, it satisfies
$$
\tilde{\psi}_0(\rho;\lambda) = -2 \rho^{-1} W(\lambda)^{-1} [1+\mathcal{O}(\rho \lla \lambda \rra)]
$$
for all $\rho \in (0, \delta_0 \lla \omega \rra^{-1}]$,  $\omega\in \RR$, and $\epsilon\in [0,\frac{1}{4}]$. We also have
$$
 W(\psi_0(\cdot;\lambda), \tilde{\psi}_0(\cdot;\lambda))  = -1.
$$
\end{lemma}
\begin{proof} Taking the Taylor expansion we have $\psi_0(\rho;\lambda) = -\frac{1}{6}\rho^2W(\lambda) [1+ \mathcal{O}(\rho \lla \lambda \rra)]$. Hence for sufficiently small $\delta_0 > 0$, we have $\psi_0(\rho;\lambda) \neq 0$ for $\rho\in (0,\delta_0 \lla \omega \rra ^{-1}]$, and thus $\psi_0(\rho;\lambda)^{-1} = -6 \rho^{-2}  W(\lambda)^{-1} [1 + \mathcal{O}(\rho \lla \lambda \rra)]$. Hence $\tilde{\psi}_0$ is well-defined for $\rho\in (0,\delta_0 \lla \omega \rra ^{-1}]$. Differentiating $\tilde{\psi}_0$ with respect to $\rho$ we can see that it is indeed a solution to \eqref{eqn:ODEv} for $V=0$.

Now compute
\begin{align*}
\frac{\int_\rho^{\delta_0\lla \omega\rra^{-1} } \psi_0(s;\lambda)^{-2} ds}{12 \rho^{-3} W(\lambda)^{-2}} & = \frac{1}{12} \rho^3 W(\lambda)^2 \int_\rho^{\delta_0\lla \omega\rra^{-1} } \psi_0(s;\lambda)^{-2} ds \\
& = \frac{1}{12} \rho^3 W(\lambda)^2 \int_\rho^{\delta_0\lla \omega\rra^{-1} } \lf( -6 s^{-2}  W(\lambda)^{-1}\rt)^2 [1 + \mathcal{O}(s \lla \lambda \rra)] ds \\
& = 3 \rho^3 \int_\rho^{\delta_0\lla \omega\rra^{-1} } s^{-4} [1+\mathcal{O}(s\lla \lambda \rra)] ds \\
& = 3 \rho^3 \lf[ -\frac{1}{3} s^{-3} \rt]^{s=\delta_0 \lla \omega \rra^{-1}}_{s=\rho} + \mathcal{O}(\rho \lla \lambda \rra) \\
& = 1 + \mathcal{O}(\rho \lla \lambda \rra),
\end{align*}
for $\lambda = \epsilon + i\omega$ with $\epsilon \in [0,\frac{1}{4}]$. The value of the Wronskian follows from a straightforward computation.
\end{proof}

\subsection{The fundamental system}

In the following, we use Volterra iterations to construct solutions for \eqref{eqn:ODEv} (see \cite[Lemma 2.4]{sss2010}).

\begin{lemma}\label{lemma:pertsol0} Let $\lambda=\epsilon+i \omega$. There exists $\delta_0 > 0$ such that \eqref{eqn:ODEv} has a solution of the form
$$
v_0(\rho;\lambda)  = \psi_0(\rho;\lambda) \lf[1+ \mathcal{O} (\rho^2 \lla \omega\rra^0) \rt] 
$$
for all $\rho \in \lf[0,\delta_0\lla \omega\rra^{-1}\rt]$, $\omega \in \RR$, $\epsilon \in [0,\frac{1}{4}]$.
\end{lemma}
\begin{proof} We use Volterra iterations as in \cite{donn2017}. Using the fundamental system $\{\psi_0, \psi_1\}$, the variation of constants formula suggests that a solution $v_0$ should satisfy the integral equation 
\begin{align*}
v_0(\rho;\lambda) = \psi_0(\rho;\lambda) & - \frac{\psi_0(\rho,\lambda)}{ W(\lambda)} \int^\rho_0 \psi_1(s;\lambda) \frac{V(s)}{1-s^2}v_0(s;\lambda) ds \\
& + \frac{\psi_1(\rho,\lambda)}{ W(\lambda)} \int^\rho_0 \psi_0(s;\lambda) \frac{V(s)}{1-s^2}v_0(s;\lambda) ds.
\end{align*}

For sufficiently small $\delta_0 > 0$, we have $\psi_0(\rho;\lambda) \neq 0$ for $\rho\in (0,\delta_0 \lla \omega \rra ^{-1}]$. So we can set $h_0 := \frac{v_0}{\psi_0}$, and we obtain the equation
\begin{equation} \label{eqn:ODEh0} 
h_0 (\rho;\lambda) = 1 + \int_0^\rho K(\rho,s;\lambda) h_0(s;\lambda) ds,
\end{equation}
where
$$
K(\rho,s;\lambda) = \frac{1}{W(\lambda)}\lf[ \frac{\psi_1(\rho;\lambda)}{\psi_0(\rho;\lambda)} \psi_0(s;\lambda)^2 - \psi_0(s;\lambda) \psi_1(s;\lambda) \rt] \frac{V(s)}{1-s^2}.
$$
As mentioned before, by Taylor expansion we have $\psi_0(\rho;\lambda)^{-1} = -6 \rho^{-2}  W(\lambda)^{-1} [1 + \mathcal{O}(\rho \lla \lambda \rra)]$. Since $s\leq \rho$ , we have,
\begin{align*}
\lf| \frac{1}{W(\lambda)} \frac{\psi_1(\rho;\lambda)}{\psi_0(\rho;\lambda)} \psi_0(s;\lambda)^2 \rt| \lesssim & \psi_1(\rho;\lambda) \rho^{-2} s^4 \lesssim  \psi_1(\rho;\lambda) s^2,
\end{align*}
and similarly
$$
\lf| \frac{1}{W(\lambda)} \psi_0(s;\lambda) \psi_1(s;\lambda)  \rt|  \lesssim \psi_1(s;\lambda) s^2 \lesssim  \psi_1(s;\lambda) s^2 .
$$
Since $|\psi_1(\rho;\lambda)| \lesssim \rho^{-1}$ for $\rho\leq \delta_0\lla \omega\rra^{-1}$, we conclude that
\begin{equation}\label{eqn:kbound0}
|K(\rho, s; \lambda) | \lesssim (\rho^{-1} + s^{-1} ) s^2 \leq s
\end{equation}
for all $0\leq s \leq \rho \leq \delta_0\lla \omega \rra^{-1}$. Thus we have
$$
\int_0^{\delta_0 \lla \omega \rra ^{-1}} \sup_{\rho \in \lf(0,\delta_0 \lla \omega \rra^{-1}\rt)} |K(\rho,s;\lambda)| ds \lesssim \lla \omega \rra ^{-2}.
$$
Hence \eqref{eqn:ODEh0} has a solution $h_0$ that satisfies
$$
| h_0(\rho;\lambda) - 1| \lesssim \int_0^\rho | K(\rho, s;\lambda) | ds \lesssim \rho^2.
$$
Re-inserting this into the equation we find
$$
h_0(\rho;\lambda) = 1 + O(\rho^2).
$$

Since every function behaves like symbol under differentiation with respect to both $\rho$ and $\omega$, we indeed have $h_0(\rho;\lambda) = 1 + \mathcal{O}(\rho^2\lla \omega\rra^{0})$ (see for example \cite[Appendix B]{dss2011}). 
\end{proof}

Similar to before, we can define a second solution for small $\rho$.

\begin{lemma}\label{lemma:pertsol02}  Let $\lambda=\epsilon+i \omega$. There exists $\delta_0>0$ such that
$$
\tilde{v}_0(\rho;\lambda) := v_0(\rho;\lambda) \int_\rho^{\delta_0 \lla \omega \rra^{-1}} v_0(s;\lambda)^{-2} ds
$$
is also a solution to \eqref{eqn:ODEv} on $\rho \in (0,\delta_0\lla \omega\rra^{-1}]$. Moreover, $\tilde{v}_0$ is of the form
$$
\tilde{v}_0(\rho;\lambda) = \tilde{\psi}_0(\rho;\lambda) \lf[1+ \mathcal{O}(\rho^2 \lla \omega\rra^0) \rt] ,
$$
for all $\rho \in \lf(0,\delta_0\lla \omega\rra^{-1}\rt]$, $\omega \in \RR$, $\epsilon \in [0,\frac{1}{4}]$, where $\tilde{\psi}_0$ is from Lemma \ref{lemma:sol0}. We also have
$$
W(v_0(\cdot ; \lambda) , \tilde{v}_0(\cdot ; \lambda)) = -1.
$$
\end{lemma}

\begin{proof} We choose $\delta_0 > 0$ so small that $v_0$, $\psi_0$, and $\tilde{\psi}_0$ have no zeros on $\rho\in [0,\delta_0\lla \omega\rra^{-1}]$. Then $\tilde{v}_0$ is well-defined on $\rho \in (0,\delta_0\lla \omega\rra^{-1}]$. We can directly verify that $\tilde{v}_0$ indeed solves \eqref{eqn:ODEv}.

Now notice that at $\rho = \delta_0\lla \omega\rra^{-1}$, both $\tilde{v}_0$ and $\tilde{\psi}_0$ are $0$, so the statement of the lemma holds. Then using Lemmas \ref{lemma:sol0} and \ref{lemma:pertsol0} we compute
\begin{align*}
 & \frac{\tilde{v}_0(\rho;\lambda)}{\tilde{\psi}_0(\rho;\lambda)} - 1  \\
  = &  \frac{1}{\tilde{\psi}_0(\rho;\lambda)} \lf( \tilde{v}_0(\rho;\lambda)-\tilde{\psi}_0(\rho;\lambda) \rt) \\
 = & \frac{\psi_0 (\rho;\lambda)}{\tilde{\psi}_0(\rho;\lambda)} \lf( \lf[1+ O (\rho^2 \lla \omega\rra^0)  \rt] \int_\rho^{\delta_0 \lla \omega \rra^{-1}}v_0(s;\lambda)^{-2} ds  -  \int_\rho^{\delta_0 \lla \omega \rra^{-1}} \psi_0(s;\lambda)^{-2} ds   \rt) \\
 = &  \frac{\psi_0 (\rho;\lambda)}{\tilde{\psi}_0(\rho;\lambda)} \times\\
 & \lf(\int_\rho^{\delta_0 \lla \omega \rra^{-1}} \psi_0(s;\lambda)^{-2}  \mathcal{O} (s^2 \lla \omega\rra^0 ) ds   +  \mathcal{O} (\rho^2 \lla \omega\rra^0)\int_\rho^{\delta_0 \lla \omega \rra^{-1}} \psi_0(s;\lambda)^{-2}  \lf[1+ \mathcal{O} (s^2 \lla \omega\rra^0)  \rt] ds \rt) .
\end{align*}
For $\delta_0>0$ small enough, we also have $\tilde{\psi}_0(\rho;\lambda)^{-1} = -\frac{1}{2}\rho W(\lambda)[1+\mathcal{O}(\rho\lla \omega \rra)]$ for $\rho \in (0,\delta_0\lla \omega\rra^{-1}]$, from Lemma \ref{lemma:sol0}. So using the asymptotics of $\psi_0$ and $\tilde{\psi}_0$, we find 
$$
\frac{\psi_0 (\rho;\lambda)}{\tilde{\psi}_0(\rho;\lambda)} = \mathcal{O}(\rho^2 \lla \omega \rra^3) \mathcal{O}(\rho \lla \omega \rra^{3}) = \mathcal{O}(\rho ^3 \lla \omega \rra ^{6} ).
$$
Moreover,
$$
\int_\rho^{\delta_0 \lla \omega \rra^{-1}} \psi_0(s;\lambda)^{-2}  \mathcal{O}(s^2 \lla \omega\rra^0 ) ds =\int_\rho^{\delta_0 \lla \omega \rra^{-1}}  \mathcal{O}(s^{-4} \lla \omega\rra^{-6}) \mathcal{O}(s^2 \lla \omega\rra^0 ) ds = \mathcal{O}(\rho^{-1} \lla \omega \rra ^{-6}),
$$ 
and
\begin{align*}
& \mathcal{O} (\rho^2 \lla \omega\rra^0) \int_\rho^{\delta_0 \lla \omega \rra^{-1}} \psi_0(s;\lambda)^{-2}  \lf[1+ \mathcal{O} (s^2 \lla \omega\rra^0)  \rt] ds  \\
 = & \mathcal{O} (\rho^2 \lla \omega\rra^0)\int_\rho^{\delta_0 \lla \omega \rra^{-1}} \mathcal{O}(s^{-4} \lla \omega\rra^{-6}) + \mathcal{O}(s^{-2} \lla \omega\rra^{-6}) ds = \mathcal{O}(\rho^{-1} \lla \omega \rra^{-6}).
\end{align*}
So we obtain
$$
 \frac{\tilde{v}_0(\rho;\lambda)}{\tilde{\psi}_0(\rho;\lambda)} - 1  = \mathcal{O}( \rho^2 \lla \omega \rra^0).
$$
The value of the Wronskian follows from a straightforward computation.
\end{proof}

Next, we construct perturbed solutions from $\psi_1$ and $\tilde{\psi}_1$.

\begin{lemma}\label{lemma:pertsol1} Let $\lambda=\epsilon+i \omega$. There exists $\delta_1>0$, such that \eqref{eqn:ODEv} has solutions of the form
\begin{align*}
v_1(\rho;\lambda) & = \psi_1(\rho;\lambda) \lf[1 +  \frac{1-2\lambda}{(3-2\lambda)(1+2\lambda)} a_1(\rho)+ \mathcal{O} (\rho^0 (1-\rho)\lla \omega\rra^{-2}) \rt],  \\
\tilde{v}_1(\rho;\lambda) & = \tilde{\psi}_1(\rho;\lambda) \lf[1 -  \frac{1-2\lambda}{(3-2\lambda)(1+2\lambda)} a_1(\rho)+  \mathcal{O}(\rho^0 (1-\rho)\lla \omega\rra^{-2}) \rt],
\end{align*}
or alternatively we can also write
\begin{align*}
v_1(\rho;\lambda) & = \psi_1(\rho;\lambda) \lf[1 + \mathcal{O}_o(\lla \omega\rra^{-1}) a_1(\rho)+ \mathcal{O} (\rho^0 (1-\rho)\lla \omega\rra^{-2}) \rt] , \\
\tilde{v}_1(\rho;\lambda) & = \tilde{\psi}_1(\rho;\lambda) \lf[1 +\mathcal{O}_o(\lla \omega\rra^{-1}) a_1(\rho) + \mathcal{O} (\rho^0 (1-\rho)\lla \omega\rra^{-2}) \rt],
\end{align*}
for all $\rho \in \lf[\delta_1 \lla \omega \rra^{-1}, 1\rt], \omega \in \RR, \epsilon \in \lf[ 0, \frac{1}{4} \rt]$, where
$$
a_1(\rho) = - \int_\rho^1 V(s) ds.
$$
\end{lemma}

\begin{proof} This also follows from Volterra iterations, but is technically a little more complicated.

\textbf{Existence of solution}

Motivated by the variation of constants formula, solutions to \eqref{eqn:ODEv} should satisfy the integral equation
\begin{align*}
v_1(\rho;\lambda) = \psi_1(\rho;\lambda) & + \frac{\psi_1(\rho,\lambda)}{ W(\lambda)} \int_\rho^{\rho_1} \tilde{\psi}_1(s;\lambda) \frac{V(s)}{1-s^2}v_1(s;\lambda) ds \\
& - \frac{\tilde{\psi}_1(\rho,\lambda)}{ W(\lambda)} \int_\rho^{\rho_1} \psi_1(s;\lambda) \frac{V(s)}{1-s^2}v_1(s;\lambda) ds
\end{align*}
for some constant $\rho_1$. Since $\Re\lambda  \geq 0$, $|\psi_1(\rho;\lambda)| > 0$, so we can set $h_1 := \frac{v_1}{\psi_1}$ and obtain the Volterra equation
\begin{equation}\label{eqn:ODEh1}
h_1(\rho;\lambda) = 1 + \int_\rho^{\rho_1} K (\rho,s;\lambda) h_1(s;\lambda) ds,
\end{equation}
where
$$
 K (\rho,s;\lambda) = \frac{1}{W(\lambda)} \lf( \psi_1(s;\lambda)\tilde{\psi}_1(s;\lambda) -\frac{\tilde{\psi}_1(\rho;\lambda)}{\psi_1(\rho;\lambda)} \psi_1(s;\lambda)^2 \rt) \frac{V(s)}{1-s^2}.
$$
We first compute
$$
\psi_1(s;\lambda)\tilde{\psi}_1(s;\lambda)  = \frac{\lf(1-s^2 \rt) \lf(4-s^2(1-2\lambda)^2 \rt)}{s^2},
$$
so
\begin{align*}
|\psi_1(s;\lambda)\tilde{\psi}_1(s;\lambda)| & \lesssim \frac{1-s}{s^2} \lf| 4-s^2(1-2\lambda)^2 \rt| \\
& \lesssim \frac{1-s}{s^2} + (1-s)\lla \lambda \rra^2 \\
& \leq  (1-\rho)^\frac{1}{2}(1-s)^\frac{1}{2} \lf(\frac{1}{s^2} + \lla \lambda \rra^2 \rt),
\end{align*}
for all $0 \leq \rho \leq s \leq 1$ and $\Re(\lambda) \in [0,\frac{1}{4}]$. Moreover,
$$
\frac{\tilde{\psi}_1(\rho;\lambda)}{\psi_1(\rho;\lambda)} \psi_1(s;\lambda)^2 = \lf( \frac{1-\rho}{1+\rho}\rt)^{\frac{1}{2}-\lambda} \frac{2+\rho(1-2\lambda)}{2+\rho(-1+2\lambda)}\frac{(1-s)^{\frac{1}{2}+\lambda}(1+s)^{\frac{3}{2}-\lambda}}{s^2}(2+s(-1+2\lambda))^2 ,
$$
so
\begin{align*}
\lf| \frac{\tilde{\psi}_1(\rho;\lambda)}{\psi_1(\rho;\lambda)} \psi_1(s;\lambda)^2  \rt| & \lesssim \frac{ (1-\rho)^\frac{1}{2}(1-s)^\frac{1}{2}}{s^2}\lf|\frac{2+\rho(1-2\lambda)}{2+\rho(-1+2\lambda)} \rt| \lf| 2+s(-1+2\lambda) \rt|^2 \\
& \lesssim  (1-\rho)^\frac{1}{2}(1-s)^\frac{1}{2} \lf(\frac{1}{s^2} + \lla \lambda \rra^2 \rt),
\end{align*}
for all $0 \leq \rho \leq s \leq 1$ and $\Re(\lambda) \in [0,\frac{1}{4}]$. Hence, when restricted to $\delta_1 \lla \omega \rra^{-1} \leq \rho \leq s < 1$ for some small $\delta_1>0$, we obtain the bound for the kernel
\begin{align*}
| K(\rho, s; \lambda) | \lesssim &  \frac{1}{W(\lambda)}\frac{V(s)}{1-s^2}  (1-\rho)^\frac{1}{2}(1-s)^\frac{1}{2} \lf(\frac{1}{s^2} + \lla \lambda \rra^2 \rt) \\
\lesssim & (1-\rho)^\frac{1}{2}(1-s)^{-\frac{1}{2}} \lla \omega \rra^{-1}.
\end{align*}
Now we can choose $\rho_1 = 1$ and have
$$
\int_{\delta_1 \lla \omega \rra^{-1}}^1 \sup_{\rho\in \lf(\delta_1 \lla \omega \rra^{-1},1 \rt)} |K(\rho,s;\lambda)| ds \lesssim \lla \omega \rra^{-1},
$$
thus the Volterra \eqref{eqn:ODEh1} has a solution $h_1$ that satisfies
\begin{equation}\label{eqn:h1bound}
|h_1(\rho;\lambda) - 1| \lesssim \int_{\rho}^1 |K(\rho,s;\lambda)|  ds \lesssim (1-\rho) \lla \omega \rra ^{-1}
\end{equation}
for all $\rho\in \lf[\delta_1 \lla \omega \rra^{-1},1 \rt]$. Re-inserting this into the equation we find
$$
h_1(\rho;\lambda) = 1 + \int_\rho^1 K(\rho, s; \lambda) ds + O((1-\rho)^2\lla \omega \rra ^{-2}).
$$

\textbf{Form of the integral of the kernel}

Now we look at the term
\begin{align*}
& \int_\rho^1 K(\rho,s;\lambda) ds \\
 =& \frac{1}{W(\lambda)} \lf(  \int_\rho^1 \frac{V(s)}{s^2} \lf(4-s^2(1-2\lambda)^2 \rt) ds  - \lf(\frac{1-\rho}{1+\rho} \rt)^{\frac{1}{2}-\lambda}\frac{2+\rho(1-2\lambda)}{2+\rho(-1+2\lambda)} I(\rho;\lambda) \rt) \\
=: & \frac{1}{W(\lambda)} (A+B)
\end{align*}
where
$$
I(\rho;\lambda) = \int_\rho^1 \frac{V(s)}{s^2} \lf(\frac{1-s}{1+s} \rt)^{-\frac{1}{2}+\lambda}(2+s(-1+2\lambda))^2 ds.
$$

We first look at the term $A$,
$$
\int_\rho^1 \frac{V(s)}{s^2} \lf(4-s^2(1-2\lambda)^2 \rt) ds  = 4 \int_\rho^1  \frac{ V(s)}{s^2} - \int_\rho^1 V(s) ds (1-2\lambda)^2.
$$
Here we have
\begin{align*}
\lf| \int_\rho^1\frac{V(s)}{s^2} ds \rt| \leq & \|V\|_{L^{\infty}}   \int_\rho^1 \frac{1}{s^2} ds   =  \|V\|_{L^{\infty}}  \lf( - 1 + \rho^{-1}  \rt)  \lesssim \rho^{-1} (1-\rho) \lesssim (1-\rho) \lla \omega \rra,
\end{align*}
where the last inequality follows because $\rho \in \lf[ \delta_1 \lla \omega \rra^{-1} , 1 \rt]$. 

For the term $B$, we insert the identity
$$
\lf(\frac{1-s}{1+s}\rt)^{-\frac{1}{2}+\lambda} \frac{(2+s(-1+2\lambda))}{s^2} = -s \; \PD_s\lf(\frac{(1-s)^{\frac{1}{2}+\lambda}(1+s)^{\frac{3}{2}-\lambda}}{s^2} \rt)
$$
into $I(\rho;\lambda)$. Integrating by parts once, we get
\begin{align*}
I(\rho; \lambda) = & - \int_\rho^1 V(s)  (2s +s^2(-1+2\lambda)) \PD_s\lf(\frac{(1-s)^{\frac{1}{2}+\lambda}(1+s)^{\frac{3}{2}-\lambda}}{s^2} \rt) ds \\
= & V(\rho) (2+\rho(-1+2\lambda)) \frac{(1-\rho)^{\frac{1}{2}+\lambda}(1+\rho)^{\frac{3}{2}-\lambda}}{\rho}  \\
& + \int_\rho^1 \frac{(1-s)^{\frac{1}{2}+\lambda}(1+s)^{\frac{3}{2}-\lambda}}{s^2} \PD_s\lf[  V(s)(2s +s^2(-1+2\lambda))\rt] ds.
\end{align*}
So we have now
\begin{align*}
&  \lf(\frac{1-\rho}{1+\rho} \rt)^{\frac{1}{2}-\lambda}\frac{2+\rho(1-2\lambda)}{2+\rho(-1+2\lambda)}  I(\rho;\lambda) \\
= &  (1-\rho^2) \lf( \frac{2V(\rho)}{\rho} + V(\rho) (1-2\lambda) \rt)\\
&  +   \lf(\frac{1-\rho}{1+\rho} \rt)^{\frac{1}{2}-\lambda}\frac{2+\rho(1-2\lambda)}{2+\rho(-1+2\lambda)}  \int_\rho^1 \frac{(1-s)^{\frac{1}{2}+\lambda}(1+s)^{\frac{3}{2}-\lambda}}{s^2} \PD_s\lf[ V(s)(2s+s^2 (-1+2\lambda)) \rt] ds .
\end{align*}
We estimate the boundary term 
$$
\lf| (1-\rho^2) \lf( \frac{2V(\rho)}{\rho} + V(\rho) (1-2\lambda) \rt) \rt| \lesssim  \rho^{-1} (1-\rho)   +  (1-\rho) \lla \omega \rra  \lesssim (1-\rho) \lla \omega \rra .
$$
For remaining term, notice that for $\Re \lambda \in [0,\frac{4}{1}]$,
$$
\lf(\frac{1-\rho}{1+\rho} \rt)^{\frac{1}{2}-\lambda}\frac{2+\rho(1-2\lambda)}{2+\rho(-1+2\lambda)} 
$$
is bounded, so we only need to look at the integral. When the derivative falls on $V$, the whole term is of $O((1-\rho)\lla \omega \rra)$ which is already good, thus we only need to consider when the derivative falls on the rest of the term. We look at
\begin{align*}
& \int_\rho^1 \frac{(1-s)^{\frac{1}{2}+\lambda}(1+s)^{\frac{3}{2}-\lambda}}{s^2} V(s) \PD_s  \lf[  (2s +s^2 (-1+2\lambda)) \rt] ds \\
 = &  \int_\rho^1 \lf( \frac{1-s}{1+s} \rt)^{\frac{1}{2}+\lambda} (1+s)^2 \lf( \frac{2 V(s)}{s^2} + \frac{2 V(s)}{s} (-1+2\lambda) \rt) ds \\
 = &  2\int_\rho^1 \lf( \frac{1-s}{1+s} \rt)^{\frac{1}{2}+\lambda} (1+s)^2 \frac{V(s)}{s^2} ds  +  (-2+4\lambda) \int_\rho^1 \lf( \frac{1-s}{1+s} \rt)^{\frac{1}{2}+\lambda} (1+s)^2  \frac{V(s)}{s} ds .
\end{align*}
In the first term, similar to before, we have, for $\Re \lambda \in [0, \frac{1}{4}]$, 
\begin{align*}
 \lf| \int_\rho^1 \lf( \frac{1-s}{1+s} \rt)^{\frac{1}{2}+\lambda} (1+s)^2 \frac{V(s)}{s^2} ds  \rt| \leq &  \int_\rho^1 \lf( \frac{1-s}{1+s} \rt)^{\frac{1}{2}+\Re \lambda} (1+s)^2 \lf| \frac{V(s)}{s^2} \rt|  ds \\
\lesssim &  \|V\|_{L^{\infty}} \int_\rho^1  \frac{1}{s^2}  ds \\
 \lesssim & \rho^{-1}(1-\rho) \lesssim (1-\rho) \lla \omega \rra.
\end{align*}
For the next part we use
$$
\lf(\frac{1-s}{1+s} \rt)^{\frac{1}{2}+\lambda} = -\frac{(1+s)^2}{3+2\lambda} \PD_s\lf( \frac{1-s}{1+s}\rt)^{\frac{3}{2}+\lambda},
$$
and integrate by parts again. Then we have
\begin{align*}
& (-2+4\lambda) \int_\rho^1 \lf( \frac{1-s}{1+s} \rt)^{\frac{1}{2}+\lambda} (1+s)^2  \frac{V(s)}{s} ds \\
 = & \frac{1-2\lambda}{3+2\lambda} \int_\rho^1 \frac{V(s)(1+s)^4}{s} \PD_s \lf( \frac{1-s}{1+s}\rt)^{\frac{3}{2}+\lambda} ds \\
 = & \frac{1-2\lambda}{3+2\lambda}  \lf(  \frac{V(\rho)}{\rho}(1+\rho)^4 \lf( \frac{1-\rho}{1+\rho} \rt)^{\frac{3}{2}+\lambda} -  \int_\rho^1 \PD_s \lf[ \frac{V(s)(1+s)^4}{s}  \rt] \lf( \frac{1-s}{1+s}\rt)^{\frac{3}{2}+\lambda} \rt).
\end{align*}
Now 
$$
\lf| \frac{1-2\lambda}{3+2\lambda} \frac{V(\rho)}{\rho}(1+\rho)^4 \lf( \frac{1-\rho}{1+\rho} \rt)^{\frac{3}{2}+\lambda}  \rt| \lesssim \rho^{-1} (1-\rho)  \lesssim (1-\rho) \lla \omega \rra.
$$
And for the integral, as before, it is only better when derivative falls on $V$, so we only need to compute
\begin{align*}
\lf| \int_\rho^1 \PD_s \lf[ \frac{(1+s)^4}{s}  \rt] V(s) \lf( \frac{1-s}{1+s}\rt)^{\frac{3}{2}+\lambda} ds \rt| \leq &  \int_\rho^1  \frac{\lf| (3s-1)(1+s)^3 \rt|}{s^2} \lf| V(s) \rt| \lf( \frac{1-s}{1+s}\rt)^{\frac{3}{2}+\Re\lambda} ds \\
\lesssim & \int_\rho^1 \frac{1}{s^2} ds \lesssim \rho^{-1} (1-\rho) \lesssim (1-\rho)\lla \omega\rra  .
\end{align*}
Hence we obtain
$$
\lf(\frac{1-\rho}{1+\rho} \rt)^{\frac{1}{2}-\lambda}\frac{2+\rho(1-2\lambda)}{2+\rho(-1+2\lambda)}  I(\rho;\lambda) = O((1-\rho)\lla \omega\rra  ).
$$

Putting everything back together we see that
$$
W(\lambda)\int_\rho^1 K(\rho,s;\lambda) ds = (1-2\lambda)^2\int_\rho^1 V(s) ds + O((1-\rho) \lla \omega \rra ).
$$
Finally, observe that we have
$$
\frac{(1-2\lambda)^2}{W(\lambda)} = \frac{1-2\lambda}{(3-2\lambda)(1+2\lambda)} = \mathcal{O}_o(\lla \omega \rra^{-1}) + \mathcal{O}(\lla \omega \rra^{-2})
$$
as stated. The second solution is given by $\tilde{v}_1(\rho;\lambda) := v_1(\rho; 1-\lambda)$, so again we have
$$
\frac{1-2(1-\lambda)}{(3-2(1-\lambda))(1+2(1-\lambda))} = \frac{-1+2\lambda}{(1+2\lambda)(3-2\lambda)} =  \mathcal{O}_o(\lla \omega \rra^{-1}) + \mathcal{O}(\lla \omega \rra^{-2}).
$$ 

\textbf{Estimating the derivatives}

The fact that $h_1$ is indeed of symbol type is shown in Appendix \ref{app:volterra}.
\end{proof}

Since $\delta_1$ is arbitrary, we can now choose $\delta_1 \leq \delta_0$. Then we can use the fundamental system $ \{ v_0, \tilde{v}_0 \}$ to construct $v_1$ and $\tilde{v}_1$ on the interval $\rho\in(0,\delta_0 \lla \omega \rra^{-1} ]$. These happen to be perturbations of $\psi_1$ and $\tilde{\psi}_1$.

\begin{lemma}\label{lemma:pertsol12} The solutions $v_1, \tilde{v}_1$ can be extended to $\rho\in(0,\delta_0 \lla \omega \rra^{-1}]$, and are of the form 
\begin{align*}
v_1(\rho;\lambda) &= \psi_1(\rho;\lambda)  [1  +  \mathcal{O}_o(\lla \omega \rra^{-1}) +  \mathcal{O} (\rho^0  \lla \omega\rra^{-2}) ], \\ 
\tilde{v}_1(\rho;\lambda) &= \tilde{\psi}_1(\rho;\lambda) [ 1 +  \mathcal{O}_o(\lla \omega \rra^{-1}) +  \mathcal{O} (\rho^0  \lla \omega\rra^{-2})]
\end{align*}
for all $\rho\in(0,\delta_0 \lla \omega \rra^{-1}]$, $\omega \in \RR$, $\epsilon \in [0,\frac{1}{4}]$.
\end{lemma}
\begin{proof}
On $\rho\in(0,\delta_0\lla \omega\rra^{-1}]$, there exist $a(\lambda)$ and $b(\lambda)$ such that
$$
v_1(\rho;\lambda) = a(\lambda)v_0(\rho;\lambda) + b(\lambda)\tilde{v}_0(\rho;\lambda),
$$
where
\begin{align*}
a(\lambda) & = \frac{W(v_1(\cdot;\lambda),\tilde{v}_0(\cdot;\lambda))}{W(v_0(\cdot;\lambda),\tilde{v}_0(\cdot;\lambda))} = -W(v_1(\cdot;\lambda),\tilde{v}_0(\cdot;\lambda)) \\
b(\lambda) & = - \frac{W(v_1(\cdot;\lambda),v_0(\cdot;\lambda))}{W(v_0(\cdot;\lambda),\tilde{v}_0(\cdot;\lambda))} = W(v_1(\cdot;\lambda),v_0(\cdot;\lambda)).
\end{align*}
So far we only have valid expressions of $v_1$ and $\tilde{v}_1$ on $\rho\in[\delta_1 \lla \omega \rra ^{-1},1]$ from Lemma \ref{lemma:pertsol1}. Since we have chosen $\delta_1 \leq \delta_0$, these expressions are valid at the point $\rho=\delta_0\lla \omega \rra^{-1}$. In order to compute $W(v_1(\cdot;\lambda),\tilde{v}_0(\cdot;\lambda))$ and $W(v_1(\cdot;\lambda),v_0(\cdot;\lambda))$, we evaluate at $\rho = \delta_0 \lla \omega \rra ^{-1}$. Using Lemmas \ref{lemma:pertsol1} and \ref{lemma:pertsol02}, we have
\begin{align*}
& W(v_1(\cdot;\lambda),\tilde{v}_0(\cdot;\lambda)) \\
= & W(\psi_1(\cdot;\lambda), \tilde{\psi}_0(\cdot;\lambda))\lf[ 1 +  \mathcal{O} (\lla \omega \rra^{-2}) \rt]  \lf[1 +  \frac{1-2\lambda}{(3-2\lambda)(1+2\lambda)} a_1(\delta_0 \lla \omega \rra^{-1})+  \mathcal{O} (\lla \omega\rra^{-2}) \rt] \\
& + \tilde{\psi}_0( \delta_0 \lla \omega \rra ^{-1}; \lambda)\psi_1( \delta_0 \lla \omega \rra ^{-1}; \lambda) \mathcal{O} (\lla \omega \rra^{-1}) \\
= &  W(\psi_1(\cdot;\lambda), \tilde{\psi}_0(\cdot;\lambda)) \lf[1 +  \frac{1-2\lambda}{(3-2\lambda)(1+2\lambda)} a_1(\delta_0 \lla \omega \rra^{-1})+  \mathcal{O} (\lla \omega\rra^{-2}) \rt] + \mathcal{O} (\lla \omega\rra^{-2}),
\end{align*}
where the second line follows from $\tilde{\psi}_0( \delta_0 \lla \omega \rra ^{-1}; \lambda) = \mathcal{O}(\lla \omega \rra^{-2})$ and $\psi_1( \delta_0 \lla \omega \rra ^{-1}; \lambda) = \mathcal{O}(\lla \omega \rra)$. The function $a_1$ is from Lemma \ref{lemma:pertsol1}. Since $\psi_1$ and $\tilde{\psi}_0$ are explicitly defined, we can compute that
\begin{align*}
 W(\psi_1(\rho;\lambda), \tilde{\psi}_0(\rho;\lambda)) = & \psi_1(\rho;\lambda) \tilde{\psi}_0'(\rho;\lambda) - \psi_1'(\rho;\lambda)\tilde{\psi}_0(\rho;\lambda) \\
 = & \psi_1(\rho;\lambda) \psi_0'(\rho;\lambda) \int_\rho^{\delta_0 \lla \omega \rra ^{-1}}\psi_0(s;\lambda)^{-2} ds -\psi_1(\rho;\lambda) \psi_0 (\rho;\lambda)^{-1} \\ 
 & - \psi_1'(\rho;\lambda) \psi_0(\rho;\lambda) \int_\rho^{\delta_0 \lla \omega\rra ^{-1}}\psi_0(s;\lambda)^{-2} ds \\
 = & W(\psi_1(\rho;\lambda), \psi_0(\rho;\lambda)) \int_\rho^{\delta_0 \lla \omega\rra ^{-1}}\psi_0(s;\lambda)^{-2} ds - \psi_1(\rho;\lambda) \psi_0 (\rho;\lambda)^{-1} \\
 = & -W (\lambda) \int_\rho^{\delta_0 \lla \omega\rra ^{-1}}\psi_0(s;\lambda)^{-2} ds - \psi_1(\rho;\lambda) \psi_0 (\rho;\lambda)^{-1},
\end{align*}
which is a constant in $\rho$ by the structure of the ODE. Evaluating at $\rho = \delta_0 \lla \omega \rra^{-1}$, we see that
$$
 W(\psi_1(\cdot;\lambda), \tilde{\psi}_0(\cdot;\lambda)) = \psi_1(\delta_0 \lla \omega \rra^{-1};\lambda) \psi_0 (\delta_0 \lla \omega \rra^{-1};\lambda)^{-1} = \mathcal{O}(\lla \omega \rra^0).
$$
But for now we do not evaluate it, instead we use that the expression valid for all $\rho\in(0,\delta_0\lla \omega\rra^{-1}]$, not just at one point. So we obtain
\begin{align*}
a(\lambda) =& - W(v_1(\rho;\lambda),\tilde{v}_0(\rho;\lambda)) \\
= & \lf( W (\lambda) \int_\rho^{\delta_0 \lla \omega\rra ^{-1}}\psi_0(s;\lambda)^{-2} ds + \psi_1(\rho;\lambda) \psi_0 (\rho;\lambda)^{-1} \rt) \\
& \times   \lf[1 +  \frac{1-2\lambda}{(3-2\lambda)(1+2\lambda)} a_1(\delta_0 \lla \omega \rra^{-1})+  \mathcal{O} (\lla \omega\rra^{-2}) \rt]
\end{align*}
for $\rho\in(0,\delta_0\lla \omega\rra^{-1}]$. On the other hand, evaluating at $\rho = \delta_0 \lla \omega \rra ^{-1}$ we have
\begin{align*}
b(\lambda) = & W(v_1(\cdot;\lambda),v_0(\cdot;\lambda)) \\
= & W(\psi_1(\cdot;\lambda), \psi_0(\cdot;\lambda)) \lf[ 1 +  \mathcal{O} (\lla \omega \rra^{-2}) \rt]  \lf[1 +  \frac{1-2\lambda}{(3-2\lambda)(1+2\lambda)} a_1(\delta_0 \lla \omega \rra ^{-1})+  \mathcal{O} (\lla \omega\rra^{-2}) \rt] \\
& + \psi_0( \delta_0 \lla \omega \rra ^{-1}; \lambda)\psi_1( \delta_0 \lla \omega \rra ^{-1}; \lambda) \mathcal{O} (\lla \omega \rra^{-1}) \\
= & W(\psi_1(\cdot;\lambda), \psi_0(\cdot;\lambda))\lf[1 +  \frac{1-2\lambda}{(3-2\lambda)(1+2\lambda)} a_1(\delta_0 \lla \omega \rra ^{-1})+  \mathcal{O} (\lla \omega\rra^{-2}) \rt] \\
= &-W(\lambda) \lf[1 +  \frac{1-2\lambda}{(3-2\lambda)(1+2\lambda)} a_1(\delta_0 \lla \omega \rra ^{-1})+  \mathcal{O} (\lla \omega\rra^{-2}) \rt],
\end{align*}
where in the second line we used $\psi_0( \delta_0 \lla \omega \rra ^{-1}; \lambda) = \mathcal{O}(\lla \omega \rra)$, $\psi_1( \delta_0 \lla \omega \rra ^{-1}; \lambda) = \mathcal{O}(\lla \omega \rra)$, and $W(\lambda) = \mathcal{O}(\lla \omega \rra ^3)$. Hence
\begin{align*}
v_1(\rho;\lambda) = &  v_0(\rho;\lambda) \lf( W (\lambda) \int_\rho^{\delta_0 \lla \omega\rra ^{-1}}\psi_0(s;\lambda)^{-2} ds + \psi_1(\rho;\lambda) \psi_0 (\rho;\lambda)^{-1} \rt) \\
& \times  \lf[1 +  \frac{1-2\lambda}{(3-2\lambda)(1+2\lambda)} a_1(\delta_0 \lla \omega \rra^{-1})+  \mathcal{O} (\lla \omega\rra^{-2}) \rt] \\
&  -\tilde{v}_0(\rho;\lambda) W(\lambda) \lf[1 +  \frac{1-2\lambda}{(3-2\lambda)(1+2\lambda)} a_1(\delta_0 \lla \omega \rra ^{-1})+  \mathcal{O} (\lla \omega\rra^{-2}) \rt]  \\
 = & W(\lambda)\tilde{\psi}_0(\rho;\lambda)[1+ \mathcal{O}(\rho^0 \lla \omega \rra^{-2})] \\ 
 & + \psi_1(\rho;\lambda) [1+\mathcal{O}(\rho^2 \lla \omega\rra^0)] \lf[1 +  \frac{1-2\lambda}{(3-2\lambda)(1+2\lambda)} a_1(\delta_0 \lla \omega \rra^{-1})+  \mathcal{O} (\lla \omega\rra^{-2}) \rt]  \\
 & - W(\lambda) \tilde{\psi}_0(\rho;\lambda) [1+ \mathcal{O}(\rho^0 \lla \omega \rra^{-2})] \\ 
 = &  W(\lambda)\tilde{\psi}_0(\rho;\lambda) \lf[ \mathcal{O}(\rho^0 \lla \omega \rra^{-2} ) \rt] \\  
 & + \psi_1(\rho;\lambda)  \lf[1 + \frac{1-2\lambda}{(3-2\lambda)(1+2\lambda)} a_1(\delta_0 \lla \omega \rra^{-1})+  \mathcal{O} (\rho^0 \lla \omega\rra^{-2}) \rt] .
\end{align*}
Then we use $\tilde{\psi}_0(\rho;\lambda) = \mathcal{O}(\rho^{-1} \lla \omega \rra^{-3})$ and $\psi_1(\rho;\lambda)^{-1} =  \mathcal{O} (\rho \lla \omega\rra^0)$ for small $\rho$, and see that
\begin{align*}
\frac{v_1(\rho;\lambda)}{\psi_1(\rho;\lambda)}= & \frac{\tilde{\psi}_0(\rho;\lambda)}{\psi_1(\rho;\lambda)} \mathcal{O} (\lla \omega \rra^{3})  \lf[ \mathcal{O}(\rho^0 \lla \omega \rra^{-2} ) \rt] \\ 
&  + \lf[1 +  \frac{1-2\lambda}{(3-2\lambda)(1+2\lambda)} a_1(\delta_0 \lla \omega \rra ^{-1})+  \mathcal{O} (\rho^0 \lla \omega\rra^{-2}) \rt] \\
= &  1 +  \frac{1-2\lambda}{(3-2\lambda)(1+2\lambda)} a_1(\delta_0 \lla \omega \rra ^{-1})+  \mathcal{O} (\rho^0 \lla \omega\rra^{-2})\\
= & 1  +  \mathcal{O}_o(\lla \omega \rra^{-1}) +  \mathcal{O} (\rho^0 \lla \omega\rra^{-2})
\end{align*}
Here we have used that $a_1(\delta_0 \lla \omega \rra^{-1}) = \mathcal{O}(\lla \omega \rra^0)$ is an even function in $\omega$. The second solution is again given by $\tilde{v}_1(\rho;\lambda) = v_1(\rho; 1-\lambda)$ as in the last part of the proof of Lemma \ref{lemma:pertsol1}. 
\end{proof}

In the next step, we express $v_0$ in terms of the fundamental system $\{v_0, \tilde{v}_0 \}$ on the interval $\rho\in [\delta_1 \lla\omega\rra^{-1},1]$.

\begin{lemma}\label{lemma:pertsol0all} The solution $v_0$ has the representation
$$
v_0(\rho;\lambda)  = [1+ \mathcal{O}_o(\lla \omega\rra^{-1})+ \mathcal{O} (\lla \omega\rra^{-2})]v_1(\rho;\lambda) - [1+ \mathcal{O}_o (\lla \omega\rra^{-1}) + \mathcal{O}(\lla \omega\rra^{-2})]\tilde{v}_1(\rho;\lambda)
$$
for all $\rho\in[\delta_1 \lla \omega \rra ^{-1},1]$, $\omega \in \RR$, $\epsilon \in [0,\frac{1}{4}]$.
\end{lemma}

\begin{proof}
Since the Wronskian is constant, we can evaluate at $\rho = 1$ and get
\begin{align*}
& W(v_1(\cdot;\lambda), \tilde{v}_1(\cdot;\lambda)) \\
= & W(\psi_1(\cdot;\lambda), \tilde{\psi}_1(\cdot;\lambda)) \lf[1+ \mathcal{O}( \rho^0 (1-\rho)\lla \omega \rra^{-1}  ) \rt] + \psi_1(1;\lambda) \tilde{\psi}_1(1;\lambda)  \mathcal{O}(\rho^{-1} (1-\rho)^0 \lla \omega \rra^{-1}) \\
= & W(\psi_1(\cdot;\lambda), \tilde{\psi}_1(\cdot;\lambda)) = W(\lambda).
\end{align*}
So there exist $a(\lambda)$ and $b(\lambda)$ such that
$$
v_0 (\rho; \lambda) = a(\lambda) v_1 (\rho; \lambda) + b(\lambda) \tilde{v}_1 (\rho; \lambda),
$$
given by
\begin{align*}
a(\lambda) &= \frac{W(v_0(\cdot;\lambda), \tilde{v}_1(\cdot;\lambda))}{W(v_1(\cdot;\lambda), \tilde{v}_1(\cdot;\lambda))} = \frac{W(v_0(\cdot;\lambda), \tilde{v}_1(\cdot;\lambda))}{W(\lambda)} ,\\
b(\lambda) & = - \frac{W(v_0(\cdot;\lambda), v_1(\cdot;\lambda))}{W(v_1(\cdot;\lambda), \tilde{v}_1(\cdot;\lambda))} =  - \frac{W(v_0(\cdot;\lambda), v_1(\cdot;\lambda))}{W(\lambda)} .
\end{align*}
By using Lemma \ref{lemma:pertsol1} and evaluating at $\rho = \delta_0 \lla \omega \rra^{-1}$, we have
\begin{align*}
W(v_0(\cdot;\lambda), \tilde{v}_1(\cdot;\lambda)) = & W(\psi_0(\cdot;\lambda), \tilde{\psi}_1(\cdot;\lambda)) \lf[ 1 +   \mathcal{O}_o(\lla \omega \rra^{-1}) a_1(\delta_0 \lla \omega \rra ^{-1})+   \mathcal{O} (\lla \omega\rra^{-2}) \rt] \\
& + \tilde{\psi}_1(\delta_0 \lla \omega \rra^{-1};\lambda) \psi_0 (\delta_0 \lla \omega \rra^{-1};\lambda) \mathcal{O}(\lla \omega\rra^{-1}) \\
= & W(\lambda)   [1+ \mathcal{O}_o(\lla \omega\rra^{-1})+ \mathcal{O} (\lla \omega\rra^{-2})],
\end{align*}
hence
$$
a(\lambda) =  [1+ \mathcal{O}_o(\lla \omega\rra^{-1})+ \mathcal{O} (\lla \omega\rra^{-2})].
$$
Similarly, 
$$
W(v_0(\cdot;\lambda),v_1(\cdot;\lambda)) = W(\lambda)   [1+ \mathcal{O}_o(\lla \omega\rra^{-1})+ \mathcal{O} (\lla \omega\rra^{-2})],
$$
so
$$
b(\lambda) =  - [1+ \mathcal{O}_o(\lla \omega\rra^{-1})+ \mathcal{O}(\lla \omega\rra^{-2})].
$$
\end{proof}

\subsection{The Wronskian}

The transformation
$$
u_j(\rho) = \rho^{-2}(1-\rho^2)^{-\frac{1}{4}-\frac{\lambda}{2}}v_j(\rho)
$$
now gives solutions $u_j, j=0,1$ to \eqref{eqn:ODEpotential} with $F_\lambda = 0$. Near $\rho = 0$ we have
$$
\lim_{\rho \to 0^+} u_0(\rho;\lambda) = \lim_{\rho \to 0^+} \frac{v_0(\rho;\lambda)}{\rho^2} =  \lim_{\rho \to 0^+} \frac{ \psi_0(\rho;\lambda)}{\rho^2} = -\frac{1}{6} W(\lambda) = O(\lla \lambda \rra^3),
$$
and we see that $u_0 (\cdot ;\lambda ) \in C([0,1)) \cap C^\infty(0,1)$. Taking its derivative with respect to $\rho$, we have, for $\rho \in [0,\delta_0 \lla \omega \rra^{-1}], \lambda = \epsilon + i \omega$,
$$
|u_0'(\rho ; \lambda) | \lesssim \rho^{-1} \lla \omega \rra ^3.
$$
Hence $u_0(\cdot;\lambda) \in H^1(\BB_{1-\delta}^5)$ for all $\delta\in (0,1)$. On the other hand, near $\rho = 1$, from the transformation we have
\begin{align*}
& \lim_{\rho \to 1^-} u_1(\rho;\lambda) \\
= & \lim_{\rho \to 1^-} \lf[\rho^{-3}(1+\rho)^{\frac{1}{2} - \lambda} (2+\rho(-1+2\lambda)) \lf[ 1+  \mathcal{O}_0(\lla \omega \rra^{-1}) a_1(\rho) + \mathcal{O} (\rho^0 (1-\rho) \lla \omega \rra^{-2}) \rt] \rt] \\
= & 2^{\frac{1}{2} - \lambda} (1+2\lambda)  = O(\lla \lambda \rra ),
\end{align*}
and its derivative is 
$$
\lim_{\rho \to 1^-} u_1'(\rho;\lambda)  = O(\lla \lambda \rra^2).
$$
Hence $u_1(\cdot;\lambda) \in C^1(0,1] \cap C^\infty(0,1)$, so $u_1 \in H^1 (\BB^5 \backslash \{0\})$.

Another two solutions to \eqref{eqn:ODEpotential} are given by, for $j=0,1$,
$$
\tilde{u}_j(\rho) = \rho^{-2}(1-\rho^2)^{-\frac{1}{4}-\frac{\lambda}{2}}\tilde{v}_j(\rho).
$$

The next step is to construct the Wronskian of the fundamental system $\{ u_0 , u_1\}$. Here we turn our attention to the specific potential $V(\rho) = - \frac{35}{4}$, i.e., we are looking at \eqref{eqn:ODElambdaV}.

\begin{lemma}\label{lemma:wronskian}
We have
$$
W(u_0(\cdot;\lambda), u_1(\cdot;\lambda)) = W(\lambda) w_0(\lambda) \rho^{-4}(1-\rho^2)^{-\frac{1}{2}-\lambda}
$$
where $w_0(\epsilon + i \omega ) = 1+\mathcal{O}(\lla \omega \rra ^{-1} ) + \mathcal{O}(\lla \omega \rra ^{-2})$ for $\rho\in(0,1)$, $\epsilon \in [0,\frac{1}{4}]$, $\omega \in \RR$. Moreover, $|w_0(\lambda)| \gtrsim 1$ for all $\lambda \in \mathbb{C}$ with $\Re \lambda \in [0,\frac{1}{4}]$.
\end{lemma}

\begin{proof}
From the transformation we compute directly
\begin{align*}
W(u_0(\cdot;\lambda), u_1(\cdot;\lambda)) & = W(v_0(\cdot;\lambda),v_1(\cdot;\lambda))\rho^{-4}(1-\rho^2)^{-\frac{1}{2}-\lambda} \\
& = W(\lambda) \lf[1+\mathcal{O}_o(\lla \omega \rra ^{-1} ) + \mathcal{O}(\lla \omega \rra ^{-2} ) \rt] \rho^{-4}(1-\rho^2)^{-\frac{1}{2}-\lambda},
\end{align*}
which is in the stated form.

Next we show that $|w_0 (\lambda)| \neq 0$. Since 
$$
w_0 (\lambda) = \frac{ W(v_0(\cdot;\lambda),v_1(\cdot;\lambda))}{W(\lambda)},
$$
so $w_0 \neq 0$ if and only if $W(v_0(\cdot;\lambda),v_1(\cdot;\lambda)) \neq 0$. Moreover, since 
$$
W(v_0(\cdot;\lambda),v_1(\cdot;\lambda)) = \rho^4 (1-\rho^2)^{\frac{1}{2} + \lambda} W(u_0(\cdot;\lambda),u_1(\cdot;\lambda)),
$$
$W(v_0(\cdot;\lambda),v_1(\cdot;\lambda)) \neq 0$ if and only if $W(u_0(\rho;\lambda),u_1(\rho;\lambda)) \neq 0$ for $\rho \in (0,1)$. So we define $h(\rho^2) = u(\rho)$ and set $z=\rho^2$, then \eqref{eqn:ODEpotential} with $F_\lambda = 0$ is transformed to the hypergeometric differential equation \eqref{eqn:ODEhgde1}. This equation was studied in Lemma \ref{lemma:specL}. By comparing the asymptotics as $\rho \to 0^+ $ we see that
\begin{align*}
u_0(\rho;\lambda) & =  h_0(\rho^2; \lambda)O(\lla \omega \rra^3), \\
\tilde{u}_0(\rho;\lambda) &= \tilde{h}_0(\rho^2; \lambda) O(\lla \omega \rra^{-3}).
\end{align*}
and
$$
u_1(\rho;\lambda) = h_1(\rho^2;\lambda).
$$
So the Wronskian $W(u_0(\rho;\lambda), u_1(\rho;\lambda)) \neq 0$ if and only if $h_0(\cdot;\lambda)$ and $h_1(\cdot;\lambda)$ are linearly independent. From \eqref{eqn:connform}, $h_0(\cdot;\lambda)$ and $h_1(\cdot;\lambda)$ are linearly dependent only when
$$
-a = - \frac{\lambda}{2} - \frac{5}{2}  \in \mathbb{N}_0 \text{ or } -b = -\frac{\lambda}{2}+\frac{1}{2} \in \mathbb{N}_0.
$$
But this never happens for $\lambda = \epsilon + i \omega$, $\epsilon \in [0,\frac{1}{4}]$. Therefore we obtain
$$
|w_0 (\lambda)| =  \lf|\frac{ W(v_0(\cdot;\lambda),v_1(\cdot;\lambda))}{W(\lambda)} \rt| \gtrsim 1
$$
as stated.
\end{proof}

This shows that $\{u_0(\cdot;\lambda, u_1(\cdot;\lambda) \}$ is a fundamental system for \eqref{eqn:ODElambdaV} with $F_\lambda = 0$.

\begin{cor}\label{cor:w0} We have
$$
\frac{1}{w_0(\lambda)} = 1+ \mathcal{O}_o \lf(\lla \omega \rra^{-1} \rt) + \mathcal{O} \lf(\lla \omega \rra^{-2} \rt)
$$
for the function $w_0$ from Lemma \ref{lemma:wronskian}, where $\lambda = \epsilon+i \omega$, $\epsilon \in [0,\frac{1}{4}]$, $\omega \in \RR$.
\end{cor}

\begin{proof}
We write
\begin{align*}
\frac{1}{w_0(\lambda)} & = \frac{\overline{w_0(\lambda)}}{|w_0(\lambda)|^2} \\
& = \lf[1+ \mathcal{O}_o \lf(\lla \omega \rra^{-1} \rt) + \mathcal{O} \lf(\lla \omega \rra^{-2} \rt) \rt] \lf[1+  \mathcal{O} \lf(\lla \omega \rra^{-1} \rt) \rt] \\
& = 1+ \mathcal{O}_o \lf(\lla \omega \rra^{-1} \rt) + O \lf(\lla \omega \rra^{-2} \rt).
\end{align*}
\end{proof}

\subsection{Decomposition of the Green function}

We first set
\begin{align*}
\varphi_1(\rho;\lambda)  & := \rho^{-2}(1-\rho^2)^{-\frac{1}{4}-\frac{\lambda}{2}}\psi_1(\rho;\lambda)=\rho^{-3}(1+\rho)^{\frac{1}{2}-\lambda}(2+\rho(-1+2\lambda)) , \\
\tilde{\varphi}_1 (\rho;\lambda)  & := \rho^{-2}(1-\rho^2)^{-\frac{1}{4}-\frac{\lambda}{2}}\tilde{\psi}_1(\rho;\lambda)=\rho^{-3}(1-\rho)^{\frac{1}{2}-\lambda}(2+\rho(1 - 2\lambda)),
\end{align*}
and
$$
\varphi_0 (\rho;\lambda):= \varphi_1(\rho;\lambda)  - \tilde{\varphi}_1 (\rho;\lambda). 
$$
Then $\{\varphi_0(\cdot;\lambda), \varphi_1(\cdot;\lambda) \}$ is a fundamental system of \eqref{eqn:ODEpotential} with $F_\lambda=V=0$. The Green function is given by
$$
G_0(\rho,s;\lambda) = \frac{s^4(1-s^2)^{-\frac{1}{2}+\lambda}}{(3-2\lambda)(1+2\lambda)(1-2\lambda)} \begin{cases} \varphi_0(\rho;\lambda) \varphi_1(s;\lambda) \text{ if }   \rho \leq s \\ \varphi_0(s;\lambda) \varphi_1(\rho;\lambda) \text{ if } \rho \geq s \end{cases},
$$
which is the Green function that gives the semigroup $\mathbf{S}_0$ using Laplace inversion.

General solutions to \eqref{eqn:ODElambdaV} are given by
\begin{align*}
u(\rho;\lambda) = & c_0(\lambda) u_0(\rho;\lambda) + c_1(\lambda) u_1 (\rho;\lambda) \\
& - \int_0^\rho \frac{u_0(s;\lambda) u_1 (\rho;\lambda)}{W(u_0(\cdot;\lambda),u_1(\cdot;\lambda))(s)} \frac{F_\lambda (s)}{1-s^2} ds -\int_\rho^1 \frac{u_0(\rho;\lambda) u_1 (s;\lambda)}{W(u_0(\cdot;\lambda),u_1(\cdot;\lambda))(s)} \frac{F_\lambda (s)}{1-s^2} ds ,
\end{align*}
where $c_0(\lambda),c_1(\lambda)$ are arbitrary constants. But since neither $u_0(\cdot;\lambda)$ nor $u_1(\cdot;\lambda)$ belongs to $H^1 (\BB^5)$, we choose $c_0$ and $c_1$ to be identically 0. Then there is a unique solution in  $H^1 (\BB^5)$ given by	
$$
u(\rho;\lambda) =  - \int_0^\rho \frac{u_0(s;\lambda) u_1 (\rho;\lambda)}{W(u_0(\cdot;\lambda),u_1(\cdot;\lambda))(s)} \frac{F_\lambda (s)}{1-s^2} ds -\int_\rho^1 \frac{u_0(\rho;\lambda) u_1 (s;\lambda)}{W(u_0(\cdot;\lambda),u_1(\cdot;\lambda))(s)} \frac{F_\lambda (s)}{1-s^2} ds.
$$
So the Green function for \eqref{eqn:ODElambdaV} is given by
$$
G(\rho,s;\lambda) = \frac{s^4(1-s^2)^{-\frac{1}{2}+\lambda}}{(3-2\lambda)(1+2\lambda)(1-2\lambda)w_0(\lambda)} \begin{cases} u_0(\rho;\lambda) u_1(s;\lambda) \text{ if }   \rho \leq s \\ u_1(\rho;\lambda) u_0(s;\lambda)  \text{ if } \rho \geq s  \end{cases} ,
$$
which is the Green function in the integral kernel of \eqref{eqn:laplaceinv}.

We define a smooth cut-off $\chi: \RR \to [0,1]$ with $\chi(x)=1$ for $|x| \leq \delta_1$ and $\chi(x)=0$ for $|x| \geq \delta_0$, where $\delta_0$ and $\delta_1$ are chosen from before. We decompose $G$ in the following way.

\begin{lemma}\label{lemma:decomposeG} We have
$$
G(\rho,s ;\lambda) = G_0(\rho, s; \lambda) + \sum_{n=1}^6 G_n(\rho,s;\lambda),
$$
where
\begin{align*}
G_1 (\rho,s;\lambda)& = 1_{\RR_+}(s-\rho)\chi(\rho\lla \omega \rra)s^4(1-s^2)^{-\frac{1}{2}+\lambda}\frac{\varphi_0(\rho;\lambda) \varphi_1(s;\lambda) \gamma_1(\rho,s;\lambda) }{(3-2\lambda)(1+2\lambda)(1-2\lambda)}  \\
G_2 (\rho,s;\lambda)& = 1_{\RR_+}(s-\rho)\lf[1-\chi(\rho\lla \omega \rra) \rt] s^4(1-s^2)^{-\frac{1}{2}+\lambda}\frac{ \varphi_1(\rho;\lambda) \varphi_1(s;\lambda) \gamma_2(\rho,s;\lambda)}{(3-2\lambda)(1+2\lambda)(1-2\lambda)}\\
G_3 (\rho,s;\lambda)& = 1_{\RR_+}(s-\rho)\lf[ 1-\chi(\rho\lla \omega \rra) \rt]s^4(1-s^2)^{-\frac{1}{2}+\lambda}\frac{\tilde{\varphi}_1(\rho;\lambda) \varphi_1(s;\lambda) \gamma_3(\rho,s;\lambda)}{(3-2\lambda)(1+2\lambda)(1-2\lambda)}\\
G_4 (\rho,s;\lambda)& = 1_{\RR_+}(\rho-s)\chi(s\lla \omega \rra) s^4(1-s^2)^{-\frac{1}{2}+\lambda} \frac{\varphi_1(\rho;\lambda)\varphi_0(s;\lambda)  \gamma_4(\rho,s;\lambda)}{(3-2\lambda)(1+2\lambda)(1-2\lambda)} \\
G_5(\rho,s;\lambda) & = 1_{\RR_+}(\rho-s)\lf[ 1- \chi(s\lla \omega \rra) \rt] s^4(1-s^2)^{-\frac{1}{2}+\lambda}\frac{\varphi_1(\rho;\lambda)\varphi_1(s;\lambda)  \gamma_5(\rho,s;\lambda)}{(3-2\lambda)(1+2\lambda)(1-2\lambda)}  \\
G_6(\rho,s;\lambda) & = 1_{\RR_+}(\rho-s)\lf[ 1- \chi(s\lla \omega \rra) \rt] s^4(1-s^2)^{-\frac{1}{2}+\lambda}\frac{\varphi_1(\rho;\lambda)\tilde{\varphi}_1(s;\lambda)  \gamma_6(\rho,s;\lambda)}{(3-2\lambda)(1+2\lambda)(1-2\lambda)} ,
\end{align*}
and
 $$ \gamma_n(\rho,s;\lambda) = \mathcal{O}(\rho^0 s^0) O_0( \lla \omega \rra^{-1} ) +  \mathcal{O}((1-\rho)^0 s^0 \lla \omega \rra^{-2}) +  \mathcal{O}(\rho^0 (1-s) \lla \omega \rra ^{-2}), $$
for all $\rho, s\in (0,1)$, $\lambda = \epsilon+ i \omega$, $\epsilon\in [0,\frac{1}{4}]$, $\omega \in \RR$, $n = \{1,2, \ldots, 6\}$.
\end{lemma}
\begin{proof} From Lemmas \ref{lemma:pertsol1} and \ref{lemma:pertsol12}, we have
\begin{align}
u_1(\rho;\lambda)= & \chi(\rho\lla \omega \rra) u_1(\rho;\lambda) + [1-\chi(\rho \lla \omega \rra)] u_1(\rho;\lambda) \nonumber \\
= & \chi(\rho\lla \omega \rra)\varphi_1(\rho;\lambda) [1+ \mathcal{O}_o(\lla \omega \rra^{-1}) +  \mathcal{O}(\rho^0\lla \omega \rra^{-2})] \nonumber \\
& + [1-\chi(\rho \lla \omega \rra)] \varphi_1(\rho;\lambda)[1+ \mathcal{O}_o(\lla \omega \rra^{-1})a_1(\rho) +  \mathcal{O}(\rho^0 (1-\rho) \lla \omega \rra^{-2})] \nonumber \\
= & \varphi_1 (\rho;\lambda)[1+  \mathcal{O}(\rho^0)\mathcal{O}_o(\lla \omega\rra^{-1}) +  \mathcal{O} (\rho^0(1-\rho)\lla \omega \rra^{-2})]. \label{eqn:u1pert}
\end{align}

On the other hand, using Lemmas \ref{lemma:pertsol0}, \ref{lemma:pertsol1}, and \ref{lemma:pertsol0all}, we have
\begin{align}
u_0(\rho;\lambda) = &  \chi(\rho \lla \omega \rra) u_0 (\rho;\lambda) + \lf[ 1-  \chi(\rho \lla \omega \rra) \rt] u_0(\rho;\lambda) \nonumber\\
= &  \chi(\rho \lla \omega \rra) \varphi_0(\rho;\lambda) \lf[ 1+  \mathcal{O}(\rho^0 \lla \omega \rra^{-2}) \rt] \nonumber\\
& + \lf[ 1-  \chi(\rho \lla \omega \rra) \rt] u_1(\rho;\lambda) \lf[ 1+  \mathcal{O}_o(\lla \omega \rra^{-1}) +  \mathcal{O}(\lla \omega \rra^{-2}) \rt] \nonumber \\
& -\lf[ 1-  \chi(\rho \lla \omega \rra) \rt] \tilde{u}_1(\rho;\lambda) \lf[ 1+  \mathcal{O}_o(\lla \omega \rra^{-1}) +  \mathcal{O}(\lla \omega \rra^{-2}) \rt] \nonumber  \\
= &  \chi(\rho \lla \omega \rra) \varphi_0(\rho;\lambda) \lf[ 1+  \mathcal{O}(\rho^0 \lla \omega \rra^{-2}) \rt] \nonumber \\
& + \lf[ 1-  \chi(\rho \lla \omega \rra) \rt] \varphi_1(\rho;\lambda)\lf[ 1+ \mathcal{O}(\rho^0)\mathcal{O}_o(\lla \omega \rra^{-1}) +  \mathcal{O}(\rho^0(1-\rho)\lla \omega \rra^{-2}) + \mathcal{O}(\lla \omega \rra^{-2}) \rt] \nonumber \\
& - \lf[ 1-  \chi(\rho \lla \omega \rra) \rt] \tilde{\varphi}_1(\rho;\lambda)\lf[ 1+ \mathcal{O}(\rho^0)\mathcal{O}_o(\lla \omega \rra^{-1}) +  \mathcal{O}(\rho^0(1-\rho)\lla \omega \rra^{-2}) +  \mathcal{O}(\lla \omega \rra^{-2}) \rt] \label{eqn:u0pert}.
\end{align}
Then we obtain the stated decomposition from Corollary \ref{cor:w0}.
\end{proof}

\subsection{The semigroup}

Using Lemma \ref{lemma:decomposeG}, the representation of the semigroup in \eqref{eqn:laplaceinv} is given by
\begin{equation}\label{eqn:semigpdecomp}
[ \mathbf{S}(\tau) \tilde{\mathbf{f}} ]_1 (\rho) = [ \mathbf{S}_0(\tau) \tilde{\mathbf{f}} ]_1 (\rho)  + \frac{1}{2\pi i} + \sum_{n=1}^6 \lim_{N\to\infty} \int_{\epsilon- i N}^{\epsilon + i N} e^{\lambda \tau} \int_0^1 G_n(\rho ,s ; \lambda) F_\lambda(s) ds d\lambda 
\end{equation}
for any $\epsilon> 0$, $\tilde{\mathbf{f}} = (\tilde{f}_1, \tilde{f}_2) \in \rg (\mathbf{I} - \mathbf{P}) \cap C^2 \times C^1 ([0,1])$, and
$$
F_\lambda (s) = s \tilde{f}'_1(s) + \lf(\lambda+\frac{5}{2} \rt)\tilde{f}_1(s) + \tilde{f}_2(s).
$$
So we define, for $n\in \{1, 2, \ldots,6\}$, $f \in C([0,1])$, the operators
\begin{align*}
T_{n,\epsilon} (\tau) f(\rho)  := & \frac{1}{2\pi i} \lim_{N \to \infty} \int_{\epsilon- i N}^{\epsilon + i N} e^{\lambda \tau} \int_0^1 G_n(\rho ,s ; \lambda) f(s) ds d\lambda  \\
= & \frac{1}{2\pi }  \int_{-\infty}^{\infty} e^{(\epsilon+i \omega) \tau} \int_0^1 G_n(\rho ,s ; \epsilon+i \omega) f(s) ds d\omega
\end{align*}
for $\tau \geq 0$ and $\rho \in (0,1)$. We would like to take the limit $\epsilon \to 0^+$. Indeed, for all $\rho \in (0,1]$, $\epsilon \in [0,\frac{1}{4}]$, and $\omega \in \RR$ we have $|\varphi_1(\rho; \epsilon+i\omega)|+| \tilde{\varphi}_1(\rho; \epsilon+i\omega)| \lesssim \rho^{-3} + \rho^{-2}\lla \omega \rra $, so $|G_n(\rho,s; \epsilon+ i \omega) | \lesssim \rho^{-3} s(1-s)^{-\frac{1}{2}+\epsilon} \lla \omega \rra^{-2}$. By dominated convergence and Fubini, we have
\begin{align*}
T_n(\tau) f(\rho) :=  \lim_{\epsilon \to 0+} T_{n,\epsilon} (\tau) f(\rho) & = \frac{1}{2\pi } \int_{-\infty}^{\infty} e^{i \omega \tau} \int_0^1 G_n(\rho ,s ; \omega) f(s) ds d\omega \\
& = \frac{1}{2\pi} \int_0^1 \int_{-\infty}^{\infty} e^{i \omega \tau} G_n(\rho ,s ; \omega) d\omega f(s) ds .
\end{align*}

\subsection{Representations of $\varphi_0$}

In this section we apply the fundamental theorem of calculus to $\varphi_0(\rho; \lambda)\rho^3$ to obtain two ways of writing this term. These representations will be useful when we estimate the oscillatory integrals. Since $\lim_{\rho \to 0} \varphi_0(\rho)\rho^3=0$, we have
\begin{align*}
\varphi_0(\rho; \lambda)\rho^3 & = -\frac{1}{2}(3-2\lambda) (1-2\lambda) \int_0^\rho \lf( (1+t_1)^{-\frac{1}{2}-\lambda} - (1-t)^{-\frac{1}{2}-\lambda} \rt)t_1  dt_1 \\
& =  -\frac{\rho^2}{2}(3-2\lambda) (1-2\lambda)\int_{-1}^1  (1+ \rho t_1)^{-\frac{1}{2}-\lambda} t_1 dt_1.
\end{align*}
Then we have
\begin{equation}\label{eqn:1red}
\varphi_0(\rho; \lambda) = \frac{1}{2 \rho}(3-2\lambda) (-1+2\lambda) \int_{-1}^1  (1+ \rho t_1)^{-\frac{1}{2}-\lambda} t_1 dt_1
\end{equation}

Notice that $(1+ \rho t_1)^{-\frac{1}{2}-\lambda} t_1 |_{\rho=0} = t_1$, so we can write 
\begin{align*}
(1+ \rho t_1)^{-\frac{1}{2}-\lambda} t_1 = & \int_0^\rho \frac{d}{dt_2} (1+ t_1t_2)^{-\frac{1}{2}-\lambda} t_1 dt_2 + t_1 \\
= & -\frac{\rho(1 + 2\lambda) }{2} \int_0^1 (1+\rho t_1 t_2)^{-\frac{3}{2}-\lambda}  t_1^2 dt_2  + t_1
\end{align*}
Then
\begin{align*}
\int_{-1}^1  (1+ \rho t_1)^{-\frac{1}{2}-\lambda} t_1 dt_1 = & -\frac{\rho(1 + 2\lambda)}{2} \int_{-1}^1 \int_0^1 (1+\rho t_1t_2)^{-\frac{3}{2}-\lambda} t_1^2 dt_2  dt_1 + \int_{-1}^1 t_1 dt_1 \\
=& -\frac{\rho(1 + 2\lambda)}{2} \int_{-1}^1 \int_0^1 (1+\rho t_1 t_2)^{-\frac{3}{2}-\lambda}  t_1^2 dt_2  dt_1 .
\end{align*}
Hence we can further reduce $\varphi_0$ in the form
\begin{equation}\label{eqn:2red}
\varphi_0(\rho; \lambda) = -\frac{(3-2\lambda) (-1+2\lambda) (1 + 2\lambda)}{4} \int_{-1}^1 \int_0^1 (1+\rho t_1 t_2)^{-\frac{3}{2}-\lambda} t_1^2 dt_2  dt_1.
\end{equation}

\section{Strichartz estimates of the full semigroup}

\subsection{Preliminary}

We first introduce some results that will be useful. We start with recalling the result on oscillatory integrals from \cite{donn2017}.
\begin{lemma}[{\cite[Lemma 4.2]{donn2017}}]\label{lemma:o} We have
$$
\lf| \int_\RR e^{i a \omega} [\mathcal{O}_o(\lla \omega \rra^{-1}) + \mathcal{O}(\lla \omega \rra^{-2})] d\omega  \rt| \lesssim \lla a \rra^{-2}
$$
for all $a \in \RR \backslash \{0\}$. \qed
\end{lemma} 

The next result is in the same spirit as Lemma \ref{lemma:o}.

\begin{lemma}\label{lemma:a}
Let $\alpha \in (0,1)$. Then we have
$$
\lf| \int_\RR e^{i a \omega} \mathcal{O}(\lla \omega \rra^{-\alpha}) d\omega \rt| \lesssim |a|^{-1+\alpha} \lla a \rra ^{-2}
$$
for all $a\in \RR\backslash \{0\}$.
\end{lemma}
\begin{proof}
Without loss of generality assume $a>0$. We decompose
\begin{align*}
& \int_\RR e^{i a \omega} \mathcal{O}(\lla \omega \rra^{-\alpha}) d\omega \\
= &   \int_\RR e^{i a \omega} \mathcal{O}(\lla \omega \rra^{-\alpha}) \chi(a \omega) d\omega  + \int_\RR e^{i a \omega} \mathcal{O}(\lla \omega \rra^{-\alpha}) [1-\chi(a \omega) ] d\omega  \\
= : & I_1(a) + I_2(a).
\end{align*}
So we have
$$
I_1(a)  = \int_\RR e^{i a \omega} \mathcal{O}( \omega ^{-\alpha}) \chi(a \omega) d\omega   =  a^{-1} \int_\RR e^{i \omega} \mathcal{O}\lf( \lf(\frac{\omega}{a}\rt)^{-\alpha} \rt) \chi(\omega) d\omega = O\lf(a^{-1+\alpha} \rt),
$$
and an integration by parts gives
\begin{align*}
I_2(a) = & a^{-1}\int_\RR [1-\chi(a\omega)] e^{i a \omega } \mathcal{O} \lf(\omega^{-\alpha-1}\rt) d\omega + \int_\RR \chi'(a\omega) e^{i  a \omega } \lf(\omega^{-\alpha}\rt) d\omega \\
= & a^{-2}\int_\RR [1-\chi(\omega)] e^{i \omega } \mathcal{O} \lf( \lf( \frac{\omega}{a}\rt)^{-\alpha-1}\rt) d\omega + a^{-1} \int_\RR \chi'(\omega) e^{i   \omega } \lf(\lf( \frac{\omega}{a} \rt)^{-\alpha}\rt) d\omega \\
= & O\lf(a^{-1+\alpha} \rt).
\end{align*}
Moreover, for $a \geq 1$, we can do integration by parts as many times as we want and get
$$
 \int_\RR e^{i a \omega} \mathcal{O}(\lla \omega \rra^{-\alpha}) d\omega  = O( \lla a \rra ^{-2}).
$$
This gives the stated bound.
\end{proof}

We also have two results for the non-oscillatory case.

\begin{lemma}\label{lemma:rho1}
We have
$$
\lf| \rho^{-n} \int_\RR \lf[ 1- \chi(\rho\lla \omega \rra) \rt] e^{i a\omega} \mathcal{O}(\lla \omega \rra ^{-n-1}) d\omega \rt| \lesssim \lla a \rra ^{-2},
$$
for all $n \geq 1$, $\rho \in (0,1)$, and $a\in \RR$.
\end{lemma}
\begin{proof}
We compute
\begin{align*}
& \rho^{-n} \int_\RR \lf[ 1- \chi(\rho\lla \omega \rra) \rt] \mathcal{O}(\lla \omega \rra ^{-n-1}) d\omega  \\
= & \rho^{-n} \int_\RR \chi(|\omega|) \lf[ 1- \chi(\rho\lla \omega \rra) \rt] \mathcal{O}(\lla \omega \rra ^{-n-1}) d\omega  +  \rho^{-n} \int_\RR \lf[1 - \chi(|\omega|)\rt] \lf[ 1- \chi(\rho |\omega|) \rt] \mathcal{O}(|\omega| ^{-n-1}) d\omega \\
= &  \int_\RR \chi(|\omega|) \lf[ 1- \chi(\rho\lla \omega \rra ) \rt] \mathcal{O}(\lla \omega \rra ^{-1}) d\omega  +   \int_\RR \lf[1 - \chi\lf(\lf|\frac{\omega}{\rho} \rt| \rt)\rt] \lf[ 1- \chi(|\omega |) \rt] \mathcal{O}(\rho^0| \omega | ^{-n-1}) d\omega   \\
= & \mathcal{O}(\rho^0).
\end{align*}
Then the statement follows from two integrations by parts.
\end{proof}

The next lemma is similar.
\begin{lemma}\label{lemma:rho2}
We have
$$
\lf| \rho^{-n}  \int_\RR  \lf[ 1- \chi(\rho\lla \omega \rra ) \rt] e^{i a \omega } \mathcal{O} ( \lla \omega \rra^{-n} ) d\omega \rt|  \lesssim  |a|^{-1} \lla a \rra^{-2},
$$
for $n \geq 2$, $\rho \in (0,1)$, and $a\in \RR \backslash \{0\}$.
\end{lemma}

\begin{proof}
Without loss of generality we assume $a>0$. Integrating by parts once we get
\begin{align*}
&  \rho^{-n}  \int_\RR  \lf[ 1- \chi(\rho\lla \omega \rra ) \rt] e^{i a \omega} \mathcal{O} ( \lla \omega \rra^{-n}) ) d\omega  \\
 = & a^{-1} \rho^{-n} \int_\RR  \lf[ 1- \chi(\rho\lla \omega \rra ) \rt] e^{i a \omega}  \mathcal{O} ( \lla \omega \rra^{-n-1})  d\omega  + a^{-1}  \rho^{-n+1} \int_\RR  \chi'(\rho\lla \omega \rra )e^{i a \omega}  \mathcal{O} ( \lla \omega \rra^{-n} ) d\omega \\
= & a^{-1}  ( I_1(\rho,a) + I_2(\rho,a)), 
\end{align*}
where both $I_1$ and $I_2$ are $\mathcal{O}(\rho^0)$ by a similar computation as in Lemma \ref{lemma:rho1}. Then, when $a \geq 1$, we integrate by parts twice more to obtain the statement.
\end{proof}

\subsection{Kernel estimates}

Our goal is now to estimate $\|T_n(\tau) f\|_{L^5(\BB^5)}$. Applying Minkowski's inequality we get
\begin{align*}
\|T_n(\tau) f\|_{L^5(\BB^5)}  \cong & \lf\Vert \int_0^1 f(s) \int_\RR e^{i \omega \tau} G_n(\rho ,s ; i \omega) d\omega  ds \rt\Vert_{L^5_\rho(\BB^5)}  \\ 
\leq &\int_0^1  |f(s)| \lf\Vert \int_\RR e^{i \omega \tau} G_n(\rho ,s ; i \omega) d\omega \rt\Vert_{L^5_\rho (\BB^5)} ds.
\end{align*}

We have the following result.
\begin{prop}\label{prop:kernels} We have the bounds
$$
 \lf\Vert  \int_\RR e^{i \omega \tau} G_n(\rho ,s ; \omega) d\omega \rt\Vert _{L^5_\rho (\BB^5)} \lesssim s^2 (1-s)^{-\frac{1}{2}}\lf|\tau+\log(1-s)\rt|^{-\frac{1}{10}}\lla \tau + \log(1-s) \rra^{-1} 
$$
for all $\tau \geq 0$, $s\in (0,1)$, and $n \in \{1,2,\ldots, 6\}$.
\end{prop}

\begin{proof}
For $G_1$ we have 
\begin{align*}
& \int_\RR e^{i \omega \tau} G_1(\rho ,s ; \omega) d\omega  \\
= & \int_\RR  \chi(\rho\lla \omega \rra) e^{i \omega \tau} 1_{\RR_+}(s-\rho) \frac{ s(1-s)^{-\frac{1}{2}+i \omega}(2+s(-1+2 i \omega))  \varphi_0(\rho;i \omega)}{(3-2i\omega)(1+2i\omega)(1-2i\omega)} \gamma_1(\rho,s;i\omega)  d\omega \\
= & 2  \cdot 1_{\RR_+}(s-\rho)s(1-s)^{-\frac{1}{2}} \int_\RR \chi(\rho\lla \omega \rra) \frac{e^{i \omega (\tau+\log(1-s))} \varphi_0(\rho;i \omega)}{(3-2i\omega)(1+2i\omega)(1-2i\omega)} \gamma_1(\rho,s;i\omega) d\omega  \\
& -  1_{\RR_+}(s-\rho)  s^2 (1-s)^{-\frac{1}{2}} \int_\RR \chi(\rho\lla \omega \rra) \frac{e^{i \omega (\tau+\log(1-s))}   \varphi_0(\rho;i \omega)}{(3-2i\omega)(1+2i\omega)} \gamma_1(\rho,s;i\omega) d\omega \\
=: & 2 I_1(\rho,s,\tau) + I_2 (\rho,s,\tau) .
\end{align*}
For $I_1$ we use \eqref{eqn:2red}, then we have
\begin{align*}
& \lf| \int_\RR  \chi(\rho\lla \omega \rra) \frac{e^{i \omega (\tau+\log(1-s))} \varphi_0(\rho;i \omega)}{(3-2i\omega)(1+2i\omega)(1-2i\omega)} \gamma_1(\rho,s;i\omega) d\omega \rt| \\
\cong & \lf| \int_\RR  \chi(\rho\lla \omega \rra) e^{i \omega (\tau+\log(1-s))} \lf( \int_{-1}^1 \int_0^1 (1+\rho t_1 t_2)^{-\frac{3}{2}-i \omega} t_1^2 dt_2  dt_1 \rt)  \gamma_1(\rho,s;i\omega) d\omega \rt| \\
= & \lf| \int_{-1}^1 \int_0^1 (1+\rho t_1 t_2)^{-\frac{3}{2}} t_1^2 \lf( \int_\RR  \chi(\rho\lla \omega \rra) e^{i \omega (\tau+\log(1-s) + \log(1+\rho t_1 t_2))} \gamma_1(\rho,s;i\omega)  d\omega \rt)  dt_2  dt_1 \rt| \\
\leq &  \sup_{t_1\in(-1,1),t_2\in(0,1)} \lf|  \int_\RR\chi(\rho\lla \omega \rra) e^{i \omega (\tau+\log(1-s) + \log(1+\rho t_1 t_2))} \gamma_1(\rho,s;i\omega)  d\omega \rt| .
\end{align*}
The order of integrations can be changed because $\rho \in (0,1)$ is fixed. Moreover, thanks to the cut-off $\chi(\rho\lla \omega \rra)$, we have $|(1+\rho t_1 t_2)^{-\frac{3}{2}}| \lesssim 1$ and $|\log(1+\rho t_1 t_2) | \lesssim 1$ for all $t_1, t_2$ in the respective domains. So we can use Lemma \ref{lemma:o} and get
\begin{align*}
& \lf| \int_\RR \chi(\rho\lla \omega \rra) e^{i \omega (\tau+\log(1-s) + \log(1+\rho t_1 t_2))} \gamma_1(\rho,s;i\omega)  d\omega \rt| \\
 \lesssim  & \lla \tau+\log(1-s) + \log(1+\rho t_1 t_2) \rra ^{-2}  \lesssim  \lla \tau+\log(1-s) \rra ^{-2}.
\end{align*}
Plugging this back in we have
\begin{equation}\label{eqn:I1est}
|I_1(\rho,s,\tau)| \lesssim  1_{\RR_+}(s-\rho)s(1-s)^{-\frac{1}{2}} \lla \tau+\log(1-s) \rra ^{-2} , 
\end{equation}
so
\begin{align*}
\|I_1(\cdot,s,\tau)\|_{L^5(\BB^5)} & \lesssim  \lf( \int_0^1 \lf| 1_{\RR_+}(s-\rho)s(1-s)^{-\frac{1}{2}} \lla \tau+\log(1-s) \rra ^{-2}  \rt|^5 \rho^4 d\rho  \rt)^\frac{1}{5}   \\
& = s(1-s)^{-\frac{1}{2}} \lla \tau+\log(1-s) \rra ^{-2}  \lf( \int_0^s  \rho^4 d\rho \rt)^\frac{1}{5} \\
& \cong s^2 (1-s)^{-\frac{1}{2}} \lla \tau+\log(1-s) \rra ^{-2}   .
\end{align*}
For $I_2$ we directly use that 
\begin{align*}
I_2(\rho,s,\tau) = &  1_{\RR_+}(s-\rho)  s^2 (1-s)^{-\frac{1}{2}} \int_\RR e^{i \omega (\tau+\log(1-s))}   \varphi_0(\rho;i \omega) \\
& \times \mathcal{O}(\rho^0 (1-\rho)^0 s^0 (1-s)^0 \lla \omega \rra ^{-3}) \chi(\rho\lla \omega \rra) d\omega,
\end{align*}
and we look at the integral kernel,
\begin{align*}
& \int_\RR e^{i \omega (\tau+\log(1-s))}  \chi(\rho\lla \omega \rra) \mathcal{O}(\rho^0  (1-\rho)^0 s^0 (1-s)^0 \lla \omega \rra ^{-3})  \varphi_0(\rho;i \omega) d\omega  \\
= & \int_\RR e^{i \omega (\tau+\log(1-s) - \log(1+\rho))}  \chi(\rho\lla \omega \rra) \mathcal{O}(\rho^0  (1-\rho)^0 s^0 (1-s)^0\lla \omega \rra ^{-3})  \\
& \times \frac{1}{\rho^3} \lf[ (1+\rho)^{\frac{1}{2}}(2+\rho(-1+2i \omega )) - e^{iw(\log(1+\rho)- \log(1-\rho))}(1-\rho)^{\frac{1}{2}}(2+\rho(1-2i \omega )) \rt] d\omega .
\end{align*}
Since on the support of the cut-off $\chi(\rho\lla \omega \rra)$, we have
$ e^{iw(\log(1+\rho)- \log(1-\rho))} = 1+\mathcal{O}(\omega \rho)$, the function
$$
g(\rho,\omega) := (1+\rho)^{\frac{1}{2}}(2+\rho(-1+2i \omega )) - e^{iw(\log(1+\rho)- \log(1-\rho))}(1-\rho)^{\frac{1}{2}}(2+\rho(1-2i \omega ))
$$
is of symbol type. Furthermore, we have 
$$
g(0,\omega) = \PD_\rho g(\rho,\omega)|_{\rho=0} = \PD^2_\rho g(\rho,\omega)|_{\rho=0} = 0 \text{ for all } \omega \in \RR,
$$ 
so by Cauchy remainder we have
$$
g(\rho, \omega) = \frac{\rho^3}{2!} \int_0^1 \PD^3_1 g(\rho t,\omega) (1-t)^2 dt,
$$
where $\PD_1$ is the derivative with respect to the first variable. Since $|\PD_1^3g(\rho t, \omega)| \lesssim \lla \omega \rra ^3$ on the support of $\chi(\rho\lla \omega \rra)$, we obtain $|g(\rho,\omega)| \lesssim \rho^3 \lla \omega \rra^3$. Hence $g(\rho,\omega) = \mathcal{O}(\rho^3 \lla \omega \rra ^3)$, and on the support of $\chi(\rho\lla \omega \rra)$, we also have $g(\rho,\omega) = \mathcal{O} (\rho^{\frac{21}{10}} \lla \omega \rra ^{\frac{21}{10}}) $. Thus Lemma \ref{lemma:a} gives the estimate
\begin{align*}
& \lf| \int_\RR e^{i \omega (\tau+\log(1-s) )} \chi(\rho\lla \omega \rra) \mathcal{O}(\rho^0  (1-\rho)^0 s^0 (1-s)^0 \lla \omega \rra ^{-3})  \varphi_0(\rho;i \omega) d\omega \rt|  \\
= &  \lf| \rho^{-\frac{9}{10}} \int_\RR e^{i \omega (\tau+\log(1-s) - \log(1+\rho))}  \chi(\rho\lla \omega \rra)  \mathcal{O} (\rho^0 \lla \omega \rra ^{-\frac{9}{10}}) d\omega \rt| \\
\lesssim & \rho^{-\frac{9}{10}} |\tau+\log(1-s) - \log(1+\rho)|^{-\frac{1}{10}} \lla \tau+\log(1-s) - \log(1+\rho) \rra ^{-2} \\
\leq & \rho^{-\frac{9}{10}} |\tau+\log(1-s) - \log(1+\rho)|^{-\frac{1}{10}} \lla \tau+\log(1-s) \rra ^{-2},
\end{align*}
so
\begin{equation}\label{eqn:I2est}
\begin{aligned}
& |I_2(\rho,s,\tau)| \\
 \lesssim & 1_{\RR_+} (s-\rho) s^2(1-s)^{-\frac{1}{2}} \rho^{-\frac{9}{10}} |\tau+\log(1-s) - \log(1+\rho)|^{-\frac{1}{10}} \lla \tau+\log(1-s) \rra ^{-2}.
\end{aligned}
\end{equation}
Then we have
\begin{align*}
& \|I_2(\cdot,s,\tau)\|_{L^5(\BB^5)} \\
\lesssim  & s^2(1-s)^{-\frac{1}{2}}  \lla \tau+\log(1-s) \rra ^{-2}  \lf(\int_0^s \lf(  \rho^{-\frac{9}{10}} |\tau+\log(1-s) - \log(1+\rho)|^{-\frac{1}{10}} \rt)^5 \rho^4 d\rho \rt)^{\frac{1}{5}} \\
\leq  & s^2(1-s)^{-\frac{1}{2}}  \lla \tau+\log(1-s) \rra ^{-2} \lf(\int_0^s \rho^{-\frac{1}{2}}  |\tau+\log(1-s) - \log(1+\rho)|^{-\frac{1}{2}}  d\rho \rt)^{\frac{1}{5}} .
\end{align*}
We estimate
\begin{align*}
\lf| \int_0^1 \rho^{-\frac{1}{2}} \lf| x-\log(1+\rho) \rt|^{-\frac{1}{2}} d\rho \rt| \leq & \int_0^{\frac{\log(2)}{x}} \lf( e^{xy} - 1 \rt)^{-\frac{1}{2}} | x - xy| ^{-\frac{1}{2}} x e^{xy} dy\\
= & x |x|^{-\frac{1}{2}} \int_0^{\frac{\log(2)}{x}}  \lf( e^{xy} - 1 \rt)^{-\frac{1}{2}}| 1 - y| ^{-\frac{1}{2}} e^{xy} dy \\
\leq & x |x|^{-\frac{1}{2}} \int_0^{\frac{1}{x}}  (x y)^{-\frac{1}{2}}| 1 - y | ^{-\frac{1}{2}} e^{xy} dy \\
\lesssim & \int_0^{\frac{1}{|x|}} |y|^{-\frac{1}{2}}  | 1 - y| ^{-\frac{1}{2}} dy \\
\lesssim & \Big| \log|x| \Big| + 1 \lesssim |x|^{-\frac{1}{2}} + 1.
\end{align*}
Since $s\leq 1$,
$$
\lf| \int_0^s \rho^{-\frac{1}{2}}  |\tau+\log(1-s) - \log(1+\rho)|^{-\frac{1}{2}}  d\rho \rt| \lesssim  |\tau + \log(1-s) |^{-\frac{1}{2}} +1.
$$
With this, we obtain the bound
\begin{align*}
\|I_2(\cdot,s,\tau)\|_{L^5(\BB^5)}  \lesssim &  s^2(1-s)^{-\frac{1}{2}} \lf(  |\tau + \log(1-s) |^{-\frac{1}{2}} +1 \rt)^{\frac{1}{5}}   \lla \tau+\log(1-s) \rra ^{-2}  \\
\lesssim & s^2(1-s)^{-\frac{1}{2}}   |\tau + \log(1-s) |^{-\frac{1}{10}}  \lla \tau+\log(1-s) \rra ^{-1}  .
\end{align*}

Next, for $G_2$ we have 
\begin{align*}
& \int_\RR e^{i \omega \tau} G_2(\rho ,s ; \omega) d\omega   \\
= & 2 \cdot 1_{\RR_+}(s-\rho)s(1-s)^{-\frac{1}{2}} \int_\RR  \lf[ 1- \chi(\rho\lla \omega \rra) \rt] \frac{e^{i \omega (\tau+\log(1-s))} \varphi_1(\rho;i \omega)}{(3-2i\omega)(1+2i\omega)(1-2i\omega)} \gamma_2(\rho,s;i\omega)  d\omega  \\
 & - 1_{\RR_+}(s-\rho)  s^2 (1-s)^{-\frac{1}{2}} \int_\RR \lf[ 1- \chi(\rho\lla \omega \rra) \rt] \frac{e^{i \omega (\tau+\log(1-s))}   \varphi_1(\rho;i \omega)}{(3-2i\omega)(1+2i\omega)} \gamma_2(\rho,s;i\omega) d\omega  \\
=: &2 I_1(\rho,s,\tau) + I_2 (\rho, s,\tau) .
\end{align*}
In $I_1$ we look at
\begin{align*}
& \int_\RR  \lf[ 1- \chi(\rho\lla \omega \rra) \rt] \frac{e^{i \omega (\tau+\log(1-s))} \varphi_1(\rho;i \omega)}{(3-2i\omega)(1+2i\omega)(1-2i\omega)} \gamma_2(\rho,s;i\omega)  d\omega \\
= & \int_\RR  \lf[ 1- \chi(\rho\lla \omega \rra) \rt]  e^{i \omega (\tau+\log(1-s))} \varphi_1(\rho;i \omega)\mathcal{O}(\rho^0 (1-\rho)^0 s^0 (1-s)^0 \lla \omega \rra^{-4})  d\omega\\
=& (1+\rho)^{\frac{1}{2}} \rho^{-2}   \int_\RR  \lf[ 1- \chi(\rho\lla \omega \rra) \rt]  e^{i \omega (\tau+\log(1-s) - \log(1+\rho))}   \mathcal{O}(\rho^0 (1-\rho)^0 s^0 (1-s)^0 \lla \omega \rra^{-3}) d\omega 
\end{align*}
Using Lemma \ref{lemma:rho1} we get
\begin{align*}
&  \rho^{-2} \int_\RR  \lf[ 1- \chi(\rho\lla \omega \rra) \rt]  e^{i \omega (\tau+\log(1-s) - \log(1+\rho))}   \mathcal{O}(\rho^0 (1-\rho)^0 s^0 (1-s)^0 \lla \omega \rra^{-3}) d\omega \\
= & O( \lla \tau+\log(1-s) - \log(1+\rho) \rra ^{-2} ).
\end{align*}
Then since $\log(1+\rho)$ is bounded for all $\rho \in [0,1]$, we get
\begin{align*}
| I_1(\rho,s,\tau)| \lesssim & 1_{\RR_+}(s-\rho) s (1-s)^{-\frac{1}{2}} \lla \tau+\log(1-s) - \log(1+\rho) \rra ^{-2}  \\
\lesssim &  1_{\RR_+}(s-\rho) s (1-s)^{-\frac{1}{2}} \lla \tau+\log(1-s)\rra ^{-2},
\end{align*}
which is the same as \eqref{eqn:I1est} and the rest follows the same way. In $I_2$ the expression becomes 
\begin{align*}
I_2(\rho,s,\tau) = & 1_{\RR_+}(s-\rho) s^2 (1-s)^{-\frac{1}{2}} (1+\rho)^{\frac{1}{2}} \rho^{-2} \int_\RR  \lf[ 1- \chi(\rho\lla \omega \rra ) \rt] e^{i\omega (\tau + \log(1-s) - \log(1+\rho))} \\
& \times  \mathcal{O}(\rho^0 (1-\rho)^0 s^0 (1-s)^0 \lla \omega \rra^{-2}) d\omega .
\end{align*}
Here we first use Lemma \ref{lemma:rho1} to estimate
\begin{align}
& \rho^{-2} \lf|\int_\RR  \lf[ 1- \chi(\rho\lla \omega \rra ) \rt] e^{i\omega (\tau + \log(1-s) - \log(1+\rho))}  \mathcal{O} \lf( \lla \omega \rra^{-2}) \rt) d\omega \rt| \nonumber \\
 \lesssim &  \rho^{-1}  \lla \tau+\log(1-s) - \log(1+\rho) \rra ^{-2} \lesssim \rho^{-1}  \lla \tau+\log(1-s)  \rra ^{-2}. \label{eqn:G2I2est}
\end{align}
On the other hand, using Lemma \ref{lemma:rho2} we have
\begin{align*}
& \rho^{-2} \lf|\int_\RR  \lf[ 1- \chi(\rho\lla \omega \rra ) \rt] e^{i\omega (\tau + \log(1-s) - \log(1+\rho))}  \mathcal{O} \lf( \lla \omega \rra^{-2}) \rt) d\omega \rt| \\
 \lesssim &   \lf| \tau + \log(1-s) - \log(1+\rho) \rt|^{-1} \lla \tau+\log(1-s) - \log(1+\rho) \rra ^{-2} \\
  \lesssim &   \lf| \tau + \log(1-s) - \log(1+\rho) \rt|^{-1} \lla \tau+\log(1-s) \rra ^{-2}.
\end{align*}
We can then interpolate the two bounds and get
\begin{align*}
 \rho^{-2} \lf|\int_\RR  \lf[ 1- \chi(\rho\lla \omega \rra ) \rt] e^{i\omega (\tau + \log(1-s) - \log(1+\rho))} \mathcal{O} \lf( \lla \omega \rra^{-2}) \rt) d\omega \rt| \\
 \lesssim \rho^{-\frac{9}{10}}  \lf| \tau + \log(1-s) - \log(1+\rho) \rt|^{-\frac{1}{10}} \lla \tau+\log(1-s) \rra ^{-2}.
\end{align*}
Now 
\begin{align*}
 & | I_2(\rho, s, \tau) | \\
 \lesssim & 1_{\RR_+}(s-\rho) s^2 (1-s)^{-\frac{1}{2}} (1+\rho)^{\frac{1}{2}}   \rho^{-\frac{9}{10}}  \lf| \tau + \log(1-s) - \log(1+\rho) \rt|^{-\frac{1}{10}} \lla \tau+\log(1-s) \rra ^{-2}  \\
 \lesssim &  1_{\RR_+}(s-\rho) s^2 (1-s)^{-\frac{1}{2}}  \rho^{-\frac{9}{10}}  \lf| \tau + \log(1-s) - \log(1+\rho) \rt|^{-\frac{1}{10}} \lla \tau+\log(1-s) \rra ^{-2}  ,
\end{align*}
which is the same as \eqref{eqn:I2est}, and the rest follows the same way.

Similar to $G_2$, in $G_3$ we have to bound the two terms
\begin{align*}
I_1(\rho,s,\tau)  =&  1_{\RR_+}(s-\rho)s(1-s)^{-\frac{1}{2}} (1-\rho)^{\frac{1}{2}} \rho^{-2} \int_\RR  \lf[ 1- \chi(\rho\lla \omega \rra) \rt]  e^{i \omega (\tau+\log(1-s) - \log(1-\rho))}  \\
& \times  \mathcal{O}(\rho^0 (1-\rho)^0 s^0 (1-s)^0 \lla \omega \rra^{-3}) d\omega  
\end{align*}
and
\begin{align*}
I_2(\rho,s,\tau) = &  1_{\RR_+}(s-\rho) s^2 (1-s)^{-\frac{1}{2}} (1-\rho)^{\frac{1}{2}} \rho^{-2} \int_\RR  \lf[ 1- \chi(\rho\lla \omega \rra ) \rt] e^{i\omega (\tau + \log(1-s) - \log(1-\rho))} \\
& \times  \mathcal{O}(\rho^0 (1-\rho)^0 s^0 (1-s)^0 \lla \omega \rra^{-2}) d\omega.
\end{align*}
For $I_1$, from Lemma \ref{lemma:rho1} we have the bound
\begin{align*}
& (1-\rho)^{\frac{1}{2}} \rho^{-2} \int_\RR  \lf[ 1- \chi(\rho\lla \omega \rra ) \rt] e^{i\omega (\tau + \log(1-s) - \log(1-\rho))}   \mathcal{O}(\rho^0 (1-\rho)^0 s^0 (1-s)^0 \lla \omega \rra^{-2}) d\omega  \\
= & (1-\rho)^{\frac{1}{2}} O( \lla \tau + \log(1-s) - \log(1-\rho) \rra ^{-2}).
\end{align*}
We then use
\begin{equation}\label{eqn:T3est}
(1-\rho)^{\frac{1}{2}} \lla \tau+\log(1-s) - \log(1-\rho)\rra ^{-2} \lesssim \lla \tau+\log(1-s)\rra ^{-2}
\end{equation}
and obtain the exact same estimate as for $I_1$ in the $G_2$ case. For $I_2$, similar to before, we interpolate the bounds from Lemma \ref{lemma:rho1} and \ref{lemma:rho2},
\begin{align*}
& (1-\rho)^{\frac{1}{2}} \rho^{-2} \lf| \int_\RR  \lf[ 1- \chi(\rho\lla \omega \rra ) \rt] e^{i\omega (\tau + \log(1-s) - \log(1-\rho))}  \mathcal{O}(\rho^0 (1-\rho)^0 s^0 (1-s)^0 \lla \omega \rra^{-2}) d\omega \rt| \\
\lesssim &  \rho^{-\frac{9}{10}} \lf| \tau + \log(1-s) - \log(1-\rho) \rt|^{-\frac{1}{10}} (1-\rho)^{\frac{1}{2}}\lla \tau+\log(1-s) - \log(1-\rho)\rra ^{-2} \\
\lesssim  & \rho^{-\frac{9}{10}} \lf| \tau + \log(1-s) - \log(1-\rho) \rt|^{-\frac{1}{10}} \lla \tau+\log(1-s)\rra ^{-2}.
\end{align*}
Hence
\begin{align*}
& \|I_2(\cdot,s,\tau) \|_{L^5(\BB^5)} \\
\lesssim & s^2 (1-s)^{-\frac{1}{2}} \lla \tau+\log(1-s)\rra ^{-2} \lf( \int_0^s  \rho^{-\frac{1}{2}} |\tau+\log(1-s)-\log(1-\rho) |^{-\frac{1}{2}} \rt)^{\frac{1}{5}}.
\end{align*}
We estimate
\begin{align*}
 \lf| \int_0^1 \rho^{\frac{1}{2}} \lf| x - \log(1-\rho) \rt| ^{-\frac{1}{2}} d\rho \rt| 
\leq & |x| \int_0^\infty (1-e^{-|x|y})^{-\frac{1}{2}} \Big| x + |x| y \Big|^{-\frac{1}{2}} e^{-|x|y} dy \\
\lesssim & |x|^{\frac{1}{2}} \int_0^2 (|x|y)^{-\frac{1}{2}} | 1 - y|^{-\frac{1}{2}} dy + |x|^{\frac{1}{2}} \int_2^\infty (|x|y)^{-1} | 1 - y|^{-\frac{1}{2}} dy \\
\lesssim & \int_0^2 y^{-\frac{1}{2}} | 1 - y|^{-\frac{1}{2}} dy + |x|^{-\frac{1}{2}} \int_2^\infty y^{-\frac{3}{2}} dy \\
\lesssim & |x|^{-\frac{1}{2}} + 1.
\end{align*}
Using this we have
\begin{align*}
\|I_2(\cdot, s, \tau)\|_{L^5(\BB^5)} \lesssim & s^2 (1-s)^{-\frac{1}{2}} \lf( |\tau+\log(1-s)| ^{-\frac{1}{2}} + 1  \rt)^\frac{1}{5}  \lla \tau +\log(1-s)\rra ^{-2}    \\
\lesssim & s^2 (1-s)^{-\frac{1}{2}} |\tau+\log(1-s)| ^{-\frac{1}{10}}  \lla \tau +\log(1-s)\rra ^{-1}. 
\end{align*}

For $G_4$ we have
$$
\int_\RR e^{i \omega \tau} G_4(\rho ,s ; \omega) d\omega = 2 I_1(\rho,s, \tau) + I_2(\rho,s, \tau),
$$
with 
\begin{align*}
I_1(\rho, s, \tau) := &  1_{\RR_+}(\rho-s) \frac{s^4}{\rho^{3}}(1-s^2)^{-\frac{1}{2}} (1 +\rho)^{\frac{1}{2}} \\
& \times \int_\RR \chi(s\lla \omega \rra) e^{i \omega(\tau + \log(1-s^2)-\log(1+\rho))}  \frac{ \varphi_0(s;i \omega)  \gamma_4(\rho,s;i \omega)}{(3-2i \omega)(1+2i \omega)(1-2i \omega)} d\omega 
\end{align*}
and
\begin{align*}
I_2(\rho, s, \tau) := &  1_{\RR_+}(\rho-s) \frac{s^4}{\rho^2}(1-s^2)^{-\frac{1}{2}} (1 +\rho)^{\frac{1}{2}} \\
& \times \int_\RR \chi(s\lla \omega \rra) e^{i \omega(\tau + \log(1-s^2)-\log(1+\rho))}  \frac{ \varphi_0(s;i \omega)  \gamma_4(\rho,s;i \omega)}{(3-2i \omega)(-1-2i \omega)} d\omega .
\end{align*}
In $I_1$ we use the expression \eqref{eqn:2red} for $\varphi_0(s; i \omega)$, and write
\begin{align*}
I_1(\rho, s, \tau) \cong &  1_{\RR_+}(\rho-s)\frac{s^4}{\rho^{3}}(1-s^2)^{-\frac{1}{2}} (1 +\rho)^{\frac{1}{2}}  \int_{-1}^1 \int_0^1 (1+s t_1 t_2)^{-\frac{3}{2}} t_1^2 \\
& \times \int_\RR \chi(s\lla \omega \rra)  e^{i \omega(\tau + \log(1-s^2)-\log(1+\rho) - \log(1+ s t_1 t_2))} \gamma_4(\rho,s;i \omega)d\omega dt_2 dt_1 ,
\end{align*}
where the order of integration in $t_1, t_2$ and $\omega$ can be changed because $s$ is fixed. Similar to the $G_1$ case, on the support of the cut-off $\chi(s\lla \omega \rra)$,  we have $|(1+s t_1 t_2)^{-\frac{3}{2}}| \lesssim 1$ and $|\log(1+s t_1 t_2)| \lesssim 1$ for all $t_1\in[-1,1]$ and $t_2\in[0,1]$, so
\begin{align*}
& \lf|\int_{-1}^1 \int_0^1 (1+s t_1 t_2)^{-\frac{3}{2}} t_1^2 \int_\RR e^{i \omega(\tau + \log(1-s^2)-\log(1+\rho) - \log(1+ s t_1 t_2))} \gamma_4(\rho,s;i \omega)d\omega dt_2 dt_1 \rt. \\
 & \lf. \qquad \qquad \times \lf[  \mathcal{O}(\rho^0 s^0) \mathcal{O}_0( \lla \omega \rra^{-1} ) + \mathcal{O}((1-\rho)^0 s^0 \rho^0 (1-s)^0 \lla \omega \rra^{-2}) \rt]\rt|  \\
 \lesssim & \sup_{t_1\in(-1,1), t_2\in(0,1)} \lla \tau + \log(1-s^2)-\log(1+\rho) - \log(1+ s t_1 t_2) \rra^{-2} \\
 \lesssim &  \lla \tau + \log(1-s) \rra^{-2}.
\end{align*}
Hence we obtain
\begin{align}
|I_1(\rho, s, \tau)| \lesssim &  1_{\RR_+}(\rho-s)\frac{s^4}{\rho^{3}}(1-s^2)^{-\frac{1}{2}} (1 +\rho)^{\frac{1}{2}}  \lla \tau + \log(1-s) \rra^{-2}  \nonumber \\
\lesssim &  1_{\RR_+}(\rho-s)\frac{s^3}{\rho^2}(1-s)^{-\frac{1}{2}}  \lla \tau + \log(1-s) \rra^{-2}. \label{eqn:T4est1}
\end{align}
The situation for $I_2$ is similar, we use \eqref{eqn:1red} and write
\begin{align*}
I_2(\rho, s, \tau) = & 1_{\RR_+}(\rho-s) \frac{s^3}{\rho^2}(1-s^2)^{-\frac{1}{2}} (1 +\rho)^{\frac{1}{2}} \int_{-1}^1 (1+s t)^{-\frac{1}{2}} t \\
& \times \int_\RR \chi(s\lla \omega \rra) e^{i \omega(\tau + \log(1-s^2)-\log(1+\rho)-\log(1+s t))}  \frac{1-2i\omega}{1+2i\omega}\gamma_4(\rho,s;i \omega) d\omega dt 
\end{align*}
Since
$$
\frac{1-2i\omega}{1+2i\omega} = \frac{1 - 4 i \omega - 4\omega^2}{1 - 4\omega^2} = \mathcal{O}_e(\lla \omega \rra^0) + \mathcal{O}_o(\lla \omega \rra^{-1}) + \mathcal{O}(\lla \omega \rra^{-2}),$$
we see that $\frac{1-2i\omega}{1+2i\omega}\gamma_4(\rho,s;i \omega)$ has the same form as $\gamma_4(\rho,s;i \omega)$. Hence by the same logic as the $I_1$ case, we obtain the same bound
$$
| I_2(\rho, s, \tau)| \lesssim 1_{\RR_+}(\rho-s) \frac{s^3}{\rho^2}(1-s)^{-\frac{1}{2}} \lla \tau + \log(1-s)\rra^{-2} ,
$$
so for $n=1,2$ we have
\begin{align*}
\|I_n(\cdot, s, \tau)\|_{L^5(\BB^5)} \lesssim &  s^3 (1-s)^{-\frac{1}{2}} \lla \tau + \log(1-s) \rra^{-2} \lf( \int_s^1 \rho^{-6}  d\rho \rt)^{\frac{1}{5}}\\
= & s^3 (1-s)^{-\frac{1}{2}} \lla \tau + \log(1-s) \rra^{-2}  \lf( -\frac{1}{5} + \frac{1}{5 s^{5}}  \rt)^{\frac{1}{5}} \\
\lesssim & s^2 (1-s)^{-\frac{1}{2}} \lla \tau + \log(1-s) \rra^{-2} .
\end{align*}

For $G_5$ we use that on the support of the cut-off $1-\chi(s\lla \omega \rra)$ we have $s^{-1} \lesssim \lla \omega\rra$, and write
\begin{align*}
G_5(\rho,s;i \omega) 
 = &   1_{\RR_+}(\rho-s)\lf[ 1- \chi(s\lla \omega \rra) \rt] s^4(1-s)^{-\frac{1}{2}+ i \omega} \varphi_1(\rho;\lambda) s^{-2}\mathcal{O}(\lla \omega \rra^{-3}) \\
 = & 2\cdot 1_{\RR_+}(\rho-s)\lf[ 1- \chi(s\lla \omega \rra) \rt]\frac{s^4}{\rho^3} (1-s)^{-\frac{1}{2}+ i \omega} (1+\rho)^{\frac{1}{2} - i \omega }   s^{-2}\mathcal{O}(\lla \omega \rra^{-3}) \\
 & +  1_{\RR_+}(\rho-s)\lf[ 1- \chi(s\lla \omega \rra) \rt]\frac{s^3}{\rho^2} (1-s)^{-\frac{1}{2}+ i \omega} (1+\rho)^{\frac{1}{2} - i \omega }  s^{-1}\mathcal{O}(\lla \omega \rra^{-2}).
 \end{align*}
So we need to bound the two expressions
\begin{align*}
 I_1(\rho, s, \tau) =&   1_{\RR_+}(\rho-s)\frac{s^4}{\rho^3} (1-s)^{-\frac{1}{2}}(1+\rho)^{\frac{1}{2} } \\
& \times s^{-2} \int_\RR \lf[ 1- \chi(s\lla \omega \rra) \rt] e^{i \omega (\tau + \log(1-s) - \log(1+\rho)) } \mathcal{O}(\lla \omega \rra^{-3}) d\omega  ,
\end{align*}
and
\begin{align*}
I_2(\rho, s, \tau)= & 1_{\RR_+}(\rho-s)\frac{s^3}{\rho^2} (1-s)^{-\frac{1}{2}}(1+\rho)^{\frac{1}{2} } \\
& \times s^{-1} \int_\RR \lf[ 1- \chi(s\lla \omega \rra) \rt] e^{i \omega (\tau + \log(1-s) - \log(1+\rho)) } \mathcal{O}(\lla \omega \rra^{-2})  d\omega. \\
\end{align*}
Now we apply Lemma \ref{lemma:rho1}. Since $\log(1+\rho)$ is bounded, we obtain the bounds
\begin{align*}
| I_1(\rho, s, \tau) | \lesssim & 1_{\RR_+}(\rho-s)\frac{s^4}{\rho^3} (1-s)^{-\frac{1}{2}} \lla \tau + \log(1-s) \rra^{-2} \\
\leq &1_{\RR_+}(\rho-s)\frac{s^3}{\rho^2} (1-s)^{-\frac{1}{2}} \lla \tau + \log(1-s) \rra^{-2} ,
\end{align*}
and
$$
| I_2(\rho, s, \tau) | \lesssim   1_{\RR_+}(\rho-s)\frac{s^3}{\rho^2} (1-s)^{-\frac{1}{2}} \lla \tau + \log(1-s) \rra^{-2} .
$$
These bounds are the same as \eqref{eqn:T4est1}.

The last term, $G_6$, is similar to $G_5$. We look at
\begin{align*}
I_1 (\tau)f(\rho) = &  1_{\RR_+} (\rho-s)  \frac{s^4}{\rho^3}  (1+s)^{-\frac{1}{2}}(1+\rho)^{\frac{1}{2}} \\
& \times \int_{\RR_+} \lf[ 1- \chi(s\lla \omega \rra) \rt]  e^{i \omega (\tau + \log(1+s)-\log(1+\rho))}  s^{-2}\mathcal{O}(\lla \omega \rra^{-3}) d\omega 
\end{align*}
and
\begin{align*}
I_2(\tau) f(\rho) = &  1_{\RR_+}(\rho-s) \frac{s^3}{\rho^2}  (1+s)^{-\frac{1}{2}}(1+\rho)^{\frac{1}{2}} \\ 
& \times \int_{\RR_+} \lf[ 1- \chi(s\lla \omega \rra) \rt]  e^{i \omega (\tau + \log(1+s)-\log(1+\rho))}  s^{-1}\mathcal{O}(\lla \omega \rra^{-2}) d\omega .
\end{align*}
Using 
$$
 (1+s)^{-\frac{1}{2}} \lla \tau + \log(1+s) -\log(1+\rho)\rra^{-2} \lesssim  \lla \tau \rra^{-2} \lesssim  (1-s)^{-\frac{1}{2}} \lla \tau + \log(1-s) \rra^{-2},
$$
we obtain the same bounds as $G_5$.
\end{proof}

Using this, we obtain boundedness of the operators $T_n$.

\begin{prop}\label{prop:Tbound}
Let $p\in [2, \infty]$ and $q\in [\frac{10}{3},5]$ such that $\frac{1}{p}+\frac{5}{q} = \frac{3}{2}$. Then we have the bound
$$
\|T_n(\cdot) f\|_{L^p(\RR_+; L^q(\BB^5))} \lesssim \| f \|_{L^2(\BB^5)}
$$
for all $f\in C([0,1])$ and $n \in \{ 1,2, \ldots, 6\}$.
\end{prop}

\begin{proof}
From Proposition \ref{prop:kernels} we have
\begin{align*}
\|T_n(\tau) f\|_{L^5(\BB^5)} \lesssim &\int_0^1 |f(s)  | \lf\Vert \int_\RR e^{i \omega \tau} G_n(\rho ,s ; \omega) d\omega \rt\Vert _{L^5_\rho (\BB^5)} ds\\
\lesssim & \int_0^1 | f(s) | s^2(1-s)^{-\frac{1}{2}} |\tau+\log(1-s)|^{-\frac{1}{10}} \lla \tau + \log(1-s)\rra^{-1}  ds \\
\lesssim & \int_0^\infty \lf| f(1-e^{-y}) \rt| (1-e^{-y})^2 e^{-\frac{1}{2}y} |\tau-y|^{-\frac{1}{10}}\lla \tau - y \rra ^{-1} dy,
\end{align*}
by doing a change of variables $s = 1-e^{-y}$. Applying Young's inequality gives
\begin{align*}
\|T_n(\cdot) f\|_{L^2(\RR_+;L^5(\BB^5))}  \lesssim & \lf( \int_0^\infty \lf| f(1-e^{-y}) \rt| ^2 (1-e^{-y})^4 e^{-y}  dy \rt)^\frac{1}{2} \int_\RR |y|^{-\frac{1}{10}} \lla y \rra^{-1}  dy\\
 \lesssim & \|f\|_{L^2(\BB^5)}.
\end{align*}
On the other hand, Cauchy-Schwarz inequality gives
\begin{align*}
\|T_n(\tau) f\|_{L^5(\BB^5)}  \lesssim & \|f\|_{L^2(\BB^5)} \lf\Vert |\tau -\cdot|^{-\frac{1}{10}} \lla \tau - \cdot \rra^{-1}\rt\Vert_{L^2(\RR)}  \\
\lesssim & \|f\|_{L^2(\BB^5)} ,
\end{align*}
so $\|T_n(\cdot) f\|_{L^\infty(\RR_+;L^{\frac{10}{3}}(\BB^5))} \lesssim \|I_2\|_{L^\infty(\RR_+;L^{5}(\BB^5))}  \lesssim \|f\|_{L^2(\BB^5)}$. Therefore the claim follows by interpolation.
\end{proof}

\subsection{The term containing $\lambda$}

We recall that the function $F_\lambda$ contains a term $\lambda \tilde{f}_1$. For convenience we rewrite this term as
$$
F_\lambda (s) = s \tilde{f}'_1(s) + \lf(\lambda+\frac{5}{2} \rt)\tilde{f}_1(s) + \tilde{f}_2(s)  = s \tilde{f}'_1(s) + \lf(\lambda+\frac{1}{2} \rt)\tilde{f}_1(s) + 2 \tilde{f}_1(s) + \tilde{f}_2(s)
$$
and define for $\tau \geq 0$, $\rho \in(0,1)$, $f \in C^1([0,1])$, and $n \in \{1,2,\ldots,6 \}$, 
$$
\dot{T}_{n,\epsilon}(\tau) f(\rho) := \frac{1}{2\pi i} \lim_{N\to\infty} \int_{\epsilon-i  N}^{\epsilon + i N} \frac{1}{2}(1+2\lambda) e^{\lambda \tau} \int_0^1 G_n(\rho, s; \lambda) f(s) ds d\lambda.
$$
In order to prove the Strichartz estimates, we also need to have bounds for the operators $\dot{T}_{n}$.

We first present two lemmas concerning norm equivalence.
\begin{lemma}\label{lemma:1to5} We have
$$
\|(\cdot)^{-1} f\|_{L^2(\BB^5)} = \|(\cdot) f\|_{L^2(0,1)} \lesssim \|f\|_{H^1(\BB^5)} 
$$
for all $f \in C^1([0,1])$.
\end{lemma}
\begin{proof}
We compute
\begin{align*}
\int_0^1 |f(s)|^2 s^2 ds = & \frac{1}{3} \int_0^1 |f(s)|^2 \PD_s(s^3) ds \\
\lesssim & |f(1)|^2 + \int_0^1 |f(s)|  |f'(s)| s^3 ds \\
\lesssim & \|f\|_{H^1(\BB^5)} + \epsilon \int_0^1 |f(s)|^2 s^2 ds + \frac{1}{\epsilon} \int_0^1 |f'(s)|^2 s^4 ds ,
\end{align*}
where we have used Lemma \ref{lemma:equivnorm}. Since this is true for any $\epsilon>0$, we obtain the desired conclusion.
\end{proof}

\begin{lemma}\label{lemma:5to5} We have
$$
\|(\cdot) f\|_{L^5(\BB^5)} \lesssim \|f\|_{H^1(\BB^5)}
$$
for all $f \in C^1([0,1])$.
\end{lemma}
\begin{proof}
We compute
\begin{align*}
\|(\cdot) f\|_{L^5(\BB^5)} = & \lf( \int_0^1 \lf( |f(\rho)|\rho \rt)^5 \rho^4 d\rho \rt)^\frac{1}{5} \leq  \lf( \int_0^1 \lf(|f(\rho)| \rho\rt) ^5 \rho^2 d\rho \rt)^\frac{1}{5} \\
= & \|(\cdot) f \|_{L^5(\BB^3)} \lesssim \|(\cdot) f \|_{L^6(\BB^3)} \lesssim \|(\cdot) f \|_{H^1(\BB^3)} \\
=  &\lf(\int_0^1 (|f(\rho)|\rho)^2 \rho^2 d\rho \rt)^\frac{1}{2} + \lf( \int_0^1 \lf(\frac{d}{d\rho} (f(\rho) \rho)\rt)^2 \rho^2  d\rho \rt)^\frac{1}{2} \\
\lesssim & \|f\|_{L^2 (\BB^5)} + \|f'\|_{L^2 (\BB^5)} + \|(\cdot) f\|_{L^2(0,1)} \\
\lesssim & \|f\|_{H^1(\BB^5)},
\end{align*}
where we used Sobolev embedding in 3 dimensions $H^1 \hookrightarrow L^6 (\BB^3)$, and Lemma \ref{lemma:1to5}.
\end{proof}

With the extra growth in $\lambda$, we no longer have absolute convergence in the $\lambda$ integral. Hence we must first do an integration by parts in $s$ to obtain the desired decay, and treat the boundary terms separately.

\begin{prop}\label{prop:dotTbound} Let $p\in [2, \infty]$ and $q\in [\frac{10}{3},5]$ be such that $\frac{1}{p}+\frac{5}{q} = \frac{3}{2}$. For $n \in \{ 1,2, \ldots, 6\}$, setting $\dot{T}_n(\tau) f(\rho) := \lim_{\epsilon \to 0+} \dot{T}_{n,\epsilon}(\tau) f(\rho)$, we have the bound
$$
\|\dot{T}_n(\cdot) f\|_{L^p(\RR_+; L^q(\BB^5))} \lesssim \| f \|_{H^1(\BB^5)}
$$
for all $f\in C^1([0,1])$.
\end{prop}

\begin{proof} Let $\lambda=\epsilon+i\omega$ where $\epsilon \in [0, \frac{1}{4}]$ and $\omega\in \RR$.

For $n=\{1,2,3\}$, we have that each $G_n$ is given by
$$
G_n(\rho,s;\lambda) = g_n(\rho;\lambda) 1_{\RR_+}(s-\rho) s (1-s)^{-\frac{1}{2}+\lambda} (2+s(-1+2\lambda)) \gamma_n(\rho,s;\lambda),
$$
where
\begin{align*}
g_1(\rho;\lambda) & = \chi(\rho \lla \omega \rra ) \frac{\varphi_0(\rho;\lambda)}{(3-2\lambda)(1+2\lambda)(1-2\lambda)} ,\\
g_2(\rho;\lambda) & = \lf[ 1-\chi(\rho \lla \omega \rra ) \rt] \frac{\varphi_1(\rho;\lambda)}{(3-2\lambda)(1+2\lambda)(1-2\lambda)}, \\
g_3(\rho;\lambda) & = \lf[ 1-\chi(\rho \lla \omega \rra ) \rt] \frac{\tilde{\varphi}_1(\rho;\lambda)}{(3-2\lambda)(1+2\lambda)(1-2\lambda)}.
\end{align*}
We note that $|g_n(\rho; \lambda)| \lesssim \rho^{-3} \lla \omega \rra^{-2}$. Doing an integration by parts on the $s$-integral, we get
\begin{align}
& \int_0^1 1_{\RR_+}(s-\rho) s (1-s)^{-\frac{1}{2}+\lambda} (2+s(-1+2\lambda)) f(s) \gamma_n(\rho,s;\lambda) ds \nonumber \\
= &  -\frac{2}{1+2 \lambda} \int_\rho^1 \PD_s(1-s)^{\frac{1}{2}+\lambda} s(2+s(-1+2\lambda)) f(s) \gamma_n(\rho,s;\lambda)ds \nonumber \\
= & \frac{-2}{1+2\lambda} (1-\rho)^{\frac{1}{2}+\lambda} \rho(2+\rho(-1+2\lambda))f(\rho) \gamma_n(\rho,\rho;\lambda) \nonumber \\
& + \frac{2}{1+2\lambda} \int_\rho^1 (1-s)^{\frac{1}{2}+\lambda} \PD_s\lf[ s(2+s(-1+2\lambda)) f(s) \gamma_n(\rho,s;\lambda) \rt] ds. \label{eqn:intbyparts1}
\end{align}
Inserting the integral terms back to $\dot{T}_{n,\epsilon}$ we see that the $\lambda$ integral is now absolutely convergent, so limits exist. Now
\begin{align*}
& \PD_s\lf[ s(2+s(-1+2i\omega)) f(s) \gamma_n(\rho,s;i\omega) \rt]  \\
= &  (2+s(-1+2i\omega)) s f(s) \PD_s\gamma_n(\rho,s;\lambda)  +  (2+2s(-1+2i\omega))f(s) \gamma_n(\rho,s;i\omega) \\
& + f'(s) s (2+s(-1+2i\omega))\gamma_n(\rho,s;i\omega),
\end{align*}
and by Lemma \ref{lemma:decomposeG} we see that 
\begin{equation}\label{eqn:dsform}
s \PD_s\gamma_n(\rho,s;\lambda) = O(\rho^0 s^0) O_0( \lla \omega \rra^{-1} ) + O(\rho^0(1-\rho)^0 s^0 (1-s)^0 \lla \omega \rra ^{-2}).
\end{equation}
So we could follow what we did in Proposition \ref{prop:kernels}. Using $\|(\cdot) f\|_{L^2(0,1)} + \|(\cdot)^2 f'\|_{L^2(0,1)} \lesssim  \|f\|_{H^1(\BB^5)}$, we obtain the desired bound. The boundary terms give rise to operators
$$
B_{n,\epsilon}(\tau)f(\rho) := \frac{(1-\rho)^\frac{1}{2} \rho f(\rho)}{2\pi i} \lim_{N\to\infty}\int_{\epsilon-iN}^{\epsilon+i N} e^{\lambda \tau} (1-\rho)^{\lambda} (2+\rho(-1+2\lambda)) \gamma_n(\rho,\rho;\lambda) g_n(\rho;\lambda) d\lambda.
$$
The integrands are bounded by $C\rho^{-3} \lla \omega\rra^{-2}$, we can take the limits
$$
B_n(\tau) f(\rho) := \lim_{\epsilon \to 0} B_{n,\epsilon}(\tau)f(\rho).
$$
For $B_1$ we have
\begin{align*}
B_1(\tau) f(\rho) = & \frac{2 (1-\rho)^\frac{1}{2} \rho f(\rho)}{2\pi } \int_\RR  e^{i\omega( \tau+\log(1-\rho))} \frac{\chi(\rho \lla \omega \rra) \varphi_0(\rho;i\omega)}{(3-2i\omega) (1-2i \omega)(1+2i\omega)} \gamma_1(\rho,\rho;i\omega) d\omega \\
& + \frac{2 (1-\rho)^\frac{1}{2} \rho^2 f(\rho)}{2\pi } \int_\RR  e^{i\omega( \tau+\log(1-\rho))} \frac{\chi(\rho \lla \omega \rra) \varphi_0(\rho;i\omega)}{(3-2i\omega) (-1-2i \omega)} \gamma_1(\rho,\rho;i\omega) d\omega .
\end{align*}
For the two terms, we use the two representations \eqref{eqn:2red} and \eqref{eqn:1red} respectively. Since $\frac{1-2i\omega}{1+2i\omega} = \mathcal{O}_e (\lla \omega\rra^0) + \mathcal{O}(\lla \omega\rra^{-1})$, by Lemma \ref{lemma:o}, we obtain
$$
|B_1(\tau) f(\rho)| \lesssim \rho |f(\rho)|  (1-\rho)^\frac{1}{2} \lla \tau+\log(1-\rho)\rra^{-2} \lesssim \rho|f(\rho)| \lla \tau \rra^{-2}.
$$
For $B_2$ we have
$$
B_2 (\tau) f(\rho)  = \frac{\rho f(\rho) (1-\rho)^{\frac{1}{2}} (1+\rho)^{\frac{1}{2}}}{2\pi } \int_\RR [1-\chi(\rho \lla \omega \rra ]  e^{i \omega (\tau + \log(1-\rho)-\log(1+\rho))} \rho^{-1} \mathcal{O}(\lla \omega \rra^{-2}) d\omega,
$$
and using Lemma \ref{lemma:rho1} we have
$$
|B_2(\tau) f(\rho)| \lesssim \rho |f(\rho)|  (1-\rho)^\frac{1}{2} \lla \tau+\log(1-\rho)\rra^{-2} \lesssim \rho|f(\rho)|\lla \tau \rra^{-2}.
$$
Similarly, 
$$
B_3 (\tau) f(\rho) = \frac{\rho f(\rho) (1-\rho) (1+\rho)^{\frac{1}{2}}}{2\pi} \int_\RR [1-\chi(\rho \lla \omega \rra ]  e^{i \omega \tau} \rho^{-1} \mathcal{O}(\lla \omega \rra^{-2}) d\omega,
$$
so we also have $|B_3(\tau) f(\rho)| \lesssim \rho|f(\rho)| \lla \tau \rra^{-2}$. Then by Lemma \ref{lemma:5to5} we obtain the bounds $\|B_n(\tau) f\|_{L^5(\BB^5)} \lesssim \lla \tau \rra^{-2} \|f\|_{H^1(\BB^5)}$ for $n=\{1,2,3\}$.

For $n=4$, we decompose $\dot{T}_{4,\epsilon}$ into two parts according to $\varphi_1(\rho;\lambda)$. We have
\begin{align*}
 \dot{T}_{4,\epsilon}(\tau) f(\rho) =  & \frac{1}{4 \pi i} \lim_{N\to\infty} \int_{\epsilon-i  N}^{\epsilon + i N}  \frac{e^{\lambda \tau}\varphi_1(\rho;\lambda)}{(3-2\lambda)(1-2\lambda)} \\
& \times  \int_0^1 1_{\RR_+}(\rho-s) \chi(s\lla \omega \rra)s^4(1-s^2)^{-\frac{1}{2}+\lambda} \varphi_0(s;\lambda)  \gamma_4(\rho,s;\lambda) f(s) ds d\lambda \\
= & I_{1,\epsilon}(\tau, \rho) + I_{2,\epsilon}(\tau, \rho) ,
\end{align*}
where
\begin{align*}
I_{1,\epsilon}(\tau, \rho) = & \frac{\rho^{-3}(1+\rho)^\frac{1}{2}}{4 \pi i} \lim_{N\to\infty} \int_{\epsilon-i  N}^{\epsilon + i N}  \frac{  (1+\rho)^{- \lambda}e^{\lambda \tau}}{(3-2\lambda)(1-2\lambda)} \\
& \times  \int_0^1 1_{\RR_+}(\rho-s) \chi(s\lla \omega \rra)s^4(1-s^2)^{-\frac{1}{2}+\lambda} \varphi_0(s;\lambda)  \gamma_4(\rho,s;\lambda) f(s) ds d\lambda,\\
I_{2,\epsilon}(\tau, \rho) = & \frac{\rho^{-2}(1+\rho)^\frac{1}{2}}{4 \pi i} \lim_{N\to\infty} \int_{\epsilon-i  N}^{\epsilon + i N}  \frac{  (1+\rho)^{- \lambda}e^{\lambda \tau}}{(3-2\lambda)} \\
& \times  \int_0^1 1_{\RR_+}(\rho-s) \chi(s\lla \omega \rra)s^4(1-s^2)^{-\frac{1}{2}+\lambda} \varphi_0(s;\lambda)  \gamma_4(\rho,s;\lambda) f(s) ds d\lambda.
\end{align*}
The integrals in $I_1$ are already absolutely convergent, so we can take the limit and use Fubini,
\begin{align*}
 I_1(\tau,\rho) = \lim_{\epsilon \to 0+} I_{1,\epsilon}(\rho)  = & \rho^{-3}(1+\rho)^{\frac{1}{2}} \int_0^\rho s^4 (1-s^2)^{-\frac{1}{2}} f(s) \\
 & \times \int_\RR \frac{e^{i\omega(\tau-\log(1+\rho)+\log(1-s^2))}}{(3-2i \omega)(1-2i \omega)} \chi(s\lla \omega\rra)\varphi_0(s;i \omega) \gamma_4(\rho,s;i\omega) d\omega ds.
\end{align*}
This is similar to the $G_4$ case in Proposition \ref{prop:kernels}. We use the expression \eqref{eqn:1red} for $\varphi_0(s;i \omega)$ and obtain
\begin{align*}
\|I_1(\tau,\cdot)\|_{L^5(\BB^5)} \lesssim \|(\cdot)f\|_{L^2(0,1)}  \lla \tau \rra^{-2}\lesssim \|f\|_{H^1(\BB^5)} \lla \tau \rra^{-2}.
\end{align*}
For $I_{2,\epsilon}$, the $s$-dependent part is given by
$$
1_{\RR_+}(\rho-s) \chi(s\lla \omega \rra) s\lf[(1-s)^{-\frac{1}{2} + \lambda} (2+s(-1+2\lambda)) - (1+ s)^{-\frac{1}{2} + \lambda} (2+s(1- 2\lambda)) \rt] \gamma_4(\rho,s;\lambda).
$$
Now we write
\begin{align*}
& (1-s)^{-\frac{1}{2} + \lambda} (2+s(-1+2\lambda)) - (1+ s)^{-\frac{1}{2} + \lambda} (2+s(1- 2\lambda)) \\
& =  2 s^2 \PD_s\lf[\frac{(1+s)^{\frac{1}{2} + \lambda} - (1-s)^{\frac{1}{2} + \lambda}}{s} \rt]
\end{align*}
and do an integration by parts in the $s$ integral. We obtain
\begin{align*}
& \int_0^1 1_{\RR_+}(\rho-s) \chi(s\lla \omega \rra) f(s) \gamma_4(\rho,s;\lambda) \\
& \times s\lf[(1-s)^{-\frac{1}{2} + \lambda} (2+s(-1+2\lambda)) - (1+ s)^{-\frac{1}{2} + \lambda} (2+s(1- 2\lambda)) \rt]  ds \\
= & 2 \int_0^\rho \chi(s\lla \omega \rra) s^3 f(s) \gamma_4(\rho,s;\lambda)\PD_s\lf[\frac{(1+s)^{\frac{1}{2} + \lambda} - (1-s)^{\frac{1}{2} + \lambda}}{s} \rt] ds \\
= & \chi(\rho \lla \omega \rra)\rho^2  f(\rho) \gamma_4(\rho,\rho; \lambda) \lf( (1+\rho)^{\frac{1}{2} + \lambda} - (1-\rho)^{\frac{1}{2} + \lambda} \rt) \\
& + \int_0^\rho \frac{(1+s)^{\frac{1}{2} + \lambda} - (1-s)^{\frac{1}{2} + \lambda}}{s} \PD_s\lf[  \chi(s\lla \omega \rra) s^3 f(s) \gamma_4(\rho,s;\lambda) \rt] ds.
\end{align*}
For the integral term, reinserting back into $I_{2,\epsilon}$, we see that the $\lambda$ integrals are now absolutely convergent, so we can take the limits and use Fubini. When the derivative does not fall on the cut-off $\chi(s\lla \omega \rra)$, notice that
$$
\PD_s\lf[  s^3 f(s) \gamma_4(\rho,s;\lambda) \rt] = 3s^2 f(s)  \gamma_4(\rho,s;\lambda) +  s^3 f'(s)  \gamma_4(\rho,s;\lambda) + s^3 f(s) \PD_s \gamma_4(\rho,s;\lambda),
$$
and again we use that $s\PD_s\gamma_4(\rho,s;\lambda)$ has the form as in \eqref{eqn:dsform}. Moreover, we write
\begin{equation}\label{eqn:3red}
\frac{(1+s)^{\frac{1}{2} +i\omega} - (1-s)^{\frac{1}{2} + i\omega}}{s} = \frac{1}{2}(1+2i\omega) \int_{-1}^1 (1+st)^{-\frac{1}{2}+i\omega} dt.
\end{equation}
Now $\frac{1+2i\omega}{3-2i\omega} = \mathcal{O}_e (\lla \omega\rra^0) + \mathcal{O}(\lla \omega\rra^{-1})$, so we can proceed as in the $G_4$ case in Proposition \ref{prop:kernels} to obtain the bound by $\|(\cdot) f\|_{L^2(0,1)} + \|(\cdot)^2 f'\|_{L^2(0,1)}$. When the derivative falls on the cut-off, we have to look at the term
$$
\rho^{-2}(1+\rho)^\frac{1}{2} s^2 (1\pm s)^{\frac{1}{2}} f(s)  \int_\RR e^{i\omega(\tau-\log(1+\rho)+\log(1\pm s) )} \chi'(s\lla \omega\rra) \mathcal{O}(\lla \omega \rra^{-1}) d\omega .
$$
Since $\rho\geq s$, and on the cut-off we have also $ s\lla \omega \rra \simeq 1$, so
\begin{align*}
& \rho^{-2}(1+\rho)^\frac{1}{2} s^2 (1\pm s)^{\frac{1}{2}} | f(s)| \lf| \int_\RR e^{i\omega(\tau-\log(1+\rho)+\log(1\pm s) )} \chi'(s\lla \omega\rra) \mathcal{O}(\lla \omega \rra^{-1}) d\omega \rt| \\
\lesssim &  \frac{s^2}{\rho^2}|f(s)| \lf| s^{-1}\int_\RR e^{i\omega(\tau-\log(1+\rho)+\log(1\pm s) )} \chi'(s\lla \omega\rra) \mathcal{O}(\lla \omega \rra^{-2}) d\omega  \rt| \\
\lesssim & \frac{s^2}{\rho^2}|f(s)| \lla \tau+ \log(1-s) \rra^{-2} \lesssim  \frac{s^2}{\rho^2}(1-s)^{-\frac{1}{2}} |f(s)| \lla \tau + \log(1-s) \rra^{-2} 
\end{align*}
using the same logic as in the proof of Lemma \ref{lemma:rho1}. Then again we obtain the bound $\|(\cdot) f\|_{L^2(0,1)}$ when taking the $L^2(\RR_+;L^5(\BB^5))$ norm. On the other hand, for the boundary term we can again take the limit
\begin{align*}
& B_4(\tau) f(\rho) \\
 := & \frac{(1+\rho)^\frac{1}{2} f(\rho) }{4 \pi } \lim_{\epsilon \to 0+}   \lim_{N\to\infty} \int_{\epsilon-i  N}^{\epsilon + i N}  \frac{ e^{\lambda (\tau+\log(1+\rho))}}{(3-2\lambda)} \\
 & \qquad \qquad\qquad\qquad \times  \chi(\rho \lla \omega \rra)\gamma_4(\rho,\rho; \lambda) \lf( (1+\rho)^{\frac{1}{2} + \lambda} - (1-\rho)^{\frac{1}{2} + \lambda} \rt) d\lambda \\
= & \frac{(1+\rho)^\frac{1}{2} f(\rho) }{4 \pi } \int_\RR  \frac{ e^{i\omega( \tau+\log(1+\rho))}}{(3-2 i \omega)} \chi(\rho \lla \omega \rra) \gamma_4(\rho,\rho;  i \omega) \lf( (1+\rho)^{\frac{1}{2} +  i \omega} - (1-\rho)^{\frac{1}{2} +  i \omega} \rt) d \omega.
\end{align*}
By \eqref{eqn:3red} and Lemma \ref{lemma:o} we obtain $|B_4(\tau)f(\rho)| \lesssim \rho|f(\rho)|\lla \tau \rra ^{-2}$, and again Lemma \ref{lemma:5to5} gives that $\|B_4 (\tau) f\|_{L^5(\BB^5)} \lesssim  \|f\|_{H^1(\BB^5)} \lla \tau \rra^{-2}$.

For $n=\{ 5, 6\}$, the $s$ dependent part is 
$$
1_{\RR_+}(\rho-s)[1-\chi(s\lla \omega\rra)] s(1 \pm s)^{-\frac{1}{2}+\lambda}(2\pm s(1-2\lambda) \gamma_n(\rho,s;\lambda),
$$
while the $s$ independent part splits into two terms,
$$
\frac{e^{\lambda \tau}\varphi_1(\rho;\lambda)}{(3-2\lambda)(1-2\lambda)} = \frac{2e^{\lambda \tau} \rho^{-3}(1+\rho)^{\frac{1}{2}-\lambda}}{(3-2\lambda)(1-2\lambda)} - \frac{e^{\lambda \tau} \rho^{-3}(1+\rho)^{\frac{1}{2}-\lambda}}{(3-2\lambda)}.
$$
Again the first part already has decay $\lla \omega \rra^{-2}$, so we can follow what we did for $G_5$ and $G_6$ in Proposition \ref{prop:kernels} to obtain the bounds by $\|(\cdot)f\|_{L^2(0,1)}\lesssim \|f\|_{H^1(\BB^5)}$. For the second part we do an integration by parts in $s$,
\begin{align}
&\int_0^1 1_{\RR_+}(\rho-s)(1 \pm s)^{-\frac{1}{2}+\lambda} [1-\chi(s\lla \omega\rra)] s (2\pm s(1-2\lambda)) f(s) \gamma_n(\rho,s;\lambda) ds \nonumber \\
=& \frac{ \pm 2}{1+2\lambda}\int_0^\rho  \PD_s (1 \pm s)^{\frac{1}{2}+\lambda} [1-\chi(s\lla \omega\rra)] s(2\pm s(1-2\lambda)) f(s) \gamma_n(\rho,s;\lambda) ds \nonumber \\
=& \frac{ \pm 2}{1+2\lambda}(1 \pm \rho)^{\frac{1}{2}+\lambda} [1-\chi(\rho\lla \omega\rra)] \rho(2\pm \rho(1-2\lambda)) f(\rho) \gamma_n(\rho,\rho;\lambda) \nonumber \\
& -\frac{\pm 2}{1+2\lambda} \int_0^\rho(1 \pm s)^{\frac{1}{2}+\lambda} \PD_s \lf[[1-\chi(s\lla \omega\rra)] s(2\pm s(1-2\lambda)) f(s) \gamma_n(\rho,s;\lambda)  \rt] ds. \label{eqn:intbyparts2}
\end{align}
Using that $s \PD_s\gamma_n(\rho,s;\lambda)$ has the form \eqref{eqn:dsform}, the integral part is treated the same way as before and is bounded by $\|(\cdot) f\|_{L^2(0,1)} + \|(\cdot)^2 f'\|_{L^2(0,1)} \lesssim  \|f\|_{H^1(\BB^5)}$. The boundary term is given by
\begin{align*}
& B_n(\tau)f(\rho)\\
= & \frac{(1+\rho)^\frac{1}{2}(1\pm \rho)^\frac{1}{2} f(\rho) }{4 \pi } \lim_{\epsilon \to 0+}   \lim_{N\to\infty} \int_{\epsilon-i  N}^{\epsilon + i N} e^{\lambda(\tau-\log(1+\rho)+\log(1\pm \rho))}[1-\chi(\rho\lla \omega\rra)] \mathcal{O}(\rho^0 \lla \omega \rra^{-2}) d\lambda \\
=&\frac{(1+\rho)^\frac{1}{2}(1\pm \rho)^\frac{1}{2} f(\rho) }{4 \pi }  \int_\RR e^{i\omega(\tau-\log(1+\rho)+\log(1\pm \rho))}[1-\chi(\rho\lla \omega\rra)] \mathcal{O}(\rho^0 \lla \omega \rra^{-2}) d\omega.
\end{align*}
So Lemma \ref{lemma:rho1} gives that 
$$
|B_n(\tau)f(\rho)| \lesssim \rho |f(\rho)| (1\pm \rho)^{\frac{1}{2}} \lla \tau-\log(1+\rho)+\log(1\pm \rho) \rra^{-2} \lesssim \rho |f(\rho)| \lla \tau \rra^{-2},
$$
and as before we have $\|B_n(\tau)f\|_{L^5(\BB^5)} \lesssim \|f\|_{H^1(\BB^5)}  \lla \tau \rra^{-2}$.
\end{proof}

\subsection{Strichartz estimates}

\begin{theorem}\label{thm:strichartz} Let $p\in [2, \infty]$ and $q\in [\frac{10}{3},5]$ such that $\frac{1}{p}+\frac{5}{q} = \frac{3}{2}$. Then we have the bound
$$
\lf\Vert \lf[\mathbf{S}(\cdot) (\mathbf{I}-\mathbf{P})\mathbf{f} \rt]_1 \rt\Vert _{L^p(\RR_+,L^q(\BB^5))} \lesssim \|(\mathbf{I}-\mathbf{P})\mathbf{f} \|_{\mathcal{H}}
$$
for all $\mathbf{f} \in \mathcal{H}$. And hence we also have
$$
\lf\Vert  \int_0^\tau  \lf[\mathbf{S}(\tau -\sigma) (\mathbf{I}-\mathbf{P}) \mathbf{h}(\sigma,\cdot) \rt]_1 d\sigma  \rt\Vert_{L_\tau^p(\RR_+;L^q(\BB^5))} \lesssim \|(\mathbf{I}-\mathbf{P})\mathbf{h} \|_{L^1(\RR_+,\mathcal{H})} 
$$
or all $\mathbf{h}\in C([0,\infty),\mathcal{H})\cap L^1(\RR_+,\mathcal{H})$
\end{theorem}
\begin{proof}
From Equation \eqref{eqn:semigpdecomp}, we have
$$
[ \mathbf{S}(\tau) \tilde{\mathbf{f}} ]_1  = [ \mathbf{S}_0(\tau) \tilde{\mathbf{f}} ]_1 (\rho)+\sum_{n=1}^6 \lf[T_n(\tau)\lf(|\cdot|\tilde{f}_1' +2 \tilde{f}_1+\tilde{f}_2 \rt) +  \dot{T}_n(\tau) \tilde{f}_1 \rt].
$$
Then by using Proposition \ref{prop:freestrichartz}, \ref{prop:Tbound}, \ref{prop:dotTbound}, the stated result follows by the same argument as \cite[Theorem 4.1]{donn2017}.
\end{proof}

\section{Improved energy estimates}

We also need to take $\epsilon \to 0+$ in Lemma \ref{lemma:growthbound}. But since the constant depends on $\epsilon$, we cannot directly do this. Instead, we show explicitly in this section that the constant $C_\epsilon$ can be chosen independent of $\epsilon$, allowing us to take $\epsilon = 0$.

\subsection{Preliminaries}

We first introduce the following result which allows us to interchange a limit and an integral, which is not absolutely convergent. This is in the same spirit as \cite[Lemma 5.1]{donn2017}.
\begin{lemma}\label{lemma:limitinterchange}
Let $f$ be a function of two variables $\epsilon \in [0,\frac{1}{4})$ and $\omega \in \RR$ satisfying the properties that
\begin{itemize}
\item $f(\epsilon, \cdot)$ is an odd function for every fixed $\epsilon$ ,
\item $f \in C^1 ([0,\frac{1}{4}) \times \RR)$ ,
\item $|f(\epsilon,\omega)| \leq C_1 \lla \omega \rra^{-1}$ and $|\PD_\omega f(\epsilon,\omega) | \leq C_2 \lla \omega \rra ^{-2}$ with $C_1, C_2$ constants independent of $\epsilon$.
\end{itemize}
Then
$$
\lim_{\epsilon \to 0+ } \int_\RR e^{i a \omega} f(\epsilon,\omega) d\omega = \int_{\RR} e^{i a \omega} f(0,\omega) d\omega
$$
for all $a\in \RR \backslash \{0\}$.
\end{lemma}
\begin{proof} We do an integration by parts,
\begin{align*}
\int_\RR e^{i a \omega} f(\epsilon,\omega) d\omega 
 =  -\frac{1}{i a} \int_\RR e^{i a \omega} \PD_\omega f(\epsilon,\omega) d\omega.
\end{align*}
This integral is now absolutely convergent, so we can take the limit in $\epsilon$ by dominated convergence, then undo the integration by parts,
\begin{align*}
\lim_{\epsilon \to 0+ } \int_\RR e^{i a \omega} f(\epsilon,\omega) d\omega 
= & \lim_{\epsilon \to 0+ } -\frac{1}{i a} \int_\RR e^{i a \omega} \PD_\omega f(\epsilon,\omega) d\omega \\
= & -\frac{1}{i a}  \int_\RR e^{i a \omega} \PD_\omega f(0,\omega) d\omega \\
= & \int_{\RR} e^{i a \omega} f(0,\omega) d\omega,
\end{align*}
as stated.
\end{proof}

\subsection{Decomposition}

Let $\tilde{\mathbf{f}} \in \rg(\mathbf{I}-\mathbf{P})\cap C^2\times C^1([0,1])$, we have
$$
\PD_\rho[\mathbf{S}(\tau)\tilde{\mathbf{f}}]_1(\rho) = \frac{1}{2\pi i}\lim_{N\to\infty} \int_{\epsilon- i N}^{\epsilon- i N} e^{\lambda \tau} \int_0^1 G'(\rho,s;\lambda) F_{\lambda}(s) ds,
$$
where
$$
G'(\rho,s;\lambda)=  \frac{s^4(1-s^2)^{-\frac{1}{2}+\lambda}}{(3-2\lambda)(1+2\lambda)(1-2\lambda)w_0(\lambda)} \begin{cases} \PD_\rho u_0(\rho;\lambda) u_1(s;\lambda) \text{ if }   \rho \leq s \\ \PD_\rho u_1(\rho;\lambda) u_0(s;\lambda)  \text{ if } \rho \geq s  \end{cases}
$$
and $F_{\lambda}(s)=s \tilde{f}'_1(s) + \lf(\lambda+\frac{5}{2} \rt)\tilde{f}_1(s) + \tilde{f}_2(s)$. The free part is given by
$$
G'_0(\rho,s;\lambda)=  \frac{s^4(1-s^2)^{-\frac{1}{2}+\lambda}}{(3-2\lambda)(1+2\lambda)(1-2\lambda)} \begin{cases} \PD_\rho \varphi_0(\rho;\lambda) \varphi_1(s;\lambda) \text{ if }   \rho \leq s \\ \PD_\rho \varphi_1(\rho;\lambda) \varphi_0(s;\lambda)  \text{ if } \rho \geq s  \end{cases}.
$$
\begin{lemma}We have the decomposition
$$
G'(\rho,s;\lambda) = G'_0(\rho,s;\lambda) + \sum_{n=1}^6\rho^{-1} \tilde{G}_n(\rho,s;\lambda) + \sum_{n=1}^6 G'_n(\rho,s;\lambda),
$$
where $\tilde{G}_n$ has the same form as in Lemma \ref{lemma:decomposeG}, and
\begin{align*}
G'_1(\rho,s;\lambda) & =  1_{\RR_+}(s-\rho) \chi(\rho\lla \omega\rra) \frac{s^4(1-s^2)^{-\frac{1}{2}+\lambda}}{1+2\lambda}\lf( g_1(\rho;\lambda) - \tilde{g}_1(\rho;\lambda) \rt) \varphi_1(s;\lambda) \gamma'_1(\rho,s ;\lambda), \\
G'_2(\rho,s;\lambda) & =  1_{\RR_+}(s-\rho) [1-\chi(\rho\lla \omega\rra)] \frac{s^4(1-s^2)^{-\frac{1}{2}+\lambda}}{1+2\lambda}g_1(\rho;\lambda) \varphi_1(s;\lambda) \gamma'_2(\rho,s ;\lambda),  \\
G'_3(\rho,s;\lambda) & =  1_{\RR_+}(s-\rho) [1-\chi(\rho\lla \omega\rra)] \frac{s^4(1-s^2)^{-\frac{1}{2}+\lambda}}{1+2\lambda}\tilde{g}_1(\rho;\lambda) \varphi_1(s;\lambda) \gamma'_3(\rho,s ;\lambda),  \\
G'_4(\rho,s;\lambda) & =  1_{\RR_+}(\rho - s)  \chi(s \lla \omega\rra)  \frac{s^4(1-s^2)^{-\frac{1}{2}+\lambda}}{1+2\lambda} g_1(\rho;\lambda)\varphi_0(s;\lambda) \gamma'_4(\rho,s ;\lambda), \\
G'_5(\rho,s;\lambda) & =  1_{\RR_+}(\rho - s) [1- \chi(s \lla \omega\rra) ] \frac{s^4(1-s^2)^{-\frac{1}{2}+\lambda}}{1+2\lambda} g_1(\rho;\lambda) \varphi_1(s;\lambda) \gamma'_5(\rho,s ;\lambda), \\
G'_6(\rho,s;\lambda) & =  1_{\RR_+}(\rho - s) [1- \chi(s \lla \omega\rra) ]  \frac{s^4(1-s^2)^{-\frac{1}{2}+\lambda}}{1+2\lambda}g_1(\rho;\lambda) \tilde{\varphi}_1(s;\lambda) \gamma'_6(\rho,s ;\lambda) ,
\end{align*}
with
\begin{align*}
g_1(\rho;\lambda) & = \rho^{-2}(1+\rho)^{-\frac{1}{2}-\lambda}, \\
 \tilde{g}_1(\rho;\lambda) & =\rho^{-2}(1-\rho)^{-\frac{1}{2}-\lambda},
\end{align*}
and
$$
\gamma'_n (\rho,s ;\lambda)  =  \mathcal{O}(\rho^0 s^0) \mathcal{O}_0( \lla \omega \rra^{-1} ) + \mathcal{O}((1-\rho)^0 s^0 \lla \omega \rra^{-2}) + \mathcal{O}(\rho^0 (1-s) \lla \omega \rra ^{-2}).
$$
\end{lemma}

\begin{proof} First of all, differentiating explicitly we have
\begin{align*}
\PD_\rho \varphi_1(\rho;\lambda) &= -3\rho^{-1} \varphi_1(\rho;\lambda) -\frac{1}{2}(1-2\lambda)(3-2\lambda)\rho^{-2} (1+\rho)^{-\frac{1}{2}-\lambda}, \\
\PD_\rho \tilde{\varphi}_1(\rho;\lambda) &= -3\rho^{-1} \tilde{\varphi}_1(\rho;\lambda)-\frac{1}{2}(1-2\lambda)(3-2\lambda)\rho^{-2} (1-\rho)^{-\frac{1}{2}-\lambda}, \\
\PD_\rho \varphi_0(\rho;\lambda) & = \PD_\rho \varphi_1(\rho;\lambda)-\PD_\rho \tilde{\varphi}_1(\rho;\lambda)  \\
&= -3\rho^{-1} \varphi_0(\rho;\lambda)-\frac{1}{2}(1-2\lambda)(3-2\lambda)\rho^{-2} \lf( (1+\rho)^{-\frac{1}{2}-\lambda}-(1-\rho)^{-\frac{1}{2}-\lambda} \rt).
\end{align*}
Now, using Equation \eqref{eqn:u1pert} we have
\begin{align*}
& \PD_\rho u_1(\rho;\lambda)\\
 = & \PD_\rho \varphi_1 (\rho;\lambda) [1+\mathcal{O}(\rho^0)\mathcal{O}_o(\lla \omega\rra^{-1})+\mathcal{O}(\rho^0 (1-\rho)\lla \omega\rra^{-2})] \\
& + \varphi_1(\rho;\lambda)[\mathcal{O}(\rho^{-1})\mathcal{O}_o(\lla \omega\rra^{-1})+\mathcal{O}(\rho^{-1} (1-\rho)^0 \lla \omega\rra^{-2})] \\
= & \PD_\rho \varphi_1(\rho;\lambda)  + \lf(  -3\rho^{-1} \varphi_1(\rho;\lambda)-\frac{1}{2}(1-2\lambda)(3-2\lambda)\rho^{-2} (1+\rho)^{-\frac{1}{2}-\lambda}\rt)\\
& \times [\mathcal{O}(\rho^0)\mathcal{O}_o(\lla \omega\rra^{-1})+\mathcal{O}(\rho^0(1-\rho)\lla \omega\rra^{-2})] \\
& + \rho^{-1}\varphi_1(\rho;\lambda)[\mathcal{O}(\rho^0)\mathcal{O}_o(\lla \omega\rra^{-1})+\mathcal{O}(\rho^0 (1-\rho)^0\lla \omega\rra^{-2})] \\
= & \PD_\rho \varphi_1(\rho;\lambda) + \rho^{-1}\varphi_1(\rho;\lambda)[\mathcal{O}(\rho^0)\mathcal{O}_o(\lla \omega\rra^{-1})+\mathcal{O}(\rho^0 (1-\rho)^0\lla \omega\rra^{-2})] \\
& - (1-2\lambda)(3-2\lambda)\rho^{-2}(1+\rho)^{-\frac{1}{2}-\lambda}[\mathcal{O}(\rho^0)O_o(\lla \omega\rra^{-1})+\mathcal{O}(\rho^0 (1-\rho)\lla \omega\rra^{-2})] .
\end{align*}
From Lemma \ref{lemma:pertsol0}, we have
\begin{align*}
& \chi(\rho\lla \omega\rra)\PD_\rho u_0(\rho;\lambda) \\
= & \chi(\rho\lla \omega\rra)\PD_\rho \varphi_0 (\rho;\lambda)[1+\mathcal{O}(\rho^2\lla \omega\rra^0)] +\chi(\rho\lla \omega\rra) \varphi_0(\rho;\lambda)\mathcal{O}(\rho \lla \omega\rra ) \\
= & \chi(\rho\lla \omega\rra)  \PD_\rho \varphi_0 (\rho;\lambda) +  \chi(\rho\lla \omega\rra) \rho^{-1} \varphi_0(\rho;\lambda)\mathcal{O}(\rho^0 \lla \omega\rra^{-2})   \\
& -\chi(\rho\lla \omega\rra)  \frac{1}{2}(1-2\lambda)(3-2\lambda)\rho^{-2} \lf( (1+\rho)^{-\frac{1}{2}-\lambda}-(1-\rho)^{-\frac{1}{2}-\lambda} \rt) \mathcal{O}(\rho^0\lla \omega\rra^{-2}) \\
= & \chi(\rho\lla \omega\rra) \PD_\rho \varphi_0 (\rho;\lambda) + \chi(\rho\lla \omega\rra)\rho^{-1} \varphi_0(\rho;\lambda) \mathcal{O}(\rho^0 \lla \omega \rra ^{-2}) \\
&  -\chi(\rho\lla \omega\rra) (1-2\lambda)(3-2\lambda)\rho^{-2}  \lf((1+\rho)^{-\frac{1}{2}-\lambda} -(1-\rho)^{-\frac{1}{2}-\lambda}   \rt)  \mathcal{O}(\rho^0\lla \omega\rra^{-2}).
\end{align*}
From Lemma \ref{lemma:pertsol0all}, we have
\begin{align*}
& [1- \chi(\rho\lla \omega\rra)] \PD_\rho u_0(\rho;\lambda) \\
= &  [1- \chi(\rho\lla \omega\rra)] \PD_\rho \varphi_1(\rho;\lambda) [1+\mathcal{O}(\rho^0)\mathcal{O}_o(\lla \omega\rra^{-1}) + \mathcal{O}(\rho^0 (1-\rho)\lla \omega\rra^{-2})+ \mathcal{O}(\lla \omega\rra^{-2})] \\
& + [1- \chi(\rho\lla \omega\rra)] \varphi_1(\rho;\lambda) [\mathcal{O}(\rho^{-1} )\mathcal{O}_o(\lla \omega\rra^{-1}) + \mathcal{O}(\rho^{-1} (1-\rho)^0\lla \omega\rra^{-2})] \\
& - [1- \chi(\rho\lla \omega\rra)] \PD_\rho \tilde{\varphi}_1(\rho;\lambda) [1+\mathcal{O}(\rho^0)\mathcal{O}_o(\lla \omega\rra^{-1}) + \mathcal{O}(\rho^0 (1-\rho)\lla \omega\rra^{-2})+ \mathcal{O}(\lla \omega\rra^{-2})] \\
& - [1- \chi(\rho\lla \omega\rra)] \tilde{\varphi}_1(\rho;\lambda) [\mathcal{O}(\rho^{-1} )\mathcal{O}_o(\lla \omega\rra^{-1}) + \mathcal{O}(\rho^{-1}(1-\rho)^0\lla \omega\rra^{-2})]  \\
= & [1- \chi(\rho\lla \omega\rra)] \PD_\rho \varphi_0(\rho;\lambda) \\
& + [1- \chi(\rho\lla \omega\rra)]\rho^{-1} \varphi_1(\rho;\lambda) [\mathcal{O}(\rho^0)O_o(\lla \omega\rra^{-1}) + \mathcal{O}(\rho^0 (1-\rho)^0\lla \omega \rra^{-2})] \\
& - [1- \chi(\rho\lla \omega\rra)]\rho^{-1} \tilde{\varphi}_1(\rho;\lambda) [\mathcal{O}(\rho^0)O_o(\lla \omega\rra^{-1}) + \mathcal{O}(\rho^0 (1-\rho)^0\lla \omega \rra^{-2})] \\
& + [1- \chi(\rho\lla \omega\rra)](1-2\lambda)(3-2\lambda)\rho^{-2} (1+\rho)^{-\frac{1}{2}-\lambda}[\mathcal{O}(\rho^0)\mathcal{O}_o(\lla \omega\rra^{-1}) + \mathcal{O}(\rho^0 (1-\rho)^0\lla \omega\rra^{-2})] \\
& -[1- \chi(\rho\lla \omega\rra)](1-2\lambda)(3-2\lambda)\rho^{-2} (1-\rho)^{-\frac{1}{2}-\lambda}[\mathcal{O}(\rho^0)\mathcal{O}_o(\lla \omega\rra^{-1}) + \mathcal{O}(\rho^0 (1-\rho)^0\lla \omega\rra^{-2})] .
\end{align*}

Combining with \eqref{eqn:u1pert} and \eqref{eqn:u0pert}, we obtain the desired decomposition.
\end{proof}

Now for $n=\{1,2, \ldots,6\}$, $\tau >0$, $\rho \in (0,1)$, and $f\in C^1([0,1])$, we define operators
\begin{align*}
 \tilde{S}_{n,\epsilon} (\tau) f(\rho) : =& \frac{1}{2\pi i} \lim_{N \to \infty} \int_{\epsilon- i N}^{\epsilon + i N} e^{\lambda \tau} \int_0^1  \rho^{-1} \tilde{G}_n(\rho ,s ; \lambda) f(s) ds d\lambda  \\
= & \frac{1}{2\pi }  \int_{\RR} e^{(\epsilon+i \omega) \tau} \int_0^1 \rho^{-1} \tilde{G}_n(\rho ,s ; \epsilon+i \omega) f(s) ds d\omega. 
\end{align*}
Then by dominated convergence and Fubini, we have
$$
\tilde{S}_n(\tau) f(\rho):= \lim_{\epsilon \to 0+} \tilde{S}_{n,\epsilon} (\tau) f(\rho) = \frac{1}{2\pi} \int_0^1f(s) \int_{\RR} e^{i \omega \tau}  \rho^{-1} \tilde{G}_n(\rho ,s ; i \omega) d\omega ds .
$$
On the other hand, for $n=\{1,2, \cdots, 6\}$, $\tau >0$, $\rho \in (0,1)$, and $f\in C^1([0,1])$, we define
\begin{align*}
S'_{n,\epsilon} (\tau) f(\rho) = & \frac{1}{2\pi i} \lim_{N \to \infty} \int_{\epsilon- i N}^{\epsilon + i N} e^{\lambda \tau} \int_0^1 G'_n(\rho ,s ; \lambda) f(s) ds d\lambda  \\
= & \frac{1}{2\pi }  \int_{\RR} e^{(\epsilon+i \omega) \tau} \int_0^1 G'_n(\rho ,s ; \epsilon+i \omega) f(s) ds d\omega,
\end{align*}
and we also want to take $\epsilon \to 0+$. Indeed, by the same argument as \cite[Lemma 5.3]{donn2017}, we have
\begin{align*}
S'_n (\tau) f(\rho) := \lim_{\epsilon \to 0+}  S'_{n,\epsilon} (\tau) f(\rho) = & \frac{1}{2\pi }  \int_{\RR} e^{i \omega \tau} \int_0^1 G'_n(\rho ,s ; i \omega) f(s) ds d\omega \\
= & \frac{1}{2\pi }   \int_0^1 f(s) \int_{\RR} e^{i \omega \tau} G'_n(\rho ,s ; i \omega)   d\omega ds.
\end{align*}

Our goal now is to obtain bounds for $\|\tilde{S}_n(\tau) f\|_{L^2(\BB^5)}$ and $\|S'_n (\tau) f\|_{L^2(\BB^5)}$. 

\subsection{Kernel estimates}

We have bounds for the kernel of the operators $\tilde{S}_n$.

\begin{lemma}\label{lemma:gtilde} We have
$$
\lf| \int_\RR e^{i \omega \tau}  \tilde{G}_n(\rho ,s ; i \omega) d\omega \rt| \lesssim \rho^{-1}s^2 (1-s)^{-\frac{1}{2}}\lla \tau  + \log(1-s) \rra^{-2},
$$
for all $\tau>0$, $\rho,s\in(0,1)$, and $n=\{1, 2, \ldots,6 \}$.
\end{lemma}

\begin{proof}
Since the form of $\tilde{G}_n$ is the same as $G_n$, the treatment here is similar to Proposition \ref{prop:kernels}. For $\tilde{G}_1$, we have
$$
\int_\RR e^{i \omega \tau} \tilde{G}_1(\rho ,s ; \omega) d\omega = 2 I_1(\rho,s,\tau) + I_2(\rho,s\tau)
$$
where
\begin{align*}
I_1(\rho,s,\tau) = &  1_{\RR_+}(s-\rho)s(1-s)^{-\frac{1}{2}} \int_\RR \chi(\rho\lla \omega \rra) \frac{e^{i \omega (\tau+\log(1-s))} \varphi_0(\rho;i \omega)}{(3-2i\omega)(1+2i\omega)(1-2i\omega)} \tilde{\gamma}_1(\rho,s;i\omega) d\omega , \\
I_2 (\rho,s,\tau) = &  1_{\RR_+}(s-\rho)  s^2 (1-s)^{-\frac{1}{2}} \int_\RR \chi(\rho\lla \omega \rra) \frac{e^{i \omega (\tau+\log(1-s))}   \varphi_0(\rho;i \omega)}{(3-2i\omega)(-1-2i\omega)} \tilde{\gamma}_1(\rho,s;i\omega) d\omega.
\end{align*}
From the proof of Proposition \ref{prop:kernels} we already see that
\begin{align*}
|I_1(\rho,s,\tau)|& \lesssim 1_{\RR_+}(s-\rho)s(1-s)^{-\frac{1}{2}} \lla \tau+\log(1-s) \rra ^{-2} \\ &\leq  \rho^{-1}s^2 (1-s)^{-\frac{1}{2}} \lla \tau+\log(1-s) \rra ^{-2}.
\end{align*}
For $I_2$, we use \eqref{eqn:1red}, and as before we have
\begin{align*}
& \lf| \int_\RR \chi(\rho\lla \omega \rra) \frac{e^{i \omega (\tau+\log(1-s))}   \varphi_0(\rho;i \omega)}{(3-2i\omega)(-1-2i\omega)} \tilde{\gamma}_1(\rho,s;i\omega) d\omega \rt| \\
\cong & \lf| \rho^{-1} \int_{-1}^1 (1+\rho t)^{-\frac{1}{2}} \int_{\RR} \chi(\rho\lla \omega \rra) e^{i \omega (\tau+\log(1-s)-\log(1-\rho t))} \frac{1-2i\omega}{1+2i\omega} \tilde{\gamma}_1(\rho,s;i\omega) d\omega dt \rt| \\
\leq & \rho^{-1} \lla \tau + \log(1-s) \rra ^{-2},
\end{align*}
This gives the bound
$$
|I_2 (\rho,s,\tau)| \lesssim  \rho^{-1} s^2 (1-s)^{-\frac{1}{2}} \lla \tau+\log(1-s) \rra ^{-2}
$$

For $\tilde{G}_2$, we do the exact same thing in $G_2$ but we stop at \eqref{eqn:G2I2est}, then we obtain the desired bound. Similarly, $\tilde{G}_3$ is exactly the same as $G_3$, where we use \eqref{eqn:T3est}.

For $n=\{4,5,6\}$, by the same method as the case of $G_n$ in Proposition \ref{prop:kernels}, we have the bounds as in \eqref{eqn:T4est1},
\begin{align*}
 \lf| \int_\RR e^{i \omega \tau}  \tilde{G}_n(\rho ,s ; i \omega) d\omega \rt|  \lesssim & 1_{\RR_+}(\rho-s)\rho^{-2} s^3 (1-s)^{-\frac{1}{2}} \lla \tau + \log(1-s) \rra^{-2}  \\
 \leq &  \rho^{-1}s^2 (1-s)^{-\frac{1}{2}} \lla \tau + \log(1-s) \rra^{-2}
\end{align*}
as desired.
\end{proof}

Similarly, we have the following result for $S'_n$.

\begin{lemma}\label{lemma:g'} We have
$$
\lf| \int_\RR e^{i \omega \tau}  G'_n(\rho ,s ; i \omega) d\omega \rt| \lesssim \rho^{-2} (1 -\rho)^{-\frac{1}{2}} s^2(1-s)^{-\frac{1}{2}}\lla \tau- \log(1- \rho) +\log (1-s) \rra^{-2}
$$
for all $\tau>0$, $\rho,s\in(0,1)$, and $n=\{1, 2, \ldots,6\}$.
\end{lemma}

\begin{proof} The logic is still very similar to the proof of Proposition \ref{prop:kernels}. For $G'_1$, we have
$$
\int_\RR e^{i \omega \tau}  G'_1(\rho ,s ; i \omega) d\omega = 2 I_1(\rho,s,\tau)+I_2(\rho,s,\tau),
$$
where
\begin{align*}
I_1(\rho,s,\tau) = & 1_{\RR+} (s-\rho)s(1-s)^{-\frac{1}{2}} \\
& \times \int_\RR \chi(\rho\lla\omega\rra) e^{i \omega ( \tau+\log (1-s))} \frac{g_1(\rho; i\omega) - \tilde{g}_1(\rho; i\omega)  }{1+2 i\omega}  \gamma'_1(\rho,s; i\omega) d\omega,\\
I_2(\rho,s,\tau) = & 1_{\RR+} (s-\rho)s^2(1-s)^{-\frac{1}{2}} \\
& \times \int_\RR \chi(\rho\lla\omega\rra) e^{i \omega ( \tau+\log (1-s))} \lf( g_1(\rho; i\omega) - \tilde{g}_1(\rho; i\omega) \rt) \frac{-1+2i\omega}{1+2i\omega}   \gamma'_1(\rho,s; i\omega) d\omega.
\end{align*}
For $I_1$, we use the expression
\begin{equation}\label{eqn:4red}
g_1(\rho;i\omega) - \tilde{g}_1(\rho;i\omega) 
= -\frac{1}{2\rho}(1+2i\omega) \int_{-1}^1 (1+\rho t)^{-\frac{3}{2}-i\omega} dt.
\end{equation}
Then, again thanks to the cut-off, we have
\begin{align*}
&  | I_1(\rho,s,\tau)  | \\
 \leq  & 1_{\RR+} (s-\rho)\rho^{-1} s(1-s)^{-\frac{1}{2}} \\
& \times \int_{-1}^1 \lf| (1+\rho t)^{-\frac{3}{2}} \int_\RR \chi(\rho\lla\omega\rra) e^{i \omega ( \tau+\log (1-s)-\log(1+\rho t))} \frac{-1+2i\omega}{1+2i\omega} \gamma'_1(\rho,s; i\omega) d\omega \rt|  dt \\
 \lesssim & 1_{\RR+} (s-\rho)\rho^{-1} s(1-s)^{-\frac{1}{2}} \lla  \tau+\log (1-s) \rra^{-2} \\
 \lesssim & \rho^{-2} (1-\rho)^{-\frac{1}{2}} s^2 (1-s)^{-\frac{1}{2}} \lla  \tau - \log(1-\rho) +\log (1-s) \rra^{-2}
\end{align*}
by Lemma \ref{lemma:o}. For $I_2$, we need to bound the expression
\begin{align*}
I_{2,\pm} (\rho,s,\tau)  = & 1_{\RR+} (s-\rho)\rho^{-2} (1\pm \rho)^{-\frac{1}{2}} s^2(1-s)^{-\frac{1}{2}} \\
& \times\int_{\RR}\chi(\rho\lla\omega\rra) e^{i \omega ( \tau- \log(1\pm \rho) +\log (1-s))}   \frac{-1+2i\omega}{1+2i\omega}  \gamma'_1(\rho,s; i\omega) d\omega.
\end{align*}
Again by Lemma \ref{lemma:o} we have
$$
|I_2(\rho,s,\tau) | \lesssim \rho^{-2} (1 -\rho)^{-\frac{1}{2}} s^2(1-s)^{-\frac{1}{2}}\lla \tau- \log(1- \rho) +\log (1-s) \rra^{-2}.
$$

For $G'_2$, the terms we need to bound are
\begin{align*}
I_1 (\rho,s,\tau) = & 1_{\RR+} (s-\rho)\rho^{-2} (1+ \rho)^{-\frac{1}{2}} s(1-s)^{-\frac{1}{2}} \\
& \times \int_\RR [1-\chi(\rho\lla\omega\rra)] e^{i \omega ( \tau-\log(1+\rho)+\log (1-s) )} \frac{ \gamma'_2(\rho,s; i\omega)}{1+2 i\omega}  d\omega
\end{align*}
and
\begin{align*}
I_2(\rho,s,\tau) = & 1_{\RR+} (s-\rho)\rho^{-2} (1+\rho)^{-\frac{1}{2}} s^2(1-s)^{-\frac{1}{2}} \\
& \times \int_\RR[1- \chi(\rho\lla\omega\rra)] e^{i \omega ( \tau-\log(1+\rho)+\log (1-s))} \frac{-1+2i\omega}{1+2i\omega} \gamma'_2(\rho,s; i\omega) d\omega.
\end{align*}
For $I_1$, notice that
$$
\frac{ \gamma'_2(\rho,s; i\omega)}{1+2 i\omega} = \mathcal{O}(\rho^0(1-\rho)^0s^0(1-s)^0 \lla \omega \rra^{-2}),
$$
we use Lemma \ref{lemma:rho1} and obtain
\begin{align*}
|I_1 (\rho,s,\tau)| &\lesssim 1_{\RR+} (s-\rho)\rho^{-1} (1+\rho)^{-\frac{1}{2}} s(1-s)^{-\frac{1}{2}}  \lla\tau+\log (1-s)\rra^{-2} \\
& \lesssim  \rho^{-2} (1-\rho)^{-\frac{1}{2}} s^2 (1-s)^{-\frac{1}{2}}  \lla\tau-\log(1-\rho)+\log (1-s)\rra^{-2}.
\end{align*}
For $I_2$, by Lemma \ref{lemma:o} we have
\begin{align*}
|I_2(\rho,s,\tau) | & \lesssim \rho^{-2} (1 +\rho)^{-\frac{1}{2}} s^2(1-s)^{-\frac{1}{2}}\lla \tau +\log (1-s) \rra^{-2}\\
& \lesssim  \rho^{-2} (1-\rho)^{-\frac{1}{2}} s^2 (1-s)^{-\frac{1}{2}}  \lla\tau-\log(1-\rho)+\log (1-s)\rra^{-2}.
\end{align*}

Next, $G'_3$ is the same as $G'_2$ with $(1-\rho)$ instead of $(1+\rho)$, and we obtain the same bound.

For $G'_4$, we first write
\begin{align}
& s^4(1-s^2)^{-\frac{1}{2}+ \lambda }\varphi_0(s;\lambda ) \nonumber \\
= & s\lf((1-s)^{-\frac{1}{2}+ \lambda }(2+s(-1+2\lambda ))-(1+s)^{-\frac{1}{2}+ \lambda }(2+s(-1+2\lambda )) \rt)  \nonumber  \\
= &s\lf((1-s)^{-\frac{1}{2}+ \lambda }(2(1-s)+s(1+2\lambda ))-(1+s)^{-\frac{1}{2}+\lambda }(2(1+s)-s(1+2\lambda )) \rt)  \nonumber  \\
 = & 2s\lf( (1-s)^{\frac{1}{2}+ \lambda } - (1+s)^{\frac{1}{2}+ \lambda } \rt) + s^2(1+2\lambda )\lf( (1-s)^{-\frac{1}{2}+ \lambda } + (1+s)^{-\frac{1}{2}+ \lambda } \rt). \label{eqn:spartdecomp}
\end{align}
Then, we write
$$
\int_\RR e^{i \omega \tau}  G'_4(\rho ,s ; i \omega) d\omega = 2 I_1(\rho,s,\tau)+I_2(\rho,s,\tau),
$$
where
\begin{align*}
I_1(\rho,s,\tau) = & 1_{\RR_+}(\rho - s) \rho^{-2}(1+\rho)^{-\frac{1}{2}} s \\
& \times \int_{\RR} \chi(s \lla \omega\rra) e^{i\omega(\tau-\log(1+\rho)) } \frac{\lf( (1-s)^{\frac{1}{2}+ i\omega} - (1+s)^{\frac{1}{2}+ i\omega} \rt) }{1+2i\omega}  \gamma'_4(\rho,s ;i\omega) d\omega ,\\
I_2(\rho,s,\tau) = & 1_{\RR_+}(\rho - s)  \rho^{-2}(1+\rho)^{-\frac{1}{2}} s^2  \\
&\times \int_{\RR} \chi(s \lla \omega\rra) e^{i\omega(\tau-\log(1+\rho)) }  \lf( (1-s)^{-\frac{1}{2}+ i\omega} + (1+s)^{-\frac{1}{2}+ i\omega} \rt) \gamma'_4(\rho,s ;i\omega) d\omega.
\end{align*}
In $I_1$, we use the expression \eqref{eqn:3red}, and because of the cut-off, we have
\begin{align*}
& \lf|I_1(\rho,s,\tau) \rt| \\
 \lesssim &  \rho^{-2}(1+\rho)^{-\frac{1}{2}} s \int_{-1}^1 \lf| (1+s t)^{-\frac{1}{2}} \int_{\RR} \chi(s \lla \omega\rra) e^{i\omega(\tau-\log(1+\rho)+\log(1+st)) } \gamma'_4(\rho,s ;i\omega)  \rt|dt \\
 \lesssim &   \rho^{-2}(1+\rho)^{-\frac{1}{2}} s \lla \tau - \log(1+\rho) \rra^{-2} \\
 \lesssim & \rho^{-2} (1 -\rho)^{-\frac{1}{2}} s^2(1-s)^{-\frac{1}{2}}\lla \tau- \log(1- \rho) +\log (1-s) \rra^{-2}.
\end{align*}
In $I_2$, we need to look at
$$
I_{2,\pm} = 1_{\RR_+}(\rho - s) \rho^{-2}(1+\rho)^{-\frac{1}{2}} s^2  (1\pm s)^{-\frac{1}{2}} \int_{\RR} \chi(s \lla \omega\rra) e^{i\omega(\tau-\log(1+\rho) + \log(1\pm s)) }  \gamma'_4(\rho,s ;i\omega) d\omega,
$$
and again we have
$$
|I_2(\rho,s,\tau) | \lesssim \rho^{-2} (1 -\rho)^{-\frac{1}{2}} s^2(1-s)^{-\frac{1}{2}}\lla \tau- \log(1- \rho) +\log (1-s) \rra^{-2}.
$$

For $G'_5$ we need to bound
\begin{align*}
I_{1} (\rho,s,\tau) = & 1_{\RR+} (\rho-s)\rho^{-2} (1+\rho)^{-\frac{1}{2}} s(1 - s)^{-\frac{1}{2}} \\
& \times \int_\RR [1-\chi(s \lla\omega\rra)] e^{i \omega ( \tau-\log(1+\rho)+\log (1- s) )} \frac{ \gamma'_5(\rho,s; i\omega)}{1+2 i\omega}  d\omega,
\end{align*}
and
\begin{align*}
I_{2}(\rho,s,\tau) = & 1_{\RR+} (\rho-s)\rho^{-2} (1+\rho)^{-\frac{1}{2}} s^2(1- s)^{-\frac{1}{2}} \\
& \times \int_\RR[1- \chi(s\lla\omega\rra)] e^{i \omega ( \tau-\log(1+\rho)+\log (1- s))} \frac{-1+2i\omega}{1+2i\omega}\gamma'_5(\rho,s; i\omega) d\omega.
\end{align*}
So for $I_1$ we use Lemma \ref{lemma:rho1} and have
\begin{align*}
 | I_1 (\rho,s,\tau)  | \lesssim &  \rho^{-2} (1+\rho)^{-\frac{1}{2}} s^2 (1 - s)^{-\frac{1}{2}} \\
& \times\lf| s^{-1} \int_\RR [1-\chi(s \lla\omega\rra)] e^{i \omega ( \tau-\log(1+\rho)+\log (1 - s) )}  \mathcal{O}(\rho^0(1-\rho)^0s^0(1-s)^0 \lla \omega \rra^{-2})  d\omega\rt| \\
\lesssim & \rho^{-2} (1+\rho)^{-\frac{1}{2}} s^2 (1 - s)^{-\frac{1}{2}} \lla \tau- \log(1 + \rho) +\log (1-s) \rra^{-2} \\
\lesssim & \rho^{-2} (1 -\rho)^{-\frac{1}{2}} s^2(1-s)^{-\frac{1}{2}}\lla \tau- \log(1- \rho) +\log (1-s) \rra^{-2}.
\end{align*}
And for $I_2$ we have
\begin{align*}
 | I_2 (\rho,s,\tau)  |\lesssim & \rho^{-2} (1+\rho)^{-\frac{1}{2}} s^2 (1 - s)^{-\frac{1}{2}} \lla \tau- \log(1 + \rho) +\log (1-s) \rra^{-2} \\
\lesssim & \rho^{-2} (1 -\rho)^{-\frac{1}{2}} s^2(1-s)^{-\frac{1}{2}}\lla \tau- \log(1- \rho) +\log (1-s) \rra^{-2}.
\end{align*}

Finally, $G'_6$ is similar to $G'_5$, but we get $(1+s)$ instead of $(1-s)$. Then, we have
$$
(1+s)^{-\frac{1}{2}}\lla \tau+\log (1+s) \rra^{-2} \lesssim \lla \tau  \rra^{-2} \lesssim (1-\rho)^{-\frac{1}{2}} (1-s)^{-\frac{1}{2}}\lla \tau- \log(1- \rho) +\log (1-s) \rra^{-2}
$$
and obtain the desired bound.
\end{proof}

\begin{lemma}\label{lemma:sbound} We have
\begin{align*}
\|\tilde{S}_n(\tau) f\|_{L^2(\BB^5)}  & \lesssim \|f\|_{L^2(\BB^5)} , \qquad n=\{1,\cdots,6\},\\
\|S'_n(\tau) f\|_{L^2(\BB^5)}  & \lesssim \|f\|_{L^2(\BB^5)}, \qquad n=\{1,\cdots,6\},
\end{align*}
for $\tau>0$ and $f\in C^1([0,1])$.
\end{lemma}

\begin{proof} We write
$$
S'_n(\tau) f(\rho)  =\frac{1}{2\pi} \int_0^1 f(s) K'_n(\rho,s;\tau) ds,
$$
where
$$
K'_n (\rho,s;\tau) =  \int_\RR e^{i \omega \tau} G'_n (\rho ,s ; i \omega) d\omega.
$$
Then Lemma \ref{lemma:g'} gives the bounds
$$
| K'_n (\rho,s;\tau) | \lesssim \rho^{-2} (1 -\rho)^{-\frac{1}{2}} s^2(1-s)^{-\frac{1}{2}}\lla \tau- \log(1- \rho) +\log (1-s) \rra^{-2}.
$$
Doing a change of variables $\rho = 1-e^{-x}$, $s=1-e^{-y}$, we get
\begin{align*}
\lf| [S'_n(\tau) f ]( 1-e^{-x}) \rt| \lesssim e^\frac{x}{2}(1-e^{-x})^{-2} \int_0^\infty \lla \tau - y+x \rra^{-2} |f(1-e^{-y})| (1-e^{-y})^2  e^{-\frac{y}{2}} dy ,
\end{align*}
so using Young's inequality, we have
\begin{align*}
\| S'_n(\tau) f \|_{L^2(\BB^5)} \cong& \lf( \int_0^1 |S'_n(\tau) f(\rho)|^2 \rho^4 d\rho \rt)^{\frac{1}{2}} \\
\cong &\lf( \int_0^\infty |[S'_n(\tau)] f( 1-e^{-x})|^2 ( 1-e^{-x})^4 e^{-x} dx \rt)^{\frac{1}{2}} \\
\lesssim & \lf(\int_0^\infty \lf| \int_0^\infty \lla \tau - y+x \rra^{-2} |f(1-e^{-y})| (1-e^{-y})^2  e^{-\frac{y}{2}}  dy \rt|^2 dx \rt)^\frac{1}{2} \\
\lesssim & \|\lla \tau - \cdot \rra \|_{L^1(\RR)} \lf( \int_0^1 |f(\rho)|^2 \rho^4 d\rho \rt)^\frac{1}{2}\\
\lesssim & \|f\|_{L^2(\BB^5)}.
\end{align*}

For the operators $\tilde{S}_n$, from Lemma \ref{lemma:gtilde} we have
\begin{align*}
\lf|\rho^{-1} \int_\RR e^{i \omega \tau}  \tilde{G}_n(\rho ,s ; i \omega) d\omega \rt| & \lesssim \rho^{-2}s^2 (1-s)^{-\frac{1}{2}}\lla \tau  + \log(1-s) \rra^{-2} \\
& \lesssim \rho^{-2} (1 -\rho)^{-\frac{1}{2}} s^2(1-s)^{-\frac{1}{2}}\lla \tau- \log(1- \rho) +\log (1-s) \rra^{-2},
\end{align*}
so the bounds for $\tilde{S}_n$ follows the same way.
\end{proof}

\subsection{The term containing $\lambda$}

As before we need to deal with the $\lambda \tilde{f}_1$ term in $F_\lambda$. Here we also need to deal with the operators
\begin{align*}
 \dot{\tilde{S}}_{n,\epsilon} (\tau) f(\rho) : =& \frac{1}{2\pi i} \lim_{N \to \infty} \int_{\epsilon- i N}^{\epsilon + i N} \frac{1}{2}(1+2\lambda)e^{\lambda \tau} \int_0^1  \rho^{-1} \tilde{G}_n(\rho ,s ; \lambda) f(s) ds d\lambda , \\
  \dot{S}'_{n,\epsilon} (\tau) f(\rho) : =& \frac{1}{2\pi i} \lim_{N \to \infty} \int_{\epsilon- i N}^{\epsilon + i N} \frac{1}{2}(1+2\lambda) e^{\lambda \tau} \int_0^1  G'_n(\rho ,s ; \lambda) f(s) ds d\lambda .
\end{align*}

\begin{lemma}\label{lemma:dotStilde} Setting $ \dot{\tilde{S}}_n(\tau) f(\rho) := \lim_{\epsilon \to 0+}  \dot{\tilde{S}}_{n,\epsilon}(\tau) f(\rho)$, we have the bound
$$
\|\dot{\tilde{S}}_n(\tau) f\|_{L^2(\BB^5)} \lesssim \| f \|_{H^1(\BB^5)}
$$
for all $\tau>0$, $f\in C^1([0,1])$, and $n \in \{1, 2, \ldots,6\}$ .
\end{lemma}
\begin{proof} We follow the logic in Proposition \ref{prop:dotTbound} by an integration by parts in $s$. Indeed, the form of $ \dot{\tilde{S}}_{n,\epsilon} (\tau)$ and $\dot{T}_{n,\epsilon}(\tau)$ only differ in a factor of $\rho^{-1}$. After an integration by parts, the integral terms can be treated the same way as in Lemma \ref{lemma:gtilde} followed by Lemma \ref{lemma:sbound}, and the $L^2(\BB^5)$ norms are bounded by $C\lf(\|f'\|_{L^2(\BB^5)} + \|(\cdot)^{-1} f\|_{L^2(\BB^5)} \rt) \lesssim \|f\|_{H^1(\BB^5)}$ for all $\tau>0$. On the other hand, the boundary terms are (pointwise) bounded by  $C |f(\rho)|\lla \tau \rra^{-2}$ (see proof of Proposition \ref{prop:dotTbound}), so the $L^2(\BB^5)$ norm is bounded by $\|f\|_{L^2(\BB^5)}$ for all $\tau>0$.
\end{proof}

\begin{lemma}\label{lemma:dotS'} Setting $ \dot{S}'_n(\tau) f(\rho) := \lim_{\epsilon \to 0+}  \dot{S}'_{n,\epsilon}(\tau) f(\rho)$, we have the bound
$$
\|\dot{S}'_n(\tau) f\|_{L^2(\BB^5)} \lesssim \| f \|_{H^1(\BB^5)}
$$
for all $\tau>0$, $f\in C^1([0,1])$, and $n \in \{1, 2, \ldots,6\}$ .
\end{lemma}
\begin{proof} Again, we follow the same method as Proposition \ref{prop:dotTbound}. The $s$ dependent part of $\dot{S}'_n$ has the same form as $\dot{T}_n$ for each $n=\{1,\cdots,6\}$.

For $n=\{1,2,3\}$, we do an integration by parts as in \eqref{eqn:intbyparts1}. The integral part is dealt with by the same way as Lemma \ref{lemma:g'} (for each fixed $N$, we use Fubini to interchange the $\omega$ and $s$ integral, then we use dominated convergence for the $s$ integral to take the limit $N\to\infty$). Then the method in Lemma \ref{lemma:sbound} gives the same bound as in Lemma \ref{lemma:dotStilde}. For the boundary term, with the logic in Lemma \ref{lemma:limitinterchange} we can take the limit $\epsilon\to 0$.   For $n=1$, we have
\begin{align*}
& B'_1(\tau)f(\rho)  \\ = &  \frac{2 (1-\rho)^{\frac{1}{2}} \rho f(\rho) }{2\pi i} \int_{\RR} e^{i\omega (\tau  + \log(1-\rho) )} \frac{\chi(\rho \lla \omega\rra) (g_1(\rho;i\omega)-\tilde{g}_1(\rho;i\omega))}{1+2i\omega} \gamma'_1(\rho,\rho;i\omega) d\omega \\
& + \frac{ (1-\rho)^{\frac{1}{2}} \rho^2 f(\rho) }{2\pi i} \int_{\RR} e^{i\omega (\tau + \log(1-\rho) )} \chi(\rho \lla \omega\rra) (g_1(\rho;i\omega)-\tilde{g}_1(\rho;i\omega) )\frac{-1+2i\omega}{1+2i\omega}\gamma'_1(\rho,\rho;i\omega) d\omega.
\end{align*}
In the first term we use \eqref{eqn:4red} and in the second term we use the expression of $g_1$ and $\tilde{g}_1$ directly, similar to proof of Lemma \ref{lemma:g'}. So we get the bound $|B'_1(\tau) f(\rho)| \lesssim |f(\rho)| \lla \tau\rra^{-2}$. For $n=\{2,3\}$, we have
\begin{align*}
B'_n(\tau)f(\rho) = & \frac{2 (1-\rho)^\frac{3}{2} (1\pm \rho)^{-\frac{1}{2}}  f(\rho) }{2\pi i} \rho^{-1} \int_{\RR} [1-\chi(\rho \lla \omega\rra)] e^{i\omega (\tau  + \log(1-\rho)- \log(1 \pm \rho) )} \frac{\gamma'_n(\rho,\rho;i\omega)}{1+2i\omega}  d\omega \\
& + \frac{(1-\rho)^\frac{1}{2} (1\pm \rho)^{-\frac{1}{2}} f(\rho) }{2\pi i} \int_{\RR} [1-\chi(\rho \lla \omega\rra)] e^{i\omega (\tau  + \log(1-\rho)- \log(1 \pm \rho))}\gamma'_n(\rho,\rho;i\omega) d\omega.
\end{align*}
The same treatment as in the proof of Lemma \ref{lemma:g'} gives $|B'_n(\tau) f(\rho)| \lesssim |f(\rho)| \lla \tau\rra^{-2}$.

For $n=4$, we decompose the operator into two parts according to \eqref{eqn:spartdecomp}. So we have
\begin{align*}
\dot{S}'_{4,\epsilon}(\tau) f(\rho) = & \frac{1}{2\pi i} \lim_{N\to\infty} \int_{\epsilon-iN}^{\epsilon+iN} e^{\lambda \tau} g_1(\rho;\lambda) \\
& \times \int_0^1 1_{\RR_+}(\rho-s) \chi(s\lla \omega\rra) s \lf((1-s)^{\frac{1}{2}+\lambda}-(1+s)^{\frac{1}{2}+\lambda} \rt) f(s) \gamma'_4(\rho,s;\lambda ) ds d\lambda\\
& +  \frac{1}{2\pi i} \lim_{N\to\infty} \int_{\epsilon-iN}^{\epsilon+iN} \frac{1}{2}(1+2\lambda) e^{\lambda \tau} g_1(\rho;\lambda) \\
& \times \int_0^1 1_{\RR_+}(\rho-s) \chi(s\lla \omega\rra) s^2 \lf((1-s)^{-\frac{1}{2}+\lambda}-(1+s)^{-\frac{1}{2}+\lambda} \rt)  f(s) \gamma'_4(\rho,s;\lambda) ds d\lambda \\
=: & I_{1,\epsilon}(\tau,\rho) + I_{2,\epsilon}(\tau,\rho)
\end{align*}
In $I_1$ we can take the limit in $\epsilon$ and change order of integration (same argument as in \cite[Lemma 5.3]{donn2017}). Then we have $I_1 (\tau,\rho) := \lim_{\epsilon\to 0+} I_{1,\epsilon}(\tau,\rho)$, and
\begin{align*}
I_1(\tau,\rho) = & \frac{\rho^{-2} (1+\rho)^{-\frac{1}{2}} f(\rho)}{2\pi}\int_0^1 1_{\RR_+}(\rho-s) s\\
& \times \int_{\RR} \chi(s\lla \omega\rra) e^{i\omega(\tau - \log(1+\rho))}\lf((1-s)^{\frac{1}{2}+\lambda}-(1+s)^{\frac{1}{2}+\lambda} \rt) \gamma'_4(\rho,s;i \omega) d\omega ds.
\end{align*}
Following the method in the proof of Lemma \ref{lemma:g'} (the $I_2$ term in $G'_4$), we get the bound $\|I_1(\tau) f\| \lesssim \|(\cdot)^{-1} f\|_{L^2(\BB^5)}  \lesssim \|f\|_{H^1(\BB^5)}$ for all $\tau>0$. In $I_2$ we do an integration by parts in $s$ using 
$$
(1\pm s)^{-\frac{1}{2}+\lambda} =  \frac{\pm 2}{1+2\lambda} \PD_s (1\pm s)^{\frac{1}{2}+\lambda},
$$
so
\begin{align*}
& \int_0^1 1_{\RR_+}(\rho-s) \chi(s\lla \omega\rra) s^2 (1 \pm s)^{-\frac{1}{2}+\lambda}  f(s) \gamma'_4(\rho,s;\lambda) ds \\
= &  \frac{\pm 2}{1+2\lambda} \int_0^\rho \PD_s (1\pm s)^{\frac{1}{2}+\lambda}  \chi(s\lla \omega\rra) s^2 f(s) \gamma'_4(\rho,s;\lambda) ds  \\
= &  \frac{\pm 2}{1+2\lambda}  (1\pm \rho )^{\frac{1}{2}+\lambda}  \chi(\rho \lla \omega\rra) \rho^2 f(\rho) \gamma'_4(\rho,\rho;\lambda) \\
& -\frac{\pm 2}{1+2\lambda} \int_0^\rho  (1\pm s)^{\frac{1}{2}+\lambda}  \PD_s\lf[\chi(s\lla \omega\rra) s^2 f(s) \gamma'_4(\rho,s;\lambda) \rt]ds.
\end{align*}
The integral part is dealt with by the same argument as in Lemma \ref{lemma:g'} (the $I_2$ term in $G'_4$), and we again obtain a bound by $C \lf(\|f'\|_{L^2(\BB^5)} + \|(\cdot)^{-1} f\|_{L^2(\BB^5)} \rt) \lesssim \|f\|_{H^1(\BB^5)}$ for all $\tau>0$. In the boundary term we take the limit $\epsilon \to 0+$ by Lemma \ref{lemma:limitinterchange}, so
\begin{align*}
B'_4(\tau)f(\rho) = \frac{\pm 2(1\pm \rho )^{\frac{1}{2}} (1+\rho)^{-\frac{1}{2}} f(\rho) }{2\pi i} \int_{\RR} \chi(\rho \lla \omega\rra)e^{i\omega(\tau - \log(1+\rho) + \log(1\pm\rho) )} \gamma'_4(\rho,\rho; i\omega) d\omega ,
\end{align*} 
and we get the bound $|B'_4(\tau)f(\rho)| \lesssim f(\rho)\lla \tau \rra^{-2}$.

For $n=\{5,6\}$, as before we do an integration by parts in $s$ as in \eqref{eqn:intbyparts2}, and the integral terms are handled by the same way as in Lemma \ref{lemma:g'} (the $I_2$ term in $G'_5$ and $G'_6$). The boundary terms are
\begin{align*}
B_n'(\tau)f(\rho) = & \frac{\pm (1\pm \rho )^{\frac{3}{2}} (1+\rho)^{-\frac{1}{2}} f(\rho) }{\pi i} \rho^{-1} \int_{\RR} [1-\chi(\rho \lla \omega \rra)] e^{i\omega(\tau - \log(1+\rho) + \log(1\pm\rho) )} 
\frac{\gamma'_n(\rho,\rho; i\omega)}{1+2i\omega} d\omega \\
& \pm \frac{(1\pm \rho )^{\frac{1}{2}} (1+\rho)^{-\frac{1}{2}} f(\rho) }{2 \pi i}  \int_{\RR} [1-\chi(\rho \lla \omega \rra)] e^{i\omega(\tau - \log(1+\rho) + \log(1\pm\rho) )} \gamma'_n(\rho,\rho; i\omega) d\omega.
\end{align*}
So we also have $|B'_n(\tau)f(\rho)| \lesssim f(\rho)\lla \tau \rra^{-2}$ for $n=\{ 5,6\}$.
\end{proof}

\subsection{Improved energy bounds}

\begin{lemma}\label{lemma:improvedenergy}  The semigroup $\mathbf{S}$ admits the bound
$$
\lf\Vert\mathbf{S} (\tau) (\mathbf{I} - \mathbf{P} ) \mathbf{f} \rt\Vert _{\mathcal{H}} \lesssim  \lf\Vert (\mathbf{I} - \mathbf{P} ) \mathbf{f} \rt\Vert_{\mathcal{H}} 
$$
for all $\tau  \geq 0$ and $f\in \mathcal{H}$.
\end{lemma}
\begin{proof}
Here we have the representation of the first component
$$
[ \mathbf{S}(\tau) \tilde{\mathbf{f}} ]_1  = [ \mathbf{S}_0(\tau) \tilde{\mathbf{f}} ]_1 (\rho)+\sum_{n=1}^6 \lf[T_n(\tau)\lf(|\cdot|\tilde{f}_1' +2 \tilde{f}_1+\tilde{f}_2 \rt) +  \dot{T}_n(\tau) \tilde{f}_1 \rt],
$$
and the definition of the second component
$$
[ \mathbf{S}(\tau) \tilde{\mathbf{f}} ]_2(\rho) = \lf[\PD_\tau + \rho\PD_\rho + \frac{3}{2} \rt] [ \mathbf{S}(\tau) \tilde{\mathbf{f}} ]_1(\rho).
$$
This is the same proof as \cite[Lemma 5.7]{donn2017}.
\end{proof}

\section{Proof of the theorem}

We look at the nonlinear problem
$$
\Phi(\tau) = \mathbf{S}(\tau)\mathbf{u} + \int_0^\tau \mathbf{S} (\tau - \sigma) \mathbf{N}(\Phi(\sigma)) d\sigma ,
$$
where
\begin{equation*}
\mathbf{N} ( \mathbf{u} ) (\rho) := \begin{pmatrix}
0 \\
N( u_1(\rho) )
\end{pmatrix}.
\end{equation*}
with
$$
N(x) = \lf| c_5  + x \rt|^{\frac{4}{3}}  \lf( c_5 + x \rt)-   c_5^{\frac{7}{3}}  - \frac{35}{4} x
$$

\begin{lemma}\label{lemma:nonlin}
We have
$$
\| \mathbf{N} ( \mathbf{u} ) \|_\mathcal{H} \lesssim \|u_1\|^2_{L^5(\mathbb{B}^5)} + \|u_1\|^{\frac{7}{3}}_{L^{\frac{14}{3}}(\mathbb{B}^5)} 
$$
and
\begin{align*}
\|  \mathbf{N} (\mathbf{u})  -  \mathbf{N} (\mathbf{v}) \|_\mathcal{H} \lesssim & \| u_1 - v_1 \|_{L^5(\BB^5)} \lf( \| u_1 \|_{L^5(\BB^5)} + \| v_1 \|_{L^5(\BB^5)} \rt) \\
& + \| u_1 - v_1 \|_{L^{\frac{14}{3}}(\BB^5)}  \lf(\| u_1\|_{L^\frac{14}{3}(\BB^5)}^{\frac{4}{3}} + \| v_1\|_{L^\frac{14}{3}(\BB^5)}^{\frac{4}{3}}\rt)
\end{align*}
for all $\mathbf{u}, \mathbf{v} \in \mathcal{H}$.
\end{lemma}
\begin{proof} We compute
$$
N'(x) = \frac{7}{3} \lf| c_5 + x \rt|^{\frac{4}{3}}  -  \frac{35}{4} = \frac{7}{3} \lf| \lf( \frac{15}{4} \rt)^{\frac{3}{4}} + x \rt|^{\frac{4}{3}}  -  \frac{35}{4},
$$
and observe that $N(0) = N'(0) = 0$. Hence
$$
|N(x)| \lesssim |x|^2 + |x|^{\frac{7}{3}}, \qquad \; |N'(x)| \lesssim |x| + |x|^{\frac{4}{3}},
$$
for all $x\in \RR$. Since $N$ is monotone, we also have
$$
| N(x) - N(y) | \lesssim |x-y| \lf( |N'(x)|  +| N'(y)| \rt) \lesssim |x-y| \lf( |x| + |x|^{\frac{4}{3}} + |y| + |y|^{\frac{4}{3}} \rt).
$$
Then by H\"older's inequality we get
\begin{align*}
\| \mathbf{N} ( \mathbf{u} ) \|_\mathcal{H} &  = \| N(u_1) \|_{L^2(\mathbb{B}^5)}   \lesssim \| u_1^2 \|_{L^2(\mathbb{B}^5)}  + \| u_1^{\frac{7}{3}} \|_{L^2(\mathbb{B}^5)} \\
& \lesssim \|u_1\|^2_{L^4(\mathbb{B}^5)} + \|u_1\|^{\frac{7}{3}}_{L^{\frac{14}{3}}(\mathbb{B}^5)} \lesssim \|u_1\|^2_{L^5(\mathbb{B}^5)} + \|u_1\|^{\frac{7}{3}}_{L^{\frac{14}{3}}(\mathbb{B}^5)},
\end{align*}
and
\begin{align*}
 \|  \mathbf{N} (\mathbf{u})  -  \mathbf{N} (\mathbf{v}) \|_\mathcal{H}  \lesssim & \| N(u_1) - N(v_1) \|_{L^2(\mathbb{B}^5)}   \\
 \lesssim &  \|u_1 - v_1 \| _{L^2(\BB^5)} \lf( \|u_1\|_{L^2(\mathbb{B}^5)} + \|u_1\|_{L^2(\mathbb{B}^5)}^{\frac{4}{3}} + \|v_1\|_{L^2(\mathbb{B}^5)} + \|v_1\|^{\frac{4}{3}}_{L^2(\mathbb{B}^5)} \rt) \\
 \lesssim & \| u_1 - v_1 \|_{L^5(\BB^5)} \lf( \| u_1 \|_{L^5(\BB^5)} + \| v_1 \|_{L^5(\BB^5)} \rt) \\
 & + \| u_1 - v_1 \|_{L^{\frac{14}{3}}(\BB^5)}  \lf(\| u_1\|_{L^\frac{14}{3}(\BB^5)}^{\frac{4}{3}} + + \| v_1\|_{L^\frac{14}{3}(\BB^5)}^{\frac{4}{3}}\rt)
\end{align*}
as required.
\end{proof}

\subsection{The nonlinear problem}

In the following, we first modify the equation to remove the unstable direction, then we show how to remove this modification. 

For $\Psi(\tau)(\rho) = (\varphi_1(\tau,\rho),\varphi_2(\tau,\rho))$, we define
$$
\| \Psi \|_{\mathcal{X}}^2 := \| \Psi \|_{L^\infty(\RR_+; \mathcal{H})}^2 + \| \varphi_1 \|_{L^2(\RR_+; L^5(\BB^5))}^2,
$$
and introduce the Banach space
$$
\mathcal{X} := \{  \Phi \in C([0,\infty), \mathcal{H}) : \varphi_1 \in L^2(\RR_+; L^5(\BB^5)), \|\Phi \|_{\mathcal{X}} < \infty \}.
$$
Moreover, we set
$$
\mathcal{X}_\delta := \{ \Phi \in \mathcal{X} : \| \Phi \| \leq \delta \}.
$$
For initial data $\mathbf{u} \in \mathcal{H}$, we define
$$
\mathbf{K}_\mathbf{u}(\Phi) (\tau) = \mathbf{S}(\tau)\lf[ \mathbf{u} - \mathbf{C}(\Phi,\mathbf{u}) \rt] + \int_0^\tau \mathbf{S}(\tau - \sigma) \mathbf{N}(\Phi(\mathbf{u})) d\sigma,
$$
where
$$
\mathbf{C}(\Phi,\mathbf{u}) := \mathbf{P} \lf[ \mathbf{u} + \int_0^\infty e^{-\sigma} \mathbf{N}(\Phi(\sigma)) d\sigma \rt].
$$

\begin{lemma}\label{lemma:selfmap} There exist $c, \delta > 0$ such that if $\| \mathbf{u} \|_\mathcal{H} \leq \frac{\delta}{c}$ and $\Phi \in \mathcal{X}_\delta$, then $\mathbf{K}_\mathbf{u}(\Phi) \in \mathcal{X}_\delta$.
\end{lemma}

\begin{proof}
From Lemmas \ref{lemma:improvedenergy} and \ref{lemma:nonlin}, we have
\begin{align*}
\| (\mathbf{I} - \mathbf{P})\mathbf{K}_\mathbf{u}(\Phi)(\tau) \|_\mathcal{H} & \lesssim \| \mathbf{u} \|_\mathcal{H} + \int_0^\tau \| \mathbf{N}(\Phi(\sigma)) \|_\mathcal{H} d\sigma \\
& \lesssim \frac{\delta}{c} + \int_0^\tau  \lf( \| \varphi_1(\sigma,\cdot) \|_{L^5(\BB^5)}^2 + \| \varphi_1 (\sigma,\cdot)  \|_{L^{\frac{14}{3}}(\BB^5)}^{\frac{7}{3}} \rt) d\sigma \\
& \lesssim \frac{\delta}{c} + \delta^2 + \lf( \| \varphi_1 \|^\theta_{L^2(\RR_+; L^5(\BB^5))} \| \varphi_1 \|^{1-\theta}_{L^\infty(\RR_+; L^{\frac{10}{3}}(\BB^5))} \rt)^{\frac{7}{3}} \\
& \lesssim \frac{\delta}{c} + \delta^2 + \lf( \| \varphi_1 \|^\theta_{L^2(\RR_+; L^5(\BB^5))} \| \Phi \|^{1-\theta}_{L^\infty(\RR_+; \mathcal{H})} \rt)^{\frac{7}{3}} \\
& \lesssim \frac{\delta}{c} + \delta^2 +  \delta^{\frac{7}{3}}.
\end{align*}
Moreover, from the Strichartz estimates in Theorem \ref{thm:strichartz}, we have
\begin{align*}
\| (\mathbf{I} - \mathbf{P})\mathbf{K}_\mathbf{u}(\Phi)(\tau) \|_{L^2(\RR_+; L^5(\BB^5))} & \lesssim \| \mathbf{u} \|_\mathcal{H} + \int_0^\infty \| \mathbf{N}(\Phi(\sigma)) \|_\mathcal{H} d\sigma \\
& \lesssim \frac{\delta}{c} + \delta^2 + \delta^{\frac{7}{3}}.
\end{align*}
Next, from Lemma \ref{lemma:boundedpert} we have
\begin{align*}
\mathbf{P} \mathbf{K}_\mathbf{u}(\Phi)(\tau) & = \mathbf{P}\lf[ \mathbf{S}(\tau)\lf[ \mathbf{u} - \mathbf{C}(\Phi,\mathbf{u}) \rt] + \int_0^\tau \mathbf{S}(\tau - \sigma) \mathbf{N}(\Phi(\mathbf{u})) d\sigma\rt] \\
& = -\int e^{\tau-\sigma} \mathbf{P}  \mathbf{N}(\Phi (\sigma) ) d\sigma + \int_0^\tau e^{\tau-\sigma} \mathbf{P}  \mathbf{N}(\Phi (\sigma) ) d\sigma = - \int_\tau^\infty e^{\tau-\sigma} \mathbf{P}  \mathbf{N}(\Phi (\sigma) ) d\sigma.
\end{align*}
Since we also have $\rg \mathbf{P} = \lla \mathbf{g} \rra$, so by the Riesz representation theorem, there exists $\mathbf{g}^* \in \mathcal{H}$ with $\mathbf{P}\mathbf{f} = ( \mathbf{f}|\mathbf{g}^*)_\mathcal{H} \mathbf{g}$ for all $\mathbf{f}\in\mathcal{H}$. Then
\begin{align*}
\| \mathbf{P} \mathbf{K}_\mathbf{u}(\Phi)(\tau)  \|_\mathcal{H} & \lesssim \int_\tau^\infty e^{\tau-\sigma} \lf| (\mathbf{N}(\Phi (\sigma) )  | \mathbf{g}^*)_\mathcal{H} \rt| d\sigma\\
& \lesssim \int_\tau^\infty \| \mathbf{N}(\Phi (\sigma) )  \|_\mathcal{H} d\sigma \\
& \lesssim \delta^2 + \delta^{\frac{7}{3}}.
\end{align*}
Finally,
\begin{align*}
\| \mathbf{P} \mathbf{K}_\mathbf{u}(\Phi)(\tau)  \|_{L^5(\BB^5)} & \lesssim \int_\tau^\infty e^{\tau-\sigma} \| \mathbf{N}(\Phi(\sigma)) \|_\mathcal{H} d\sigma \\
& = 1_{[0,\infty)}(\tau) \int_{\RR} 1_{[-\infty,0]}(\tau-\sigma)e^{\tau-\sigma} \| \mathbf{N}(\Phi(\sigma)) \|_\mathcal{H} d\sigma ,
\end{align*}
then Young's inequality gives
\begin{align*}
\| \mathbf{P} \mathbf{K}_\mathbf{u}(\Phi)  \|_{L^2(\RR_+; L^5(\BB^5))} \lesssim & \|1_{[-\infty,0]} e^{(\cdot)} \|_{L^2(\RR)} \int_0^\infty  \| \mathbf{N}(\Phi(\sigma)) \|_\mathcal{H} d\sigma \\
\lesssim &  \delta^2 + \delta^{\frac{7}{3}}.
\end{align*}
Therefore, we have $ \| \mathbf{K}_\mathbf{u} (\Phi) \|_\mathcal{X} \lesssim \frac{\delta}{c} +  \delta^2 + \delta^{\frac{7}{3}}$, which implies the claim.
\end{proof}

\begin{lemma}\label{lemma:contraction} Let $\delta>0$ be small and $ \mathbf{u} \in \mathcal{H}$. Then
$$
\| \mathbf{K}_\mathbf{u}(\Phi) - \mathbf{K}_\mathbf{u}(\Psi) \|_\mathcal{H} \leq \frac{1}{2} \| \Phi - \Psi \|_\mathcal{X}
$$
for all $\Phi, \Psi \in \mathcal{X}_\delta$.
\end{lemma}

\begin{proof} From Lemmas \ref{lemma:improvedenergy} and \ref{lemma:nonlin}, we have
\begin{align*}
& \| (\mathbf{I} - \mathbf{P}) \lf[ \mathbf{K}_\mathbf{u}(\Phi)(\tau) - \mathbf{K}_\mathbf{u}(\Psi) (\tau) \rt] \|_\mathcal{H} \\
\lesssim & \int_0^\tau \| \mathbf{N}(\Phi (\sigma)) -\mathbf{N}(\Psi (\sigma)) \|_\mathcal{H} d\sigma \\
\lesssim & \int_0^\tau  \| \varphi_1(\sigma) - \psi_1(\sigma) \|_{L^5(\BB^5)} \lf( \| \varphi_1(\sigma) \|_{L^5(\BB^5)} + \| \psi_1(\sigma) \|_{L^5(\BB^5)} \rt) \\
& + \| \varphi_1(\sigma) - \psi_1(\sigma) \|_{L^{\frac{14}{3}}(\BB^5)}  \lf(\| \varphi_1(\sigma) \|_{L^\frac{14}{3}(\BB^5)}^{\frac{4}{3}} + \| \psi_1(\sigma)\|_{L^\frac{14}{3}(\BB^5)}^{\frac{4}{3}}\rt) d\sigma \\
\lesssim &  \| \varphi_1 - \psi_1 \|_{L^2(\RR_+; L^5(\BB^5))} \lf( \| \varphi_1\|_{L^2(\RR_+; L^5(\BB^5))} + \| \psi_1 \|_{L^2(\RR_+; L^5(\BB^5))} \rt) \\
& + \| \varphi_1 - \psi_1 \|_{L^{\frac{7}{3}}(\RR_+;L^{\frac{14}{3}}(\BB^5))}  \lf(\| \varphi_1 \|_{L^{\frac{7}{3}}(\RR_+;L^{\frac{14}{3}}(\BB^5))}^{\frac{4}{3}} + \| \psi_1\|_{L^{\frac{7}{3}}(\RR_+;L^{\frac{14}{3}}(\BB^5))}^{\frac{4}{3}}\rt) \\
\lesssim & \delta  \| \Phi - \Psi \|_\mathcal{X}.
\end{align*}
Furthermore, from Theorem \ref{thm:strichartz} we have
\begin{align*}
\| (\mathbf{I} - \mathbf{P})\lf[ \mathbf{K}_\mathbf{u}(\Phi)(\tau) - \mathbf{K}_\mathbf{u}(\Psi) (\tau) \rt] \|_{L^2(\RR_+; L^5(\BB^5))} & \lesssim  \int_0^\infty \| \mathbf{N}(\Phi (\sigma)) -\mathbf{N}(\Psi (\sigma)) \|_\mathcal{H} d\sigma \\
& \lesssim \delta  \| \Phi - \Psi \|_\mathcal{X}.
\end{align*}
Following the same logic as in the proof of Lemma \ref{lemma:selfmap}, we have
\begin{align*}
 \|  \mathbf{P}\mathbf{K}_\mathbf{u}(\Phi)(\tau) - \mathbf{P} \mathbf{K}_\mathbf{u}(\Psi) (\tau)  \|_\mathcal{H} 
\lesssim & \int_\tau^\infty e^{\tau - \sigma}\lf| ( \mathbf{N}(\Phi (\sigma)) -\mathbf{N}(\Psi (\sigma)) | \mathbf{g}^*) \rt|_\mathcal{H} d\sigma \\
& \lesssim \delta  \| \Phi - \Psi \|_\mathcal{X}.
\end{align*}
Lastly, 
\begin{align*}
\|  \lf[ \mathbf{P} \mathbf{K}_\mathbf{u}(\Phi)(\tau) - \mathbf{P} \mathbf{K}_\mathbf{u}(\Psi) (\tau) \rt]_1 \|_{L^5(\BB^5)}  \lesssim & \int_\tau^\infty e^{\tau-\sigma} \| \mathbf{N}(\Phi (\sigma)) -\mathbf{N}(\Psi (\sigma)) \|_\mathcal{H} d\sigma,
\end{align*}
and by Young's inequality we have
\begin{align*}
\|  \lf[ \mathbf{P} \mathbf{K}_\mathbf{u}(\Phi)- \mathbf{P} \mathbf{K}_\mathbf{u}(\Psi) \rt]_1 \|_{L^2(\RR_+; L^5(\BB^5))} \lesssim \delta  \| \Phi - \Psi \|_\mathcal{X}.
\end{align*}
Hence $\| \mathbf{K}_\mathbf{u}(\Phi) - \mathbf{K}_\mathbf{u}(\Psi) \|_\mathcal{H} \lesssim \delta \| \Phi - \Psi \|_\mathcal{X}$. By choosing $\delta$ sufficiently small, we get the inequality as claimed.
\end{proof}

\begin{cor}\label{cor:fixedpoint} There exist $c,\delta>0$ such that if $\| \mathbf{u} \|_{\mathcal{H}} \leq \frac{\delta}{c}$, then there exists a unique $\Phi \in \mathcal{X}_\delta$ satisfying $\Phi = \mathbf{K}_\mathbf{u}(\Phi)$.
\end{cor}
\begin{proof}
We obtain the conclusion from Lemmas \ref{lemma:selfmap} and \ref{lemma:contraction}, and the Banach fixed point theorem.
\end{proof}

Finally, we show that if we choose the blowup time correctly, the correction term $\mathbf{C}(\Phi, \mathbf{u})$ is in fact $0$. Recall that the initial data $\Phi(0) = (\varphi_1(0),\varphi_2(0))$ are given by
\begin{align*}
\varphi_1(0) & = \psi_1(0,\rho) - c_5 = T^{\frac{3}{2}}f(T\rho) - c_5 \\
\varphi_2(0) & = \psi_2(0,\rho) - \frac{3}{2} c_5 = T^{\frac{5}{2}} g(T\rho) - \frac{3}{2} c_5.
\end{align*}
The ODE blowup solution $u^1$, transformed to similarity coordinates, is given by
$$
\psi^1(\tau,\rho) = T^\frac{3}{2} e^{-\frac{3}{2}\tau} u^1 (T-Te^{-\tau}, T e^{-\tau} \rho) = T^\frac{3}{2} e^{-\frac{3}{2}\tau}  c_5 (1-T+Te^{-\tau})^{-\frac{3}{2}}.
$$
As in \eqref{eqn:coordtrans}, we set $\psi_1^1 := \psi^1$, and
$$
\psi_2^1 (\tau,\rho) = \PD_\tau \psi^1_1 + \rho \PD_\rho \psi^1_1+ \frac{3}{2} \psi^1_1 = \frac{3}{2} c_5 T^\frac{5}{2} e^{-\frac{3}{2}\tau} (1-T+Te^{-\tau})^{-\frac{5}{2}}.
$$
Thus intial data for this solution is given by
$$
\psi_1^1 (0,\rho) = c_5 T^\frac{3}{2}, \psi_2^1 (0,\rho) = \frac{3}{2}c_5 T^\frac{5}{2}.
$$
Now we rewrite the initial condition for $\Phi$ as
\begin{align*}
\Phi(0)(\rho) & = \lf( T^{\frac{3}{2}}\lf( f(T\rho) -c_5 \rt) + c_5 T^\frac{3}{2}- c_5 , T^{\frac{5}{2}} \lf( g(T\rho) -\frac{3}{2}c_5 \rt) +\frac{3}{2} c_5 T^\frac{5}{2} - \frac{3}{2} c_5 \rt) \\
& = \mathbf{U} \lf( T,  \lf(f-c_5 , g-\frac{3}{2} c_5  \rt) \rt)(\rho),
\end{align*}
where
$$
\mathbf{U}(T,\mathbf{v})(\rho) := \lf( T^{\frac{3}{2}}v_1(T\rho)  , T^{\frac{5}{2}}v_2(T\rho) \rt) + \lf(  c_5 T^\frac{3}{2}, \frac{3}{2} T^\frac{5}{2} \rt) - \lf( c_5, \frac{3}{2} c_5 \rt) .
$$
We can see that, for $\delta>0$ sufficently small and $\mathbf{v} \in H^1\times L^2 (\BB^5_{1+\delta})$, the map
$$
\mathbf{U}(\cdot, \mathbf{v}): [1-\delta, 1+\delta] \to H^1 \times L^2 (\BB^5)
$$
is continuous (see for example the proof of \cite[Lemma 4.14]{dosc2012}). 

\begin{lemma}\label{lemma:vartime} There exist $M \geq 1$ and $\delta >0$ such that, given $\| \mathbf{v} \|_{H^1 \times L^2(\BB^5_{1+\delta})} < \frac{\delta}{M}$, there exist $T^* \in [1-\delta, 1+\delta]$ and $\| \Phi \| \in \mathcal{X}_\delta$ with $\Phi = \mathbf{K}_{\mathbf{U}(T^*,\mathbf{v})}(\Phi)$ and $\mathbf{C}(\Phi, \mathbf{U}(T^*,\mathbf{v})) = 0$.
\end{lemma}

\begin{proof} First, we have
$$
\PD_T \begin{pmatrix} c_5 T^\frac{3}{2} \\ \frac{3}{2} c_5 T^\frac{5}{2}\end{pmatrix} \Big |_{T=1} =\begin{pmatrix} \frac{3}{2} c_5 \\ \frac{15}{4}  c_5 \end{pmatrix} = \frac{3}{4} c_5 \mathbf{g}.
$$
So we can write
$$
\mathbf{U}(T, \mathbf{v})(\rho) = \lf( T^{\frac{3}{2}}v_1(T\rho)  , T^{\frac{5}{2}}v_2(T\rho) \rt) + \frac{3}{4} c_5 \mathbf{g}(T-1) + (T-1)^2 \mathbf{f}_T
$$
where $\| \mathbf{f}_T \|_\mathcal{H} \lesssim 1$ on $T \in [\frac{1}{2}, \frac{3}{2}]$. Hence
$$
\lf( \mathbf{U}(T, \mathbf{v}) | \mathbf{g} \rt) = O\lf( \frac{\delta}{M} T^0 \rt) + \frac{3}{4} c_5 \| \mathbf{g} \|^2 (T-1) + O(\delta^2 T^0)
$$
for all $T\in [1-\delta, 1+\delta]$, $\delta\in[0,\frac{1}{2}]$, and $M \geq 1$.

Now notice that
$$
\| \mathbf{U}( T,\mathbf{v} ) \|_\mathcal{H} \lesssim \|v\|_{H^1\times L^2(\BB_{1+\delta}^5)} + |T-1|,
$$
so for every $T \in [1-\delta, 1+\delta]$, by Corollary \ref{cor:fixedpoint}, there exists $\Phi_T \in \mathcal{X}_\delta$ with $\Phi_T = \mathbf{K}_{\mathbf{U}( T,\mathbf{v} )} (\Phi)$  provided that $\delta>0$ is sufficiently small and $M \geq 1$ is sufficiently large.

Recall that 
$$
\mathbf{C}(\Phi,\mathbf{u}) := \mathbf{P} \lf[ \mathbf{u} + \int_0^\infty e^{-\sigma} \mathbf{N}(\Phi(\sigma)) d\sigma \rt].
$$
and notice that
$$
\int_0^\infty e^{-\sigma} \| \mathbf{N}(\Phi(\sigma)) \|_\mathcal{H} d\sigma \lesssim \delta^2,
$$
we have
$$
\lf( \mathbf{C}(\Phi,\mathbf{u}) | \mathbf{g} \rt)_\mathcal{H} = O\lf( \frac{\delta}{M} T^0 \rt) + \frac{3}{4} c_5 \| \mathbf{g} \|^2 (T-1) + O(\delta^2 T^0).
$$
Since $\rg \mathbf{C}(\Phi,\mathbf{u}) \subset \lla \mathbf{g} \rra$, we see that $\mathbf{C}(\Phi,\mathbf{u}) = 0$ is equivalent to $T-1 = F(T)$, where $F$ is a continuous function of $T$ on $[1-\delta, 1+\delta]$ with $|F(T)| \lesssim \frac{\delta}{M} + \delta^2$. Choosing $M$ sufficiently large, we have that $\rg(1+F)$ is a continuous function from $ [1-\delta, 1+\delta]$ to itself, so $1+F$ has a fixed point $T^*$. 
\end{proof}

\subsection{Proof of the main theorem}

Finally, we turn to our main result.

\begin{proof}[Proof of  Theorem \ref{thm:main}] Let $M \geq 1$ be sufficiently large and $\delta>0$ sufficiently small. For $(f,g)$ with
$$
\| (f,g) - u^1[0] \|_{H^1\times L^2(\BB_{1+\delta}^5)} \leq \frac{\delta}{M},
$$
let $\mathbf{v} =  (f,g) - u^1[0]$. Then we have the associated $\Phi \in \mathcal{X}_\delta$ and $T$ given by Lemma \ref{lemma:vartime}. Then from \eqref{eqn:simcoor}, \eqref{eqn:coordtrans}, and \eqref{eqn:perturbation}, we have
\begin{align*}
\delta^2 \geq & \| \varphi_1 \|^2_{L^2(\RR_+; L^5(\BB^5))} =  \int_0^\infty \| \varphi_1 (\tau, \cdot) \|^2_{L^5(\BB^5)} d\tau \\ 
= & \int_0^\infty \| \psi_1(\tau,\cdot) - c_5 \|^2_{L^5(\BB^5)}  d\tau \\
= & \int_0^T  \lf\| \psi_1\lf( \log \frac{T}{T-t},\cdot \rt) - c_5 \rt\|^2_{L^5(\BB^5)}  \frac{dt}{T-t} \\
= & \int_0^T  \lf\| \psi_1\lf( \log \frac{T}{T-t},\frac{\cdot}{T-t} \rt) - c_5 \rt\|^2_{L^5(\BB^5_{T-t})}  \frac{dt}{(T-t)^3} \\
= & \int_0^T  (T-t)^3 \lf\| (T-t)^{-\frac{3}{2}}\psi_1\lf( \log\frac{T}{T-t},\frac{\cdot}{T-t} \rt) -  (T-t)^{-\frac{3}{2}} c_5 \rt\|^2_{L^5(\BB^5_{T-t})}  \frac{dt}{(T-t)^3} \\
= & \int_0^T  \| u(t,\cdot) - u^T(t, \cdot) \|^2_{L^5(\BB^5_{T-t})}  dt \\
\simeq & \int_0^T \frac{\|u - u^T(t,\cdot) \|^2_{L^5(\BB_{T-t}^3)}}{\| u^T(t,\cdot) \|^2_{L^5(\BB_{T-t}^3)}} \frac{dt}{T-t}.
\end{align*}
\end{proof}

\appendix

\section{Properties of symbol type functions}\label{app:symbol}

Here we discuss some properties of functions of symbol type. From the Leibniz rule we have the following property.

\begin{lemma} Let $f,g: (0,1) \to \mathbb{C}$, $a, b \in [0,1]$, and $\alpha, \beta \in \RR$. If $f(x) = \mathcal{O}((x-a)^\alpha)$, $g(x)=\mathcal{O}((x-b)^\beta)$, then $f(x)g(x)=\mathcal{O}((x-a)^\alpha (x-b)^\beta)$.

In particular, if $f(x) = \mathcal{O}((x-a)^\alpha)$, $g(x)=\mathcal{O}((x-a)^\beta)$, then $f(x)g(x)=\mathcal{O}((x-a)^{\alpha+\beta})$.
\end{lemma}

\begin{proof} Fix $j\in \mathbb{N}_0$. By assumption, $f$ and $g$ satisfy
$$
\lf| \PD^k_x f(x) \rt| \lesssim |x-a|^{\alpha-j} \text{ and } \lf| \PD^k_x g(x) \rt| \lesssim |x-b|^{\beta-j} 
$$
for all $x\in (0,1)$ and $k\in \mathbb{N}_0$. Then by the Leibniz rule we have
\begin{align*}
\lf| \PD^j_x \lf( f(x) g(x) \rt) \rt| \leq & \sum_{k=0}^j \begin{pmatrix} j \\ k \end{pmatrix} \lf| \PD^{j-k}_x f(x) \rt| \lf| \PD^k_x g(x) \rt| \\
\lesssim & \sum_{k=0}^j \begin{pmatrix} j \\ k \end{pmatrix} |x-a|^{\alpha -(j-k)} |x-b|^{\beta-k} \\
\lesssim & |x-a|^{\alpha-j} |x-b|^{\beta-j},
\end{align*}
since $|x-a| < 1$ and $|x-b|<1$.
\end{proof}

Symbol behaviour is also stable under integration in certain situations. An example is the following.

\begin{lemma} Let $f:(0,1) \to \mathbb{C}$ and $\alpha > 1$ with $f(x) = \mathcal{O}(x^{-\alpha})$. Define a function $g:(0,c] \to \mathbb{C}$ with $c\in (0,1)$ by
$$
g(x) = \int_x^c f(y) dy.
$$
Then $g(x)= \mathcal{O}(x^{-\alpha +1})$.
\end{lemma}

\begin{proof} First of all, since $0<x\leq c<1$ and $-\alpha+1 < 0$, we have
$$
\lf| g(x) \rt| \leq  \int_x^c |f(y)| dy \lesssim \int_x^c y^{-\alpha} dy \lesssim  c^{-\alpha +1} + x^{-\alpha+1} \lesssim  x^{-\alpha+1}.
$$
Moreover, since
$$
\PD_x \int_x^c f(y) dy = -f(x),
$$
estimates on derivatives follow from the symbol behaviour of $f$.
\end{proof}

\begin{lemma} Let $f: I\subset \RR \to \mathbb{C} $, $0\in \overline{I}$ be such that $f\neq 0$ and $f(x) = 1+ \mathcal{O}(x)$. Then there exists $\delta > 0$ such that $1/f(x) = 1+\mathcal{O}(x)$ on $(0,\delta)$.
\end{lemma}
\begin{proof}
This follows from the Taylor expansion.
\end{proof}

Next, we show some properties that are specifically used in the derivative estimates of the Volterra solution from Lemma \ref{lemma:pertsol1}.

\begin{lemma} Let $x \in [0,\infty)$ and $j \in \mathbb{N}_0$. Then
$$
| \PD^j_x (e^{-x} \mathcal{O}(x^0))| \lesssim e^{-x} x^{-j} \lla x \rra^{j}.
$$
\end{lemma}
\begin{proof}
By the Leibniz rule, we have
$$
|\PD_x^j  (e^{-x} \mathcal{O}(x^0))| =  \lf| \sum_{l=0}^j \begin{pmatrix} j \\ l \end{pmatrix} \PD^{j-l}_x e^{-x} \PD^l_x \mathcal{O}(x^0) \rt| \lesssim  e^{-x} x^{-j} \lla x \rra^j.
$$
\end{proof}

\begin{lemma} Let $x,y \in [0,\infty)$ and $j\in \mathbb{N}_0$. Then
$$
\PD^j_x \mathcal{O}(x^0(x+y)^0) = \mathcal{O} (x^{-j} (x+y)^0).
$$
\end{lemma}
\begin{proof} We have
$$
| \PD^j_x \mathcal{O}(x^0(x+y)^0)| \lesssim \sum_{l=0}^j x^{-(j-l)}(x+y)^{-l} \lesssim x^{-j},
$$
since $x+y \geq x$, so $(x+y)^{-l} \leq x^{-l}$.
\end{proof}

\begin{lemma}\label{lemma:diffsymb1} Let $x \in [0,\infty)$, $y, \omega \in \mathbb{R}$ such that $y$ and $\omega$ have the same sign, and $k \in \mathbb{N}_0$. Then
$$
\PD^k_\omega \mathcal{O}\lf( \lf( \frac{y}{\omega} + x\rt) ^0 \rt) = \mathcal{O}\lf( \lf( \frac{y}{\omega} + x\rt) ^0 \omega^{-k} \rt). 
$$
\end{lemma}
\begin{proof} First, the claim is true for $k=0$. For $k\geq 1$, we show that
$$
\PD^k_\omega \mathcal{O}\lf( \lf( \frac{y}{\omega} + x\rt) ^0 \rt) = \sum_{m=1}^k \mathcal{O}\lf( \lf( \frac{y}{\omega} + x\rt)^{-m} \rt) y^m \omega^{-k-m}.
$$
We do this by induction. When $k=1$, we have
$$
\PD_\omega \mathcal{O}\lf( \lf( \frac{y}{\omega} + x\rt) ^0 \rt) = \mathcal{O}\lf( \lf( \frac{y}{\omega} + x\rt)^{-1} \rt) y \omega^{-2}
$$
which is what the formula claimed. Next, suppose that this is true for all derivatives up to order $k-1$. Then
\begin{align*}
& \PD_\omega^k \mathcal{O}\lf( \lf( \frac{y}{\omega} + x\rt) ^0 \rt)  = \PD_\omega   \sum_{m=1}^{k-1} \mathcal{O}\lf( \lf( \frac{y}{\omega} + x\rt)^{-m} \rt) y^m \omega^{-(k-1)-m} \\
= &  \sum_{m=2}^{k} \mathcal{O}\lf( \lf( \frac{y}{\omega} + x\rt)^{-m} \rt) y^{m} \omega^{-k-m} + \sum_{m=1}^{k-1}  \mathcal{O}\lf( \lf( \frac{y}{\omega} + x\rt)^{-m} \rt) y^m \omega^{-k-m}\\
= & \sum_{m=1}^{k}  \mathcal{O}\lf( \lf( \frac{y}{\omega} + x\rt)^{-m} \rt) y^m \omega^{-k-m}.
\end{align*}
So we can estimate
\begin{align*}
\lf|  \PD_\omega^k \mathcal{O}\lf( \lf( \frac{y}{\omega} + x\rt) ^0 \rt)  \rt| \lesssim &  \sum_{m=1}^{k} \lf| \lf( \frac{y}{\omega} + x\rt)^{-m}  y^m \omega^{-k-m} \rt| \\
= & \sum_{m=1}^{k} \lf| \lf( \frac{\omega}{y+ \omega x} \rt)^m y^m \omega^{-k-m} \rt| 
\lesssim  \lf| \frac{y}{y+\omega x} \rt|^m |\omega|^{-k} \leq | \omega |^{-k},
\end{align*}
since $|y| \leq |y+\omega x|$ for $\omega$ and $y$ with the same sign and $x \geq 0$.
\end{proof}

\begin{lemma}\label{lemma:diffsymb2} Let $x,y,\omega,k$ be as in the previous lemma. Then
$$
\PD^k_\omega \lf( e^{-\frac{y}{\omega}} \mathcal{O}\lf(\lf( \frac{y}{\omega}+x\rt)^0 \rt) \rt) = \sum_{m=0}^k e^{-\frac{y}{\omega}} \mathcal{O}\lf( \lf( \frac{y}{\omega}+x\rt)^0 y^m \omega^{-k-m} \rt).
$$
\end{lemma}
\begin{proof} The claim is true for $k=0$. Now suppose the statement is true for up to $k-1$ derivatives. Then
\begin{align*}
& \PD^k_\omega \lf( e^{-\frac{y}{\omega}} \mathcal{O}\lf(\lf( \frac{y}{\omega}+x\rt)^0 \rt) \rt) =  \PD_\omega \sum_{m=0}^{k-1} e^{-\frac{y}{\omega}} \mathcal{O}\lf( \lf( -\frac{y}{\omega}+x\rt)^0 y^m \omega^{-(k-1)-m} \rt) \\
 = &  \sum_{m=1}^{k} e^{-\frac{y}{\omega}} \mathcal{O}\lf( \lf( \frac{y}{\omega}+x\rt)^0 y^m \omega^{-k-m} \rt) + \sum_{m=0}^{k-1} e^{-\frac{y}{\omega}} \mathcal{O}\lf( \lf( \frac{y}{\omega}+x\rt)^{0} y^m \omega^{-k-m} \rt)  \\
 = & \sum_{m=0}^k e^{-\frac{y}{\omega}} \mathcal{O}\lf( \lf( \frac{y}{\omega}+x\rt)^0 y^m \omega^{-k-m} \rt),
\end{align*}
as required.
\end{proof}

\section{Derivatives of the Volterra solution}\label{app:volterra}

Here we show that the solution to the Volterra equation \eqref{eqn:ODEh1} is of symbol type. We introduce the variables $x=\varphi(\rho)$, $y=\varphi(s)$ with $\varphi (s) = \frac{1}{2} \log  \frac{1+s}{1-s}$. Define $\tilde{H}(x;\lambda):= h_1(\varphi^{-1}(x);\lambda)$, by change of variables, \eqref{eqn:ODEh1} is transformed into a Volterra equation of $\tilde{H}$,
$$
\tilde{H}(x;\lambda) = 1 + \int_x^\infty \tilde{K}(x,y;\lambda) \tilde{H}(y;\lambda) dy,
$$
where $\tilde{K}(x,y;\lambda) := K(\varphi^{-1} (x), \varphi^{-1}(y) ;\lambda) (\varphi^{-1})'(y)$. Notice that we have $|\PD^j_y \varphi^{-1}(y)| \lesssim_j e^{-2y}$ for all $j \in \mathbb{N}$ and $y \geq 0$, and
$$
\lf(\frac{1-\rho}{1+\rho} \rt)^{\frac{1}{2}-\lambda}\lf(\frac{1-s}{1+s} \rt)^{-\frac{1}{2}+\lambda} = e^{(1-2\lambda)(y-x)} = e^{2 i \omega(y-x)} e^{(1-2\epsilon)(y-x)}.
$$
Since $\epsilon \in [0,\frac{1}{4}]$ and $x\leq y$, we also have $e^{(1-2\epsilon)(y-x)} \leq e^{(y-x)}$. Hence we can write
$$
\tilde{K}(x,y;\lambda) (x,y; \lambda) = a(y;\lambda) + b(x,y;\lambda) e^{-2 i \omega(y-x)},
$$
where $a(y;\lambda) =  e^{-2y} \mathcal{O}(y^0 \lla \omega \rra^{-1})$ and $b(x,y;\lambda) = e^{-x} e^{-y} \mathcal{O}(x^0 y^0 \lla \omega \rra^{-1})$.

We rewrite this equation further, by setting
$$
\tilde{H} (x;\lambda) = 1+ e^{-2x} H(x;\lambda),
$$
so
$$
H(x;\lambda) := e^{2x} (H(x;\lambda)-1).
$$
Then we obtain the Volterra equation for $H$, given by
\begin{equation}\label{eqn:volterratrans}
H(x;\lambda) = g(x;\lambda) + \int_x^\infty \hat{K}(x,y;\lambda) H(y;\lambda) dy
\end{equation}
where
$$
g(x;\lambda) := e^{2x} \int_x^\infty \tilde{K}(x,y;\lambda) dy
$$
and
$$
\hat{K}(x,y;\lambda) := \tilde{K}(x,y;\lambda) e^{2x} e^{-2y} = \hat{a}(x,y;\lambda) + \hat{b}(x,y;\lambda) e^{-2i\omega(y-x)},
$$
with $\hat{a}(x,y;\lambda) = e^{2x} e^{-4y} \mathcal{O}(y^0 \lla \omega \rra^{-1})$ and $\hat{b}(x,y ; \lambda) = e^{x} e^{-3y} \mathcal{O}(x^0y^0 \lla \omega \rra^{-1})$.

Now we turn to \eqref{eqn:volterratrans}. We would like to show that a solution $H$ to \eqref{eqn:volterratrans} is of symbol type, then conclude that a solution to \eqref{eqn:ODEh1} is also of symbol type. The argument is broken down into a few steps.

\textbf{Step 1: } Firstly, from \eqref{eqn:h1bound} we have $|H(x;\lambda)| \lesssim \lla \omega \rra^{-1}$. This also follows directly from Volterra iterations of \eqref{eqn:volterratrans}, since
$$
|g(x;\lambda)| \lesssim  e^{2x} \int_x^\infty e^{-2y} + e^{-x} e^{-y} dy \lla \omega \rra^{-1} \lesssim \lla \omega \rra^{-1} .
$$
Hence $\| g(\cdot;\lambda)\|_{L^\infty} \lesssim \lla \omega \rra^{-1} $. Moreover, the Kernel $\hat{K}$ satisfies
$$
|\hat{K}(x,y;\lambda) | \lesssim  \lf( e^{2x}e^{-4y} + e^x e^{-3y} \rt) \lla \omega\rra^{-1},
$$
so
$$
\int_{\delta_1\lla\omega\rra^{-1}}^\infty  \sup_{x\in(\delta_1\lla\omega\rra^{-1},y)}|\hat{K}(x,y;\lambda) | dy \lesssim  \int_{\delta_1\lla\omega\rra^{-1}}^\infty  \lf( e^{2x}e^{-4y} + e^x e^{-3y} \rt) dy \lla \omega\rra^{-1} \lesssim  \lla \omega\rra^{-1}.
$$

\textbf{Step 2:} We use induction to show that $|\PD^j_x H(x;\lambda)| \lesssim x^{-j} \lla \omega \rra^{-1}$ for all $j\in \mathbb{N}_0$. This is true for $j=0$ from step 1. Now suppose that $|\PD^l_x H(x;\lambda)| \lesssim x^{-l} \lla \omega \rra^{-1}$ for all $l=0,\cdots, j-1$. Taking derivatives of \eqref{eqn:volterratrans} we have
\begin{align*}
\PD^j_x H(x;\lambda) 
= & g_{j,0} (x;\lambda) + \int_x^\infty \hat{K}(x,y;\lambda) \PD^j_y h(y;\lambda) dy
\end{align*}
where
$$
g_{j,0} (x;\lambda) := \PD^j_x g(x;\lambda) + \sum_{l=0}^{j-1} \begin{pmatrix}  j  \\ l \end{pmatrix} \int_0^\infty \PD^{j-l}_x \hat{K}(x,y+x;\lambda) \PD^l _x H(y+x;\lambda) dy.
$$
We compute
\begin{align*}
\PD^j_x g(x;\lambda) = & \PD^j_x \lf( e^{2x} \int_0^\infty \tilde{K}(x,y+x;\lambda) dy \rt) \\
= & \int_0^\infty e^{-2y} \mathcal{O}((y+x)^{-j} \lla \omega \rra^{-1}) + e^{-y}\mathcal{O}(x^{-j}(y+x)^0 \lla \omega \rra^{-1}) e^{-2i\omega y} dy.
\end{align*}
Hence $|\PD^j_x g(x;\lambda)| \lesssim x^{-j} \lla \omega \rra^{-1}$. Moreover,
\begin{align*}
& | \PD^{j-l}_x \hat{K}(x,y+x;\lambda) | \\
= & \lf| e^{-4y} \PD_x^{j-l} \lf(e^{-2x} \mathcal{O}((y+x)^0 \lla \omega \rra^{-1}) \rt) \rt| + \lf| e^{-3y}  \PD_x^{j-l} \lf(e^{-2x}\mathcal{O}(x^0 (y+x)^0 \lla \omega \rra^{-1}) \rt) \rt| \\
\lesssim & x^{-(j-l)} \lla x \rra^{j-l} \lf( e^{-4y} e^{-2x} +  e^{-3y} e^{-2x} \rt) \lla \omega \rra^{-1}.
\end{align*}
We also have
$$
| \PD^l_x H(y+x;\lambda) | \lesssim (y+x)^{-l} \lla \omega \rra^{-1} \lesssim x^{-l} \lla \omega \rra^{-1} 
$$
for $l=0,\cdots,j-1$ from the inductive hypothesis. Hence
\begin{align*}
\lf|\sum_{l=0}^{j-1} \begin{pmatrix}  j  \\ l \end{pmatrix}  \int_0^\infty \PD^{j-l}_x \hat{K}(x,y+x;\lambda) \PD^l _x h(y+x;\lambda) dy \rt| \lesssim  e^{-2x} x^{-j} \lla x \rra^j \lla \omega \rra^{-1} \lesssim  x^{-j} \lla \omega \rra^{-1}.
\end{align*}
Therefore $|g_{j,0}(x;\lambda)| \lesssim x^{-j} \lla \omega \rra^{-1}$.

Now define $\tilde{g}_{j,0}(x;\lambda) := x^j g_{j,0}(x;\lambda)$ and $H_{j,0}(x;\lambda):=x^j \PD^j_x H (x;\lambda)$. Then $H_{j,0}$ solves the Volterra equation
$$
H_{j,0}(x;\lambda) = \tilde{g}_{j,0}(x;\lambda) + \int_x^\infty \hat{K}(x,y;\lambda) x^j y^{-j} H_{j,0}(y;\lambda) dy
$$
with $\| \tilde{g}_{j,0}(\cdot ;\lambda)\| _{L^\infty} \lesssim \lla \omega \rra^{-1}$ and $|\hat{K}(x,y;\lambda) x^j y^{-j}| \leq |\hat{K}(x,y;\lambda)|$. Hence Volterra iterations give the bound $\|H_{j,0}(\cdot;\lambda)\|_{L^\infty} \lesssim \lla \omega \rra^{-1}$, and therefore $|\PD^j_x H (x;\lambda)| \lesssim x^{-j} \lla \omega \rra^{-1}$.

\textbf{Step 3: } We show that $|\PD^j_x \PD^k_\omega H(x;\lambda)| \lesssim x^{-j}\lla \omega \rra^{-1-k}$ for all $j\in \mathbb{N}_0$ and $k\in \mathbb{N}_0$, by induction on $k$. The case when $k=0$ is done in step 2. Now suppose that we have $|\PD^j_x \PD^m_\omega H(x;\lambda)| \lesssim x^{-j}\lla \omega \rra^{-1-m}$ for all $j\in \mathbb{N}_0$ and $m = 0,1,\cdots k-1$.

Differentiating \eqref{eqn:volterratrans} directly, we have
\begin{align*}
\PD^j_x \PD^k_\omega H(x;\lambda) 
& = g_{j,k}(x;\lambda) + \int_x^\infty \hat{K}(x,y;\lambda) \PD^j_y \PD^k_\omega H(y;\lambda) dy
\end{align*}
where
\begin{align*}
 g_{j,k}(x;\lambda)  := & \PD^j_x \PD^k_\omega g(x;\lambda) \\
 & + \sum_{\substack{0\leq l \leq j, 0\leq m\leq k, \\ (l,m)\neq (j,k)}} \begin{pmatrix} j \\ l \end{pmatrix} \begin{pmatrix} k \\ m \end{pmatrix} \int_0^\infty \PD^{j-l}_x \PD^{k-m}_\omega \hat{K}(x,y+x;\lambda) \PD^l_x \PD^m_\omega h(y+x;\lambda) dy.
\end{align*}
Using the forms of $g$ and $\hat{K}$ and the inductive hypothesis, we find that $|g_{j,k}(x;\lambda)| \lesssim x^{-j} \lla \omega \rra^{-1}$. Then following the same argument as in step 2, Volterra iterations give the bound $|\PD^j_x \PD^k_\omega H(x;\lambda)| \lesssim x^{-j} \lla \omega \rra^{-1}$.

So now we can assume that $|\omega|\geq 1$ and use a scaling argument to obtain the decay in $\omega$. Observe that both $g$ and $\hat{K}$ have a symbol part and an oscillatory part. Since the symbol part already gives enough decay, we only need to be concerned with the oscillatory part. In $g(x;\lambda)$, we write
\begin{align*}
e^{2x} \int_x^\infty  b(x,y;\lambda) e^{-2i\omega(y-x)} dy = & e^x \int_x^\infty e^{-y} \mathcal{O}(x^0 y^0 \lla \omega \rra^{-1}) e^{-2i\omega(y-x)} dy \\
= & \int_0^\infty e^{- \frac{y}{\omega}} \mathcal{O} \lf(x^0\lf( \frac{y}{\omega} + x \rt)^0  \omega ^{-2} \rt) e^{-2iy} dy.
\end{align*}
Then no derivative will fall on the oscillatory term. Using Lemmas \ref{lemma:diffsymb1} and \ref{lemma:diffsymb2}, we have 
\begin{align*}
& \PD^j_x \PD^k_\omega \int_0^\infty e^{- \frac{y}{\omega}} \mathcal{O} \lf(x^0\lf( \frac{y}{\omega} + x \rt)^0  \omega ^{-2} \rt) e^{-2iy} dy \\
= & \sum_{m=0}^k \int_0^\infty e^{-y} \mathcal{O} \lf(x^{-j} \lf( y+ x \rt)^0 y^m \omega ^{-1-k} \rt) e^{-2i\omega y} dy .
\end{align*}
Taking the absolute value we have
\begin{align*}
\lf| \sum_{m=0}^k \int_0^\infty e^{-y} \mathcal{O} \lf(x^{-j} \lf( y+ x \rt)^0 y^m \omega ^{-1-k} \rt) e^{-2i\omega y} dy  \rt| & \lesssim \sum_{m=0}^k x^{-j} \omega^{-1-k} \int_0^\infty e^{-y} y^m dy \\
& \lesssim x^{-j} \omega^{-1-k}.
\end{align*}
Hence we indeed have $| \PD^j_x \PD^k_\omega g(x;\lambda) | \lesssim x^{-j} \omega ^{-1-k}$. Similarly, the $\hat{a}$ part of $\hat{K}$ already gives enough decay, and we write
$$
\int_x^\infty \hat{b}(x,y;\lambda) e^{-2i\omega(y-x)} H(y;\lambda) dy = \int_0^\infty \hat{b}\lf( x,\frac{y}{\omega}+x;\lambda \rt) e^{-2i y } H \lf( \frac{y}{\omega}+x;\lambda \rt) \frac{dy}{\omega} .
$$
So we can have
\begin{align*}
 &\PD^j_x \PD^k_\omega \int_0^\infty \hat{b}\lf( x,\frac{y}{\omega}+x;\lambda \rt) e^{-2i y } H\lf( \frac{y}{\omega}+x;\lambda \rt) \frac{dy}{\omega} \\
& =  b_{j,k}(x;\lambda) + \int_x^\infty \hat{b}(x,y;\lambda) \PD^j_x \PD^k_m H(y;\lambda) e^{-2i\omega(y-x)} dy,
\end{align*}
where
\begin{align*}
b_{j,k}(x;\lambda) : = & \sum_{\substack{0\leq l \leq j, 0\leq m\leq k, \\ (l,m)\neq (j,k)}} \begin{pmatrix} j \\ l \end{pmatrix} \begin{pmatrix} k \\ m \end{pmatrix} \int_0^\infty e^{-2i y } \PD^{j-l}_x \PD^{k-m}_\omega \frac{\hat{b}\lf( x,\frac{y}{\omega}+x;\lambda \rt)}{\omega} \PD^l_x \PD^m_\omega  H\lf( \frac{y}{\omega}+x;\lambda \rt) dy  \\
& + \int_0^\infty e^{-2i y } \frac{\hat{b}\lf( x,\frac{y}{\omega}+x;\lambda \rt)}{\omega} \lf( \PD_x^j \sum_{m=0}^{k-1} \begin{pmatrix}k\\m\end{pmatrix} \PD^{k-m}_1 \PD^m_2  H\lf( \frac{y}{\omega}+x;\lambda \rt) \rt) dy.
\end{align*}
By inductive hypothesis, $H$ behaves like a symbol when the second slot is differentiated up to $k-1$ times. So we have
$$
\lf| \PD^l_x \PD^m_\omega  H\lf( \frac{y}{\omega}+x;\lambda \rt)\rt| \lesssim  \lf( \frac{y}{\omega}+x \rt)^{-l} \omega^{-1-m} \lesssim x^{-l} \omega^{-1-m},
$$
for all $l \in \mathbb{N}_0$ and $m=1,\cdots, k-1$. Moreover,
$$
\lf|\PD^j_x  \sum_{m=0}^{k-1} \begin{pmatrix}k\\m\end{pmatrix} \PD^{k-m}_1 \PD^m_2  H\lf( \frac{y}{\omega}+x;\lambda \rt) \rt| \lesssim  \lf( \frac{y}{\omega}+x \rt)^{-j} \omega^{-1-k} \lesssim x^{-j} \omega^{-1-k}.
$$
Similarly,
\begin{align*}
\PD^{j-l}_x \PD^{k-m}_\omega \frac{\hat{b}\lf( x,\frac{y}{\omega}+x;\lambda \rt)}{\omega} = & \PD^{j-l}_x \PD^{k-m}_\omega \lf( e^{x}e^{-3\lf(\frac{y}{\omega}-x\rt)}\mathcal{O}\lf(x^0 \lf(\frac{y}{\omega}-x\rt)^0 \omega^{-2} \rt) \rt) \\
= & \PD^{j-l}_x \lf( e^{-2x} \sum_{m'=0}^{k-m} e^{-3\frac{y}{\omega}}\mathcal{O}\lf(x^0 \lf(\frac{y}{\omega}-x\rt)^0 y^{m'} \omega^{-2-(k-m)-m'} \rt)\rt),
\end{align*}
and
\begin{align*}
& \lf| \PD^{j-l}_x \lf( e^{-2x} \sum_{m'=0}^{k-m} \mathcal{O}\lf(x^0 \lf(\frac{y}{\omega}-x\rt)^0 y^{m'} \omega^{-(k-m)-m'} \rt)\rt) \rt| \\
& \lesssim  e^{-2x} x^{-(j-l)} \lla x\rra^{j-l}\sum_{m'=0}^{k-m} y^{m'} \omega^{-2-(k-m)-m'} .
\end{align*}
Now we can estimate
\begin{align*}
|b_{j,k}(x;\lambda)|  \lesssim & e^{-2x} x^{-j} \lla x \rra^j \lf( \sum_{m=0}^{k}\sum_{m'=0}^{k-m} \int_0^\infty e^{-3\frac{y}{\omega}} y^{m'}  dy \omega^{-3-k-m'} + e^{-2x} \int_0^\infty e^{-3\frac{y}{\omega}} dy \omega^{-3-k} \rt) \\
\lesssim &  x^{-j}\lf( \sum_{m'=0}^k \int_0^\infty e^{-3y} y^{m'} dy \omega^{-2-k} + \int_0^\infty e^{-3y} dy \omega^{-2-k}  \rt) \\
\lesssim & x^{-j} \omega^{-2-k}.
\end{align*}
Hence we conclude that
$$
\lf| g_{j,k}(x;\lambda) \rt | \lesssim x^{-j} \lla \omega \rra ^{-1-k},
$$
and solving the Volterra equation we obtain $|\PD^j_x \PD^k_\omega H(x;\lambda)| \lesssim x^{-j} \lla \omega \rra^{-1-k}$.

\textbf{Transforming back to $h_1$:} we show that $| \PD^j_\rho \PD^k_\omega h_1(\rho;\lambda)| \lesssim \rho^{-j} (1-\rho)^{1-j} \lla \omega \rra^{-1-k}$ for $j+k \geq 1$.

So far we have shown that $\tilde{H}(x;\lambda) = 1+ e^{-2x} \mathcal{O}(x^0 \lla \omega \rra^{-1})$. Hence for any $j+k \geq 1$, we have
$$
\lf| \PD^j_x \PD^k_\omega \tilde{H}(x;\lambda) \rt| \lesssim e^{-2x} x^{-j} \lla x \rra^j \lla \omega \rra^{-1-k}.
$$
Recall that $\varphi (\rho) = \frac{1}{2} \log \frac{1+\rho}{1-\rho} $, so we have
$$
\lf| \PD^j_\rho \varphi (\rho)  \rt| \lesssim (1-\rho)^{-j}
$$
for $\rho\in(0,1)$ and $j\in \mathbb{N}$. Now
$$
\PD^j_\rho  h_1(\rho;\lambda) = \PD^j_\rho \tilde{H}(\varphi(\rho);\lambda) = \sum_{l=1}^j \tilde{H}^{(j,0)}(\varphi(\rho);\lambda) \tilde{\varphi}_{j,l}(\rho)
$$
where $|\tilde{\varphi}_{j,l}(\rho)| \lesssim (1-\rho)^{-j}$. So
\begin{align*}
|\PD^j_\rho \PD^k_\omega  h_1(\rho;\lambda)| \lesssim  &\sum_{l=1}^j \lf| \tilde{H}^{(j,k)}(\varphi(\rho);\lambda) \rt| (1-\rho)^{-j} \\
\lesssim & \varphi(\rho)^{-j} \lla \varphi(\rho) \rra^j e^{-2\varphi(\rho)} \lla \omega \rra^{-1-k} (1-\rho)^{-j} \\
\lesssim & \rho^{-j} (1-\rho)^{1-j} \lla \omega \rra^{-1-k}
\end{align*}
as desired.

\bibliography{phd-bib}{}
\bibliographystyle{abbrv}

\end{document}